\documentclass[letter,11pt]{article}

\usepackage[linktocpage=true]{hyperref}
\usepackage[margin=1.2 in, top=1 in, bottom= 1.2 in]{geometry}

\usepackage[normalem]{ulem}

\makeatletter
\g@addto@macro\normalsize{
  \setlength\abovedisplayskip{8pt}
  \setlength\belowdisplayskip{8pt}
  \setlength\abovedisplayshortskip{8pt}
  \setlength\belowdisplayshortskip{8pt}
  }
\makeatother

\usepackage{tocloft}

\usepackage[english]{babel}
\usepackage{amsfonts}
\usepackage{mathrsfs}
\usepackage{bbm}
\usepackage{latexsym}
\usepackage{math dots}
\usepackage{amssymb}
\usepackage{mathtools}

\usepackage{enumitem}
\setlist{nolistsep}
\usepackage{amsthm}
\usepackage[capitalize]{cleveref}
\crefformat{footnote}{#2\footnotemark[#1]#3} 

\newcommand\eqnitem[1][]{%
  \ifx\relax#1\relax  \item \else \item[#1] \fi
  \abovedisplayskip=0pt\abovedisplayshortskip=0pt~\vspace*{-\baselineskip}}

\newtheoremstyle{plain}{3mm}{3mm}{\slshape}{}{\bfseries}{.}{.5em}{}
\newtheoremstyle{definition}{2mm}{2mm}{}{}{\bfseries}{.}{.5em}{}
\theoremstyle{plain}
	
\newtheorem{theorem}{Theorem}

\newtheorem{lemma}[theorem]{Lemma}
\newtheorem{proposition}[theorem]{Proposition}
\newtheorem{corollary}[theorem]{Corollary}

\newtheorem{question}[theorem]{Question}
\theoremstyle{definition}
\newtheorem{definition}[theorem]{Definition}
\newtheorem{remark}[theorem]{Remark}

\newtheorem{examples}[theorem]{Examples}

\theoremstyle{plain}
\newcounter{MainTheoremCounter}

\newtheorem{Maintheorem}[MainTheoremCounter]{Theorem}

\theoremstyle{plain}
\newtheorem*{namedthm}{\namedthmname}
\newcounter{namedthm}
\makeatletter
	\newenvironment{named}[2]
	{\def\namedthmname{#1}
	\refstepcounter{namedthm}
	\namedthm[#2]\def\@currentlabel{#1}}
	{\endnamedthm}
\makeatother

\usepackage{chngcntr}
\counterwithin{theorem}{section}

\numberwithin{equation}{section}

\allowdisplaybreaks

\usepackage{xcolor}
\definecolor{Color2}{rgb}{0.78, 0.11, 0.0}
\hypersetup{citecolor = black,colorlinks,
			linkcolor = black,
			urlcolor = Color2}

\usepackage{titlesec}
\titleformat{\section}
  {\Large\center\bfseries}
  {\thesection.}{.7em}{}
\titlespacing*{\section}{0pt}{3.5ex plus 0ex minus 0ex}{1.5ex plus 0ex}
\titleformat{\subsection}
  {\large\center\bfseries}
  {\thesubsection.}{.7em}{}
\titlespacing*{\subsection}{0pt}{3.5ex plus 0ex minus 0ex}{1.5ex plus 0ex}
\titleformat{\subsubsection}
  {\center\bfseries}
  {\thesubsubsection.}{.7em}{}
\titlespacing*{\subsubsection}{0pt}{3.5ex plus 0ex minus 0ex}{1.5ex plus 0ex}

\addto\captionsenglish{}

\usepackage{titling}
\newcommand{\Erdos}{Erd\H{o}s}

\newcommand{\supp}{{\normalfont\text{supp}}\,}

\newcommand{\eps}{\epsilon}
\newcommand{\N}{\mathbb{N}}
\newcommand{\Z}{\mathbb{Z}}
\newcommand{\R}{\mathbb{R}}

\newcommand{\Q}{\mathbb{Q}}

\newcommand{\Nz}{\N_0}
\newcommand{\defeq}{\vcentcolon=}

\newcommand\restr[2]{{ \left.\kern-\nulldelimiterspace #1 \right|_{#2}}}

\newcommand{\define}[1]{{\itshape #1}}

\renewcommand{\epsilon}{\varepsilon}
\renewcommand{\leq}{\leqslant}
\renewcommand{\geq}{\geqslant}
\renewcommand{\setminus}{\backslash}

\renewcommand{\P}{\mathbb{P}}

\definecolor{ggreen}{RGB}{53, 212, 62}

\usepackage[normalem]{ulem}

\newcommand{\al}{\alpha}

\newcommand{\ogamma}{\overline{\gamma}}

\newcommand{\NC}{\mathcal{N}}

\newcommand{\HC}{\mathcal{H}}

\newcommand{\DC}{\mathcal{D}}

\newcommand{\CC}{\mathcal{C}}

\newcommand{\dimh}{\dim_{\text{H}}}
\newcommand{\dimdh}{\dim_{\text{H}}}
\newcommand{\udimdh}{\overline{\dim}_{\text{H}}\hspace*{.1em}}
\newcommand{\ldimdh}{\underline{\dim}_{\text{H}}\hspace*{.1em}}

\newcommand{\dimdm}{\dim_{\text{M}}}
\newcommand{\udimdm}{\overline{\dim}_{\text{M}}\hspace*{.1em}}
\newcommand{\ldimdm}{\underline{\dim}_{\text{M}}\hspace*{.1em}}
\DeclareMathOperator{\diam}{diam}

\newcommand{\dimm}{\dim_{\text{M}}}
\newcommand{\udimm}{\overline{\dim}_{\text{M}}\hspace*{.1em}}
\newcommand{\ldimm}{\underline{\dim}_{\text{M}}\hspace*{.1em}}

\newcommand{\cmeas}{\mu}
\newcommand{\dyad}{\mathcal{P}}

\newcommand{\conv}{\mathop{\scalebox{1.5}{\raisebox{-0.2ex}{$\ast$}}}}
\newcommand{\dd}{\ \mathrm{d}}

\renewcommand{\Phi}{\mathfrak{R}}
\renewcommand{\Psi}{\mathfrak{L}}

\begin{document}

\title{\bfseries Additive and geometric transversality of fractal sets in the integers}
\author{Daniel Glasscock \and Joel Moreira \and Florian K.\ Richter}

\date{\small \today}
\maketitle

\begin{abstract}
By juxtaposing ideas from fractal geometry and dynamical systems, Furstenberg proposed a series of conjectures in the late 1960's that explore the relationship between digit expansions with respect to multiplicatively independent bases.
In this work, we introduce and study -- in the discrete context of the integers -- analogues of some of the notions and results surrounding Furstenberg's work.
In particular, we define a new class of fractal sets of integers that parallels the notion of $\times r$-invariant sets on the 1-torus and investigate the additive and geometric independence between two such fractal sets when they are structured with respect to multiplicatively independent bases.
Our main results in this direction parallel the works of Furstenberg, Hochman-Shmerkin, Shmerkin, Wu, and Lindenstrauss-Meiri-Peres and include:
\begin{itemize}
\item
a classification of all subsets of the positive integers that are simultaneously $\times r$- and $\times s$-invariant;
\item
integer analogues of two of Furstenberg's transversality conjectures pertaining to the dimensions of the intersection $A\cap B$ and the sumset $A+B$ of $\times r$- and $\times s$-invariant sets $A$ and $B$ when $r$ and $s$ are multiplicatively independent; and
\item
a description of the dimension of iterated sumsets $A+A+\cdots+A$ for any $\times r$-invariant set $A$.
\end{itemize}
We achieve these results by combining ideas from fractal geometry and ergodic theory to build a bridge between the continuous and discrete regimes.  
For the transversality results, we rely heavily on quantitative bounds on the $L^q$-dimensions of projections of restricted digit Cantor measures obtained recently by Shmerkin.
We end by outlining a number of open questions and directions regarding fractal subsets of the integers.
\end{abstract}

\small
\tableofcontents
\thispagestyle{empty}
\normalsize


\section{Introduction}
\label{sec_intro_top}

Number theorists in the first half of the 20th century were among the first to consider the degree to which base $2$ and base $3$ representations of real numbers are independent. 
An open conjecture attributed to Mahler \cite{mendesfrance1980} postulates, for example, that if $(a_n)_{n=1}^\infty \subseteq \{0,1\}$ is not eventually periodic, then at least one of the numbers $\sum_{n = 1}^\infty a_n 2^{-n}$ and $\sum_{n = 1}^\infty a_n 3^{-n}$ is transcendental. 
In a different vein, Cassels \cite{Cassels1959} and Schmidt \cite{schmidt_1960}, answering a question of Steinhaus about Cantor's middle thirds set $C$, proved that almost every number in $C/2$ (with respect to the $\log2/\log3$-dimensional Hausdorff measure) is normal to every base which is not a power of~$3$.
More general questions along these lines -- is almost every real number with respect to any continuous $\times 3$-invariant measure on $[0,1]$ normal to every base that is not a power of $3$ -- remain open, despite considerable partial progress \cite{host_1995,lindenstrauss_2001,hochman_shmerkin_2015}.

Studying the independence between different representations of real numbers remains an active area of research that brings together results and techniques from number theory, ergodic theory, and geometric measure theory.
Parallel investigations concerning representations of integers appear to be less developed but are no less natural or interesting.
It is the purpose of this paper to advance those investigations by demonstrating various forms of independence between different base representations in the non-negative integers.
One of the basic principles that underpin our results in this direction states the following:
\begin{quote}
\emph{If $r$ and $s$ are multiplicatively independent positive integers (meaning that the quantity $\log(r)/\log(s)$ is irrational) and $A$ and $B$ are subsets of the non-negative integers that are structured with respect to base-$r$ and base-$s$ representations, respectively, then $A$ and $B$ lie in general position. 
}\end{quote}
The following unresolved conjecture of Erd\H{o}s \cite{erdosmathmag} exemplifies this heuristic: \emph{for all $n \geq 9$, it is impossible the express the number $2^n$ in base $3$ using only the digits $0$ and $1$}; see \cite{dupuy_weirich_2016,lagariasternarypowersoftwo} for some recent progress.
Today, \Erdos{}' conjecture is understood as merely a special case of a much broader conjecture that asserts that any infinite set of natural numbers that has a ``simple'' description in base $r$ must have a ``complex'' description in base $s$ (see \cref{question_finiteintersection} in \cref{sec_small_intersections} for more details).
A related folklore conjecture in number theory \cite{oeisA146025} posits that $\{0, 1, 82000\}$ is exactly the set of non-negative integers that can be written in bases 2, 3, 4, and 5 using only the digits 0 and 1. A partial answer to this was given recently by Burrell and Yu \cite{burrell_yu_digit}, who proved that the set $A$ of non-negative integers that can be written in bases 4 and 5 using only the digits 0 and 1 satisfies $\big| A \cap [0,N] \big| \leq C_\eps N^{\eps}$ for all $\eps > 0$.

In this paper, we aim to 1) introduce a family of multiplicatively structured ``fractal'' subsets of the non-negative integers that naturally arise from digit restrictions, and 2) investigate the \emph{transversality}, or independence, between members of that family that are structured with respect to multiplicatively independent bases.  Our investigation is strongly motivated by the heuristic and conjectures mentioned above and by the recent resolutions of a pair of Furstenberg's conjectures concerning notions of geometric and additive transversality of fractal subsets of the reals.
Our results give integer parallels of those advancements in the reals, generalize the aforementioned result of Burrell and Yu, and make progress toward Erd\H{o}s' conjecture.

Before recounting the relevant history and stating our main results in full generality, we focus our attention on the special case of restricted digit Cantor sets in the non-negative integers.
Although restricted digit Cantor sets comprise only a small subclass of the sets that we consider, most of our results are already novel and interesting for this class. In this sense, the following section serves as a preview of our main results.

\subsection{Preview of the main results}

Let $\N = \{1, 2,3, \ldots\}$ and $\Nz = \{0, 1, 2, \ldots\}$. 
An \define{integer base-$r$ restricted digit Cantor set} is a set of non-negative integers whose base-$r$ expansion includes only digits from a fixed set $\mathcal{D}\subseteq \{0,1,\ldots, r-1\}$, i.e.,
\begin{align}\label{eqn_cantor_digit}
\left\{ \sum_{i = 0}^n a_i r^i \ \middle | \ n \in \Nz, \ a_0, \ldots, a_n \in \mathcal{D}\right\}.
\end{align}
The \emph{(mass) dimension} of such a set $A$ is $\dim A \defeq \log |\mathcal{D}| / \log r$, in the sense that $|A \cap [0,N)| = N^{\dim A + o(1)}$.
We discuss notions of dimension for more general subsets of the non-negative integers in the next section and define them precisely in \eqref{eqn_upper_mass_def} and \cref{def_mass_and_h_dimensions}.
While a number of arithmetic properties of integer restricted digit Cantor sets are well studied -- divisibility \cite{banksshparlinski}, distribution in arithmetic progressions \cite{erdos_mauduit_sarkozy_restricted_digit_cantor_sets, konyagin}, number of prime factors \cite{KMS}, character sums \cite{BCS} -- much less appears to be known about the relationship between such sets when they are structured with respect to different bases.

Let $r$ and $s$ be \emph{multiplicatively independent} positive integers, and let $A, B \subseteq \Nz$ be base-$r$ and base-$s$ restricted digit Cantor sets, respectively.
Under these assumptions, our results demonstrate that the sets $A$ and $B$ are transverse both in a geometric / probabilistic sense and in an additive combinatorial sense.  More precisely, the sets $A$ and $B$ are
\begin{itemize}
    \item geometrically / probabilistically in general position, in the sense that neither $A$ nor $B$ contains the other and in the sense that the size of $A \cap B$ is at most what is expected if $A$ and $B$ were independent random sets;
    \item additive combinatorially disjoint, in the sense that the cardinality of the sumset $A+B$ is nearly as large as possible, and hence there are only very few coincidences amongst the sums $a+b$ for $a\in A$ and $b\in B$.
\end{itemize}
Our main Theorems \ref{prop_discrete_furstenberg} and \ref{mainthm_integer_intersections} address the first point, while \cref{mainthm_integer_sumsets} addresses the second. We move now to formulate corollaries of those theorems that clearly demonstrate these notions of independence.

To describe all of the elements of a non-trivial base-$5$ restricted digit Cantor set in base 17, all 17 digits are required.  The following corollary of \cref{prop_discrete_furstenberg} generalizes this observation by showing that restricted digit Cantor structures with respect to multiplicatively independent bases are mutually incompatible. It also provides an integer analogue of a well-known theorem of Furstenberg; see \cref{thm_furstenberg} below.

\begin{named}{Corollary of \cref{prop_discrete_furstenberg}}{}
Under the assumptions on the sets $A$ and $B$ above, if $A \subseteq B$, then either $A = \{0\}$ or $B = \Nz$.
\end{named}

The finer question about the size of the intersection $A \cap B$ is addressed in \cref{mainthm_integer_intersections}.  For $N \in \N$, define $A_N = A \cap [0,N)$ and $B_N = B \cap [0,N)$.  The sets $A_N$ and $B_N$ would be probabilistically independent if $\big| A_N \cap B_N \big| / N = \big| A_N \big| \big| B_N \big| / N^2$.
Examples show that the sets $A$ and $B$ can be disjoint, even in the case that both $A$ and $B$ have a large set of allowed digits, so the inequality
\begin{align}
\label{eqn_finitistic_intersection_transversality}
    \frac{\big| A_N \cap B_N \big|}{N} \ \ll \ \frac{\big| A_N \big|}{N} \, \cdot \, \frac{\big| B_N \big|}{N}
\end{align}
for all $N$ large can be understood to demonstrate a type of asymptotic probabilistic transversality between the sets $A$ and $B$.  (As explained in the next section, such an inequality can also be interpreted as $A_N$ and $B_N$ being geometrically in general position.) \cref{mainthm_integer_intersections} shows that \eqref{eqn_finitistic_intersection_transversality} holds up to a factor of $N^\eps$; the precise extent to which \eqref{eqn_finitistic_intersection_transversality} holds remains open and is addressed briefly in \cref{sec_small_intersections}.

\begin{named}{Corollary of \cref{mainthm_integer_intersections}}{}
Under the assumptions on the sets $A$ and $B$ above, for all $\eps > 0$ and all sufficiently large $N$,
\begin{itemize}
    \item if $\dim A + \dim B \geq 1$, then
\[\frac{\big| A_N \cap B_N \big|}{N} \ \leq \ N^\eps \, \cdot \, \frac{\big| A_N \big|}{N} \, \cdot \, \frac{\big| B_N \big|}{N};\]
    \item if $\dim A + \dim B < 1$, then
\[\big| A_N \cap B_N \big|\ \leq \ N^\eps.\]
\end{itemize}
\end{named}

As an example application, let $C_{4,{\{0,1\}}}$ and $C_{5,{\{0,1\}}}$ be the sets of non-negative integers that have only digits $0$ and $1$ in their base $4$ and $5$ expansions, respectively. Since $\log 2 / \log 4 + \log 2 / \log 5 < 1$, it follows that $\big| C_{4,{\{0,1\}}} \cap C_{5,{\{0,1\}}} \big| = o(N^\eps)$, which recovers the theorem of Burrell and Yu's mentioned in the previous section.

If $X$ and $Y$ are finite sets of real numbers, then it is easy to check that
\begin{align*}
    |X| + |Y| - 1\leq |X+Y| \leq |X||Y|.
\end{align*}
Equality holds on the left if and only if $X$ and $Y$ are arithmetic progressions of the same step size.  When $|X+Y|$ is near this lower bound, inverse theorems in combinatorial number theory (e.g. \cite[Ch. 5]{tao_vu_book_2006}) provide additive structural information on the sets $X$ and $Y$. At the other end of the spectrum, equality holds on the right if and only if none of the sums $x+y$, with $x\in X$ and $y\in Y$, coincide. In this case, the sets $X$ and $Y$ lie in general position from an additive combinatorial point of view.

In this context, the inequality
\begin{align}
\label{eqn_sumset_transversality}
    \big| A_N + B_N \big| \ \gg \ \min \big( N, \, \big| A_N \big| \, \cdot \, \big| B_N \big| \big)
\end{align}
can be understood as demonstrating additive combinatorial transversality between the sets $A_N$ and $B_N$.  \cref{mainthm_integer_sumsets} shows that  \eqref{eqn_sumset_transversality} holds up to a factor of $N^\eps$; the extent to which \eqref{eqn_sumset_transversality} holds is unknown and is discussed briefly in \cref{sec_pos_den_for_sumsets}.

\begin{named}{Corollary of \cref{mainthm_integer_sumsets}}{}
Under the assumptions on the sets $A$ and $B$ above, for all $\eps > 0$ and all sufficiently large $N$,
\begin{align*}
    \big| A_N + B_N \big| \ \geq \ \min \big( N, \, \big| A_N \big| \, \cdot \, \big| B_N \big| \big) \big / N^\eps.
\end{align*}
\end{named}

Theorems \ref{prop_discrete_furstenberg}, \ref{mainthm_integer_intersections}, and \ref{mainthm_integer_sumsets} are more general than the corollaries above might suggest. 
Indeed, each result applies not only to restricted digit Cantor sets, but to a wider class of integer fractal sets called \emph{multiplicatively invariant} sets.
Moreover, each set can be replaced by a rounded image of itself under any affine transformation of $\R$.
Finally, in \cref{mainthm_integer_sumsets}, the sets $A$ and $B$ can be replaced by arbitrary subsets of $A$ and $B$, and set cardinality can be replaced with a notion of discrete Hausdorff content. 
We will introduce multiplicatively invariant sets in \cref{sec_intro_integers} and state our main results precisely there, after providing some historical context and motivation for them in the next section.

\subsection{History and context}

In the language of fractal geometry and dynamical systems, Furstenberg \cite{furstenbergdisjointness,furstenbergtransversality} established a number of conjectures and results that explore the relationship between multiplicative structures with respect to different bases in the real numbers. The notion of structure particularly relevant to this work is that of multiplicative invariance: a set $X\subseteq [0,1]$ is \define{$\times r$-invariant} if it is closed and $T_r X\subseteq X$, where $T_r\colon [0,1]\to[0,1]$ denotes the map
\[ T_r\colon x\mapsto rx\bmod1 . \]
We call a set $X\subseteq [0,1]$ \define{multiplicatively invariant} if it is $\times r$-invariant for some $r\geq 2$.

One of Furstenberg's first and most well-known results concerning multiplicatively invariant sets is the following theorem, the measure-theoretic analogue of which is the $\times 2$, $\times 3$ conjecture, a central open problem in ergodic theory.

\begin{theorem}[{\cite[Theorem 4.2]{furstenbergdisjointness}}]
\label{thm_furstenberg}
If $X\subseteq[0,1]$ is simultaneously $\times 2$- and $\times 3$-invariant, then either $X$ is finite or $X=[0,1]$.
\end{theorem}

The numbers $2$ and $3$ in \cref{thm_furstenberg} can be replaced by any pair of multiplicatively independent positive integers $r$ and $s$.
Following \cref{thm_furstenberg}, Furstenberg conjectured that if $X,Y\subseteq[0,1]$ are $\times r$- and $\times s$-invariant, respectively, then $X$ and $Y$ are \emph{transverse} in more than one sense, some of which are made precise below. While some of Furstenberg's ``transversality conjectures'' remain open, two of them were resolved recently by Hochman and Shmerkin \cite{localentropy}, Shmerkin \cite{shmerkin}, and, independently, Wu \cite{wu}.  Both of these conjectures are particularly relevant to this work, so we will expound on them further now.

In Euclidean geometry, linear subspaces $U, V\subseteq \R^d$ are said to be \define{in general position} (or \define{transverse}) if
\begin{align*}
\dim(U \cap V)\,&=\,\max\big(0, \dim U+\dim V - d \big), \text{ and} \\
\dim(U+V)\,&=\,\min\big(\dim U+\dim V, \ d \big).
\end{align*}
By analogy, Furstenberg conjectured\footnote{The intersection conjecture \eqref{furstenberg_intersection_transversality} is one of several conjectures stated in  \cite{furstenbergtransversality}.  The sumset conjecture \eqref{furstenberg_sumset_transversality} does not, as far as we are aware, appear by Furstenberg in print, but it was known to have originated with him.} that if $r$ and $s$ are multiplicatively independent and $X$ and $Y$ are $\times r$- and $\times s$-invariant subsets of $[0,1]$, then
\begin{align}
    \label{furstenberg_intersection_transversality}
    \dimh(  X \cap Y ) \,&\leq\, \max{}\big(0, \dimh X + \dimh Y - 1 \big), \text{ and}\\
    \label{furstenberg_sumset_transversality}
    \dimh(  X + Y ) \,&=\, \min{}\big(\dimh X + \dimh Y, \ 1 \big),
\end{align}
where $\dimh$ denotes the Hausdorff dimension.

With no assumptions on the sets $X, Y \subseteq [0,1]$, it is not difficult to find examples for which neither \eqref{furstenberg_intersection_transversality} nor \eqref{furstenberg_sumset_transversality} hold. Nevertheless, it is a consequence of Marstrand's projection and slicing theorems\footnote{Marstrand's slicing and projection theorems originally concerns orthogonal projections of subsets of the plane and intersections with lines. Images of the Cartesian product $X \times Y$ under orthogonal projections are, up to affine transformations which preserve dimension, sumsets of the form $\lambda X +  \eta Y$, while intersections of $X \times Y$ with lines are affinely equivalent to sets of the form $\lambda X \cap (\eta Y + \sigma)$. Also note that for sufficiently regular sets $X$ and $Y$, $\dimh (X \times Y) = \dimh X + \dimh Y$; see, for example, \cite[Corollary 8.11]{mattila_geometry_of_sets_1995}.} that for all Borel sets $X$ and $Y$, the typical dilated sets $\lambda X$ and $\eta Y$ are transverse in the sense of \eqref{furstenberg_intersection_transversality} and \eqref{furstenberg_sumset_transversality}.

\begin{theorem}[{\cite[Theorems II and III]{Marstrand_1954}}]
Let $X$ and $Y$ be Borel subsets of $[0,1]$. For Lebesgue-a.e. $\lambda, \eta, \sigma \in \R$,
\begin{align}
    \label{marstrand_slicing_result_intro}
    \dimh\big( \lambda X \cap (\eta Y + \sigma) \big) &\leq \max{}\big(0, \ \dimh (X \times Y) - 1 \big), \text{ and}\\
    \label{marstrand_result_intro}
    \dimh\big( \lambda X +  \eta Y \big) &= \min{}\big(\dimh (X \times Y) , \ 1 \big).
\end{align}
\end{theorem}

In this context, Furstenberg's conjectures in \eqref{furstenberg_intersection_transversality} and \eqref{furstenberg_sumset_transversality} say that the multiplicative invariance of the sets $X$ and $Y$ can be leveraged to change the result in Marstrand's theorem from concerning the typical sets $\lambda X \cap ( \eta Y + \sigma)$ and $\lambda X + \eta Y$ to concerning the specific ones $X \cap Y$ and $X + Y$.  In fact, Furstenberg conjectured that for $\times r$- and $\times s$-invariant sets $X$ and $Y$, the inequality in \eqref{marstrand_slicing_result_intro} and equality in \eqref{marstrand_result_intro} hold for \emph{all} non-zero $\lambda$ and $\eta$ and all $\sigma$.  Hochman and Shmerkin resolved the sumset conjecture by proving a stronger result for multiplicatively invariant measures, and several years later Shmerkin \cite{shmerkin} and Wu \cite{wu} independently resolved the intersection conjecture.  (These works resolved both conjectures for classes of attractors of iterated function systems, too.) Several more recent works offer new proofs of \eqref{furstenberg_intersection_transversality} and \eqref{furstenberg_sumset_transversality}; see, for example, \cite{austin_proof_of_furstenberg_2020, yu_improvement_to_furstenberg_arxiv,jordan_rapaport_2021,GMR_new_proof_2020}.

\begin{theorem}[{\cite{shmerkin,wu} and \cite{localentropy}}]
\label{thm_HS_localentropy_and_intersections}
Let $r$ and $s$ be multiplicatively independent positive integers, and let $X, Y \subseteq [0,1]$ be $\times r$- and $\times s$-invariant sets, respectively. For all $\lambda, \eta \in \R \setminus \{0\}$ and all $\sigma \in \R$,
\begin{align}
\label{eqn_furstenberg_intersection}
     \udimm\big( \lambda X \cap  (\eta Y + \sigma) \big) &\leq \max{} \big(0, \ \dimh X + \dimh Y - 1 \big), \text{ and} \\
\label{eqn_in_thm_HS_localentropy}
     \dimh\big( \lambda X +  \eta Y \big) &= \min{}\big(\dimh X + \dimh Y, \ 1 \big),
\end{align}
where $\udimm$ denotes the upper Minkowski dimension.
\end{theorem}

The upper bound on the dimension of fibers in \eqref{eqn_furstenberg_intersection} suffices to give the lower bound on the dimension of sumsets necessary for \eqref{eqn_in_thm_HS_localentropy}, as was observed in \cite{Furstenberg_2008}; for elaboration on the connection between the two, see the discussion following Conjecture 1.2 in \cite{localentropy}.   Shmerkin's main result in \cite{shmerkin}, which concerns the decay of $L^q$ norms of certain self-similar measures of dynamical origin, proves \eqref{eqn_furstenberg_intersection} by controlling the Frostman exponent of images of regular measures under projections. We derive a number of our main theorems from Shmerkin's work, which we elaborate on further in \cref{sec_frost_exp_projection_theorem}.

In an effort to better understand the role that the multiplicative independence between the bases plays in the sumset theorem, it is natural to ask about the sum of sets that are all structured with respect to the same base $r$.  Taking $X \subseteq [0,1]$ to be those numbers that can be written in decimal with only the digits 0, 1, and 2, we see that the equality in \eqref{furstenberg_sumset_transversality} need not hold:
\[\frac{\log 5}{\log 10} = \dimh (X+X) < 2 \dimh X = \frac{2\log 3}{\log 10}.\]
Nevertheless, it is a consequence of the following theorem of Lindenstrauss, Meiri, and Peres that the dimension of the iterated sumset $X+ \cdots + X$ approaches 1 as the number of summands increases.

\begin{theorem}[{\cite[Corollary 1.2]{lindenstrauss_meiri_peres_1999}}]
\label{LMP_theorem}
Let $(X_i)_{i =1}^\infty$ be a sequence of $\times r$-invariant subsets of $[0,1]$.  If $\sum_{i=1}^\infty \dimdh X_i / \allowbreak | \log \dimdh X_i |$ diverges, then
\begin{align*}
    \lim_{n \to \infty} \dimh \big( X_1 + \cdots + X_n \big) = 1.
\end{align*}
\end{theorem}

This theorem demonstrates that the structure captured by multiplicative invariance sits transversely to the additive structure captured by additive closure: because the sumset $X_1 + \cdots + X_n$ fills out the entire space (with respect to the Hausdorff dimension), the sets $X_i$ are not contained in an additively closed set of dimension less than 1. Dimension growth of iterated sumsets under weaker regularity conditions was studied recently in \cite{fraser_howroyd_yu_2019}.\\

While there is a strong historical precedent for the study of $\times r$-invariant subsets of the unit interval, less seems to be known in the integer and $p$-adic settings, despite the fact that many of the same objects and questions can be naturally formulated there.

Furstenberg \cite{furstenbergtransversality}, assuming a positive answer to one of his yet-unresolved transversality conjectures in the reals, drew a connection between the real and integer regimes by showing that given any finite word from the alphabet $\{0, \ldots, 9\}$, the decimal expansion of the number $2^n$ contains that word provided that $n$ is sufficiently large. 
This (conditionally) solves an analogue of Erd\H{o}s' conjecture mentioned earlier.

The folklore conjecture mentioned in the second paragraph in \cref{sec_intro_top} is profitably understood in terms of intersections of restricted digit Cantor sets and, as such, evokes the real transversality conjecture of Furstenberg in \eqref{furstenberg_intersection_transversality}.  Burrell and Yu's \cite{burrell_yu_digit} results toward a resolution of this conjecture rely heavily on Yu's work in \cite{yu_improvement_to_furstenberg_arxiv} on improvements to Shmerkin and Wu's resolution of Furstenberg's intersection conjecture.
Drawing on results in \cite{yu_improvement_to_furstenberg_arxiv}, Yu \cite{2020arXiv200405926Y} also shows that there are few solutions to the equation $x+y=z$ in which the variables come from different integer restricted digit Cantor sets.  Using projection theorems and Newhouse’s gap lemma, Yu \cite{yu_fractal_projections_arxiv} furthermore proves that there are infinitely many sums of powers of five that can be written as sums of powers of three and four.

The first author proved in \cite[Theorem 1.4]{glasscock_marstrand_for_integers} a discrete analogue of Marstrand's projection theorem, building on the work of Lima and Moreira in \cite{limamoreira}: \emph{for all $A, B \subseteq \Z$ satisfying a necessary dimension condition}\footnote{\label{dim_condition_footnote}The condition is that the upper mass dimension of $A \times B$ is equal to the upper counting dimension of $A \times B$. The upper mass dimension is defined in \eqref{eqn_upper_mass_def}, while the upper counting dimension of $A \times B$ is equal to $\limsup_{N \to \infty} \max_{z \in \Z^2} \log \big| (A \times B ) \cap (z + \{-N, \ldots, N\}^2) \big|\big/ \log N$.} \emph{and for Lebesgue-a.e. $(\lambda,\eta) \in \R^2$,}
\begin{align}
\label{marstrand_result_intro_integers}
    \udimdm \big( \lfloor \lambda A +  \eta B \rfloor \big) = \min{}\big( \udimm (A \times B) , \ 1 \big),
\end{align}
where the upper mass dimension, $\udimdm$, is defined in \eqref{eqn_upper_mass_def} below, $\lfloor \, \cdot \, \rfloor$ denotes the floor function, and $\lfloor \lambda A + \eta B \rfloor\coloneqq \big \{\lfloor \lambda a +\eta b\rfloor \ \big | \ a\in A,\, b\in B \big \}$. 
It is reasonable to conjecture by analogy that if $A$ and $B$ are restricted digit Cantor sets with respect to multiplicatively independent bases, then \eqref{marstrand_result_intro_integers} would hold for \emph{all} non-zero $\lambda, \eta \in \R$.  We show that this is indeed the case in \cref{mainthm_integer_sumsets} and its generalizations.

\subsection{Main results}
\label{sec_intro_integers}

Our primary goals for this article are to introduce the study of multiplicatively invariant subsets of the non-negative integers and to bring transversality results in the integers more in line with those in the reals by giving full-fledged analogues of Theorems \ref{thm_furstenberg}, \ref{thm_HS_localentropy_and_intersections}, and \ref{LMP_theorem}.  To that end, we begin by introducing an analogue of a $\times r$-invariant set for the integers.

Let $r \in \N$, $r \geq 2$. Define $\Phi_r\colon \Nz\to \Nz$ and $\Psi_r\colon \Nz\to \Nz$ by
\[
\Phi_r\colon n\mapsto \lfloor n/r\rfloor\qquad\text{and}\qquad
\Psi_r\colon n\mapsto n-r^k\lfloor n/r^k\rfloor,
\]
where $k=\lfloor\log n/\log r \rfloor$ when $n \geq 1$. The maps $\Phi_r$ and $\Psi_r$ are best understood using the base-$r$ representations of non-negative integers: if $n=a_k r^k + \cdots + a_1 r + a_0$, $a_k \neq 0$, is the base-$r$ representation of $n$, then
\[
\Phi_r(n)\,=\,a_{k}r^{k-1} + \cdots + a_2 r + a_1
\qquad\text{and}\qquad
\Psi_r(n)\,=\,a_{k-1}r^{k-1} + \cdots + a_1 r + a_{0}.
\]
In other words, the map $\Phi_r$ ``forgets'' the least significant digit (the right-most digit, hence the letter $\Phi$) while the map $\Psi_r$ ``forgets'' the most significant digit (the left-most digit, hence the letter $\Psi$) in base $r$. For example, in base $r=10$ we have $\Phi_{10}(71393)=7139$ and $\Psi_{10}(71393)=1393$.

\begin{definition}
\label{def_xr_invariant_integers}
A set $A\subseteq\Nz$ is \define{$\times r$-invariant} if $\Phi_r(A)\subseteq A$ and $\Psi_r(A)\subseteq A$.
We call $A\subseteq\Nz$ \define{multiplicatively invariant} if it is $\times r$-invariant for some $r\geq 2$.
\end{definition}

It may be helpful to note that a $\times r$-invariant set $A$ need not satisfy $r A \subseteq A$ and that there are examples showing that the condition $r A \subseteq A$ does not yield a natural integer analogue of the notion of $\times r$-invariance on the unit interval; see \cref{section_counterexample}.

There are many natural examples of $\times r$-invariant subsets of $\Nz$. Integer base-$r$ restricted digit Cantor sets, defined in \eqref{eqn_cantor_digit}, are clearly $\times r$-invariant. More general examples arise from symbolic subshifts of $\{0,1,\ldots,r-1\}^{\Nz}$. For any closed and left-shift-invariant set $\Sigma\subseteq\{0,1,\ldots,r-1\}^{\Nz}$, the corresponding \define{language set} is defined by
\[
\mathcal{L}(\Sigma)=\big\{ w_0w_1 \cdots w_{k} \ \big| \ w_0w_1\cdots \in \Sigma,~k\in\Nz \big\}.
\]
Any language set naturally embeds in two ways into the non-negative integers as
\begin{gather*}
\big\{w_{0} r^{k} + \cdots + w_{k-1} r + w_k \ \big | \ w_0w_1\cdots w_{k} \in \mathcal{L}(\Sigma) \big\},\\
\big\{w_{k} r^{k} + \cdots + w_{1} r + w_0 \ \big | \ w_0w_1\cdots w_{k} \in \mathcal{L}(\Sigma) \big\},
\end{gather*}
yielding sets that are $\times r$-invariant.  For more details, see \cref{def_language-set} and \cref{prop_correspondence_integers_subshifts}, and for more such examples, see \cref{examples_of_times_r_invariant_sets}. As yet another source of $\times r$-invariant subsets of the non-negative integers, we note that if $X$ is a $\times r$-invariant subset of $[0,1]$, then the set
\[
\bigcup_{k\in\Nz}\big\{\lfloor r^kx \rfloor \ \big | \  x\in X \big\}
\]
can be shown to be $\times r$-invariant; see \cref{sec_connection_01-interval} for more details.

Our first result in the integer setting is an analogue of \cref{thm_furstenberg} that demonstrates that there are no non-trivial examples of sets which exhibit structure simultaneously with respect to multiplicatively independent bases.  \cref{prop_discrete_furstenberg} is proved in \cref{sec_proof_of_thm_B} by expanding on the well-known argument that all non-zero decimal digits appear as the most significant digit of $2^n$. We define $[X]_\delta \defeq \{z\in \R \ | \ \exists x\in X~\text{with}~|z - x| \leq \delta \}$ to be the $\delta$-neighborhood of the set $X$.

\begin{Maintheorem}
\label{prop_discrete_furstenberg}
Let $r$ and $s$ be multiplicatively independent positive integers, and let $A, B \subseteq \Nz$ be $\times r$- and $\times s$-invariant sets, respectively. If $\lambda, \eta > 0$, $\sigma, \tau \in \R$ and $\delta > 0$ are such that
\begin{align}
\label{eqn_affine_falls_in_affine_neighborhood}
    \lambda A + \tau \subseteq \big[ \eta B + \sigma \big]_\delta,
\end{align}
then either $A$ is finite or $B = \Nz$.
\end{Maintheorem}

To measure the size of multiplicatively invariant subsets of $\Nz$ and their sumsets and Cartesian products, we make use of two notions of dimension in the integers that parallel the classical Minkowski and Hausdorff dimensions from geometric measure theory. The discrete analogue of the lower and upper Minkowski dimension are the \emph{lower and upper mass dimension}, defined for $A\subseteq\Nz^d$ as
\begin{align}
    \notag
    \ldimdm A &\,=\, \liminf_{N \to \infty}\, \frac{\log |A \cap [0,N)^d|}{ \log N}=\sup \left\{\gamma \geq 0 \ \middle| \ \liminf_{N \to \infty} \frac{\big|A \cap [0,N)^d\big|}{ N^\gamma} > 0 \right\},\\
    \label{eqn_upper_mass_def}
    \udimdm A &\,=\, \limsup_{N \to \infty}\, \frac{\log |A \cap [0,N)^d|}{ \log N}=\sup \left\{\gamma \geq 0 \ \middle| \ \limsup_{N \to \infty} \frac{\big|A \cap [0,N)^d\big|}{ N^\gamma} > 0 \right\}.
\end{align}
Whenever $\ldimdm A = \udimdm A$, we say that the mass dimension of $A$ exists and denote it by $\dimdm A$.
In analogy to the way in which the classical Hausdorff dimension can be defined in terms of the unlimited Hausdorff content (see \cref{sec_continuous_fractal_geometry}), the \emph{lower and upper discrete Hausdorff dimension} of $A$ are defined to be
\begin{align*}
    \ldimdh A &\,=\, \sup \left\{\gamma \geq 0 \ \middle| \ \liminf_{N \to \infty} \frac{\HC_{\geq 1}^\gamma \big(A \cap [0,N)^d \big)}{N^\gamma} > 0 \right\},\\
    \udimdh A &\,=\, \sup \left\{\gamma \geq 0 \ \middle| \ \limsup_{N \to \infty} \frac{\HC_{\geq 1}^\gamma \big(A \cap [0,N)^d \big)}{N^\gamma} > 0 \right\},
\end{align*}
where the \emph{discrete $\gamma$-Hausdorff content}, $\HC_{\geq 1}^\gamma$, is defined in \cref{def_hausdorff_content}.  If these two quantities agree then we say that the \emph{discrete Hausdorff dimension} of $A$, $\dimdh A$, exists and is equal to this quantity. 

The mass dimension and the upper discrete Hausdorff dimension are systematically studied along with a host of other discrete dimensions in \cite{barlow_taylor_92}.
We discuss these notions of dimension and the interplay between them at greater length in \cref{sec_discrete_dimensions}.
For the current discussion, it is helpful to know that
\begin{align*}
    \ldimdh \leq \ldimdm \leq \udimdm \qquad \text{ and } \qquad \ldimdh \leq \udimdh \leq \udimdm,
\end{align*}
and that for any $\times r$-invariant set $A\subseteq\Nz$, both the mass dimension $\dimdm A$ and the discrete Hausdorff dimension $\dimdh A$ exist and coincide; see \cref{lemma_dimensions_coincide_for_invariant_sets}.

Our next main results in the integer setting demonstrate geometric and additive combinatorial transversality between $\times r$- and $\times s$-invariant subsets of integers.
Thus, these results parallel the results of Hochman and Shmerkin, Shmerkin, and Wu by verifying analogues of Furstenberg's intersection and sumset conjectures.

Let $r$ and $s$ be multiplicatively independent positive integers, and let $A, B \subseteq \Nz$ be $\times r$- and $\times s$-invariant sets, respectively. Define $\ogamma = \max \big(0, \ \dimdh A + \dimdh B - 1 \big)$. (In what follows, recall the use of the floor notation $\lfloor \ \cdot \ \rfloor$ described just after \eqref{marstrand_result_intro_integers} above.)

\begin{Maintheorem}
\label{mainthm_integer_intersections}
For all $\eps, \lambda, \eta > 0$, $\sigma, \tau \in \R$, and sufficiently large $N \in \N$,
\[\big| \lfloor \lambda \big( A \cap [0,N) \big) + \tau \rfloor \cap \lfloor \eta \big( B \cap [0,N) \big) + \sigma \rfloor \big| \leq N^{\ogamma + \eps}.\]
In particular, for all $\lambda, \eta > 0$ and $\sigma, \tau \in \R$,
\[\udimdm \big( \lfloor \lambda A + \tau\rfloor \cap \lfloor \eta B + \sigma \rfloor \big)  \leq \max \big(0, \ \dimdh A + \dimdh B - 1 \big).\]
\end{Maintheorem}

The upper bound on the dimension of the set $\lfloor \lambda A + \tau \rfloor \cap \lfloor \eta B + \sigma \rfloor$ in \cref{mainthm_integer_intersections} provides an analogue in the integers to the result of Shmerkin and Wu in \eqref{eqn_furstenberg_intersection} in the reals. \cref{mainthm_integer_intersections} will be derived as a corollary of \cref{theorem_main_discrete_intersections_v2}, a stronger result proved in \cref{sec_intersections_section} in which we demonstrate that the upper bound on $\big| \lfloor \lambda \big( A \cap [0,N) \big) + \tau \rfloor \cap \lfloor \eta \big( B \cap [0,N) \big) + \sigma \rfloor \big|$ is uniform over a compact set of scaling parameters.

Our next theorem gives an integer analogue of the result of Hochman and Shmerkin in \eqref{eqn_in_thm_HS_localentropy}.  
We bound both the cardinality and the discrete Hausdorff content of the set $\lfloor \lambda A' + \eta B' \rfloor$ from below in terms of the cardinality and the discrete Hausdorff content of the product set $A' \times B'$, where $A'$ and $B'$ are arbitrary subsets of $A$ and $B$.
Note that $\dimdh (A \times B) = \dimdh A + \dimdh B$ holds because $A$ and $B$ are multiplicatively invariant (see \cref{cor_dim_of_products_of_invariant_sets}), but this equality need not hold for arbitrary subsets $A' \subseteq A$ and $B' \subseteq B$. Hence, the role played by $\dimdh A + \dimdh B$ in \cref{mainthm_integer_intersections} is now played by $\dimdh (A' \times B')$ in this next result.

\begin{Maintheorem}
\label{mainthm_integer_sumsets}
For all $\eps, \lambda, \eta > 0$, $\gamma \in [0,1]$, sufficiently large $N$  and non-empty $A' \subseteq A \cap [0,N)$, $B' \subseteq B \cap [0,N)$,
\begin{align*}
   \big|  \big\lfloor \lambda A' + \eta B' \big\rfloor \big| &\geq \frac{| A' \times B' |}{N^{\ogamma + \eps}}, \text{ and} \\
   \frac{\HC_{\geq 1}^{\gamma} \big( \big\lfloor \lambda A' + \eta B' \big\rfloor \big)}{N^{\gamma}} &\gg_{\eps,\lambda,\eta, \gamma} \frac{\HC_{\geq 1}^{\gamma + \ogamma + \eps} \big( A' \times B' \big)}{N^{\gamma+\ogamma + \eps}}.
\end{align*}
In particular, for all $\dim \in \{\ldimdm, \udimdm, \ldimdh, \udimdh\}$,
\[    \dim \big( \lfloor \lambda A + \eta B \rfloor \big) = \min \big(1, \dim (A \times B) \big),\]
and, if $\dimdh A+ \dimdh B \leq 1$, then for all $A' \subseteq A$, $B' \subseteq B$,
\[\dim \big( \lfloor \lambda A' + \eta B' \rfloor \big) = \dim \big( A' \times B' \big).\]
\end{Maintheorem}

Just as with \cref{mainthm_integer_intersections}, we derive \cref{mainthm_integer_sumsets} from a more general result, \cref{theorem_main_discrete_sums_v2} proved in \cref{sec_sumsintegers}, which demonstrates that the inequalities in \cref{mainthm_integer_sumsets} hold uniformly over the scaling parameters $\lambda$ and $\eta$. 
Both \cref{mainthm_integer_intersections} and \cref{mainthm_integer_sumsets} are proved by combing the uniformity in Shmerkin's main theorem in \cite{shmerkin} with tools from ergodic theory in an appropriate symbolic dynamic setting.
It remains an interesting question whether there is a direct way of deriving \cref{mainthm_integer_sumsets} from \cref{mainthm_integer_intersections}, in analogy to the continuous setting where it is known that upper bounds on the dimension of fibers imply lower bounds on the dimension of sumsets. 

Our final main result in the integer setting is an analogue of \cref{LMP_theorem} concerning the dimension of iterated sumsets of $\times r$-invariant sets.  Our deduction of \cref{theorem_lmp_analogue} from \cref{LMP_theorem} highlights the flexibility of the machinery developed in this paper to transfer results from the reals to the integers.

\begin{Maintheorem}
\label{theorem_lmp_analogue}
Let $(A_i)_{i =1}^\infty$ be a sequence of $\times r$-invariant subsets of $\Nz$.  If $\sum_{i=1}^\infty \dimdh A_i / \allowbreak | \log \dimdh A_i |$ diverges, then
\begin{align*}
\lim_{n \to \infty} \ldimdh \big( A_1 + \cdots + A_n \big) = 1.
\end{align*}
\end{Maintheorem}

In the same way as in the continuous regime, this theorem demonstrates that the structure captured by $\times r$-invariance in $\Nz$ sits transversely to the additive structure captured by additive closure. It also demonstrates the connection between $\times r$-invariant subsets of the integers and $\times r$-invariant subsets of $[0,1]$, and it will serve to emphasize the role multiplicative independence plays in the other results in this section.

\subsection{Overview of the paper}

The paper is organized as follows.  In \cref{sec_reals}, we derive the intersection and sumset transversality results for multiplicatively invariant subsets of $[0,1]$ from the main result in \cite{shmerkin}.  We begin \cref{sec_discrete_fractal_geometry} with the basic facts and results we need from discrete fractal geometry in \cref{sec_discrete_dimensions} and continue by connecting $\times r$-invariant subsets of $\Nz$ to symbolic dynamics and multiplicatively invariant subsets of the reals. \cref{sec_discrete_fractal_geometry} lays the groundwork for \cref{sec_integer_transversality}, where we prove our main results: Theorems \ref{prop_discrete_furstenberg}, \ref{mainthm_integer_intersections}, \ref{mainthm_integer_sumsets}, and \ref{theorem_lmp_analogue}.  We construct an example in \cref{section_counterexample} that demonstrates that \cref{mainthm_integer_sumsets} is not expected to hold under weaker assumptions. Finally, we conclude the paper with \cref{sec_future} by outlining a number of open problems and directions.

\subsection{Acknowledgements}

The authors extend a debt of gratitude to Pablo Shmerkin and Mark Pollicott, whose insightful questions led to improvements in early formulations of \cref{mainthm_integer_sumsets}, and to Sam Chow for help with \cref{subsec_polynomial}.  
We also thank Vicente Saavedra Araya who pointed out a mistake in an earlier draft of this paper.
The third author is supported by the Swiss National Science Foundation under grant number TMSGI2-211214.

\section{Sums and intersections of multiplicatively invariant subsets of the reals}
\label{sec_reals}

In this section, we prove that subsets of $[0,1]$ that are multiplicatively invariant with respect to multiplicatively independent bases are both geometrically and additive combinatorially transverse.  
Our theorems are derived from the main result of Shmerkin \cite{shmerkin}, but we give particular care on emphasizing the ``uniformity'' in the parameters. 
While most of the results in this section are already implicit in the literature, we spell out the full details to have the precise statements we need, and we provide complete proofs for the benefit of non-experts.

This is the only section in the paper in which we draw on classical fractal geometry, so we begin by establishing the basic terminology and results.

The set of real numbers, $\R$, is equipped with the usual Euclidean metric, and, for convenience, all product spaces in the work are endowed with the $L^1$ (taxicab) metric. The distance between $x,y \in \R^d$ is denoted by $|x-y|$, and the open ball centered at $x$ with radius $\delta$ is denoted $B(x,\delta)$.  Throughout the paper, a \emph{measure} refers to a non-negative-valued Radon measure on $\R^d$. The total mass of a measure $\mu$ is $\|\mu\| \defeq \mu(\R^d)$, and its support is denoted $\supp \mu$.  The push-forward of $\mu$ under a map $\varphi$ is denoted $\varphi \mu$, so that $\varphi \mu (B) = \mu (\varphi^{-1} B)$ for all measurable sets $B$.

Finally, given two positive-valued functions $f$ and $g$, we write $f \ll_{a_1,\ldots,a_k} g$ if there exists a constant $c > 0$ depending only on the quantities $a_1, \ldots, a_k$ for which $f(x) \leq c g(x)$ for all $x$ in the domain common to both $f$ and $g$.  We write $f \asymp_{a_1,\ldots,a_k} g$ if both $f \ll_{a_1,\ldots,a_k} g$ and $f \gg_{a_1,\ldots,a_k} g$.

\subsection{Fractal geometry of sets and measures in Euclidean space}
\label{sec_continuous_fractal_geometry}

In this subsection, we give a terse summary of the necessary notation, terminology, and basic results from traditional fractal geometry.  The reader interested in learning more will find most of this material in Mattila \cite[Ch. 4]{mattila_geometry_of_sets_1995}.  Throughout this subsection, $\rho$ and $\gamma$ are positive real numbers and $X \subseteq \R^d$ is non-empty. 

\begin{definition} \leavevmode 
\begin{itemize}
    \item The set $X$ is \emph{$\rho$-separated} if for all distinct $x_1, x_2 \in X$, $|x_1 - x_2| \geq \rho$.
    \item The \emph{packing number of $X$} (sometimes also called the metric entropy of $X$) at scale $\rho$ is
\[\NC(X,{\rho}) = \sup \big \{ |X_0| \ \big | \ X_0 \subseteq X \text{ is } \rho \text{-separated} \big \}.\]
    \item The \emph{upper Minkowski dimension} of $X$ is
    \[
    \udimm X = \limsup_{\rho \to 0^+} \frac{\log \NC (X, \rho)}{\log \rho^{-1}}.
    \]
    The \emph{lower Minkowski dimension}, $\ldimm X$, is defined analogously with a limit infimum in place of the limit supremum.  If the lower and upper Minkowski dimensions agree, then that value is the \emph{Minkowski dimension} of $X$, $\dimm X$.  It is easy to check that for all $\rho < 1$, $\udimm X = \limsup_{N \to \infty} \log \NC (X, \rho^{-N}) \big/ \log \rho^N$ and similarly for $\ldimm X$.
\end{itemize}
\end{definition}

\begin{definition}\leavevmode
\label{def_hausdorff_content}
\begin{itemize}
    \item The \emph{discrete Hausdorff content of $X$ at scale $\rho$ and dimension $\gamma$} is
    \[\HC_{\geq \rho}^\gamma (X) = \inf \left\{ \sum_{i\in I} \delta_i^\gamma \ \middle| \ X \subseteq \bigcup_{i\in I} B_i, \ B_i \text{ open ball of diameter } \delta_i \geq \rho \right\}.\]
    \item The \emph{unlimited Hausdorff content at dimension $\gamma$} of $X$ is
    \[\HC_{>0}^\gamma (X) = \inf \left\{ \sum_{i \in I} \delta_i^\gamma \ \middle| \ X \subseteq \bigcup_{i \in I} B_i, \ B_i \text{ open ball of diameter } \delta_i > 0 \right\}. \]
    \item The \emph{Hausdorff dimension} of $X$ is
    \begin{align*}
        \dimh X = \sup \{ \gamma \in \R \ | \ \HC_{>0}^\gamma (X) > 0 \} =\inf \{ \gamma \in \R \ | \ \HC_{>0}^\gamma (X) = 0 \}.
    \end{align*}
\end{itemize}
\end{definition}

Note that if $X$ is compact, the index set $I$ in the definitions of $\HC_{\geq \rho}^\gamma (X)$ and $\HC_{>0}^\gamma (X)$ may be taken to be finite.

\begin{remark}
\label{lemma_connection_between_discrete_and_unlimited_H_contents}
The discrete Hausdorff content tends to the unlimited Hausdorff content in the limit as the scale tends to zero.  More precisely, for $X \subseteq \R^d$ compact and $\gamma \geq 0$,
\begin{align*}
    \lim_{\rho \to 0^+} \HC_{\geq \rho}^\gamma (X) = \HC_{>0}^\gamma(X).
\end{align*}
It follows that if $\lim_{\rho \to 0} \HC_{\geq \rho}^\gamma (X)  > 0$, then $\dimh X \geq \gamma$.  The proof is straightforward; see \cite[Lemma 2.4]{GMR_new_proof_2020}.
\end{remark}

Recall the notation $[X]_\delta$ for the $\delta$-neighborhood of $X$: \[[X]_\delta \defeq \big\{z\in \R^d \ \big| \ \exists x\in X~\text{with}~ |z - x| \leq \delta \big\}.\]
The \emph{Hausdorff distance} between two non-empty, compact sets $X, Y \subseteq \R^d$ is
\[
d_{H}(X,Y)\coloneqq \inf\big\{\delta >0 \ \big| \ X\subseteq [Y]_\delta~\text{and}~Y\subseteq [X]_\delta\big\}.
\]
By the Blaschke selection theorem, the set of all non-empty, compact subsets of $\R^d$ equipped with the Hausdorff distance is a complete metric space.

\begin{lemma}
\label{lemma_subsets_for_hausdorff_distance}
Suppose $X, Y \subseteq \R^d$ are non-empty, compact and $X \subseteq [Y]_\delta$.  For all non-empty, compact $X' \subseteq X$, there exists a non-empty, compact $Y' \subseteq Y$ such that $d_H(X',Y') \leq \delta$.
\end{lemma}

\begin{proof}
Define $Y' = Y \cap [X']_\delta$.  By definition, the set $Y'$ is compact and $Y' \subseteq [X']_\delta$.  Since $X' \subseteq [Y]_\delta$, the set $Y'$ is non-empty.  To see that $X' \subseteq [Y']_\delta$, let $x \in X'$.  Since $X \subseteq [Y]_\delta$, there exists $y \in Y$ such that $|x-y|\leq \delta$.  This implies that $y \in Y \cap [X']_\delta$, which shows that $x \in [Y \cap [X']_\delta]_\delta = [Y']_\delta$, as was to be shown.
\end{proof}

We proceed with a number of straightforward lemmas that describe how the packing number and discrete Hausdorff content behave as functions of the set and the scale.
We include full proofs for completeness.

\begin{lemma}
\label{lemma_hausdorff_near_implies_contents_near}
\label{lem_box_counting_estimate}
For all $a, \rho > 0$, all non-empty, compact sets $X, Y \subseteq \R^d$ satisfying $X\subseteq [Y]_{a\rho}$, and all $\gamma \in [0,d]$,
\begin{align}
\label{eqn_metric_entropy_of_near_sets}
\NC \big(X, \rho \big) \,&\ll_{a,d} \, \NC \big(Y, \rho \big), \\
\label{eqn_hausdorff_content_of_near_sets}
\HC_{\geq \rho}^\gamma (X) \, &\ll_{a,d}\, \HC_{\geq \rho}^\gamma (Y).
\end{align}
\end{lemma}

\begin{proof}
Let $X' \subseteq X$ be a maximal $\rho$-separated subset of $X$. Define a map $\pi: X' \to Y$ by choosing for each point $x \in X'$ a point $\pi x \in Y$ such that $|x - \pi x| \leq a \rho$. Define $Y' = \pi X'$.  Since $X'$ is $\rho$-separated, there are at most $C = C(a,d) > 0$ many points of $X'$ in any closed ball of radius $(a+1)\rho$.  It follows that the map $\pi$ is at most $C$-to-1, and hence that $|Y'| \gg_{a,d} |X'|$.  It also follows that there are at most $C$ many points of $Y'$ in any closed ball of radius $\rho$.  Therefore, the set $Y'$ can be thinned to a set $Y'' \subseteq Y'$ that is  $\rho$-separated and that satisfies $|Y''| \gg_{a,d} |Y'|$.  Combining these observations,
\[\NC \big(X, \rho \big) = |X'| \ll_{a,d} |Y'| \ll_{a,d} |Y''| \leq \NC \big(Y, \rho \big)\]
which verifies \eqref{eqn_metric_entropy_of_near_sets}.

To show \eqref{eqn_hausdorff_content_of_near_sets}, let $\{B_i\}_{i\in I}$ be a collection of open balls that covers $Y$ and where $B_i$ has diameter $r_i\geq \rho$ and
$\sum_{i\in I}r_i^\gamma<2\HC_{\geq\rho}^\gamma(Y)$.
It follows that $X\subseteq \bigcup_{i\in I} [B_i]_{a\rho}$ and $[B_i]_{a\rho}$ is a ball of diameter $r_i+2a\rho\leq (2a+1)r_i$.
Therefore $\HC_{\geq\rho}^\gamma(X)\leq\sum_{i\in I}((2a+1)r_i)^\gamma\leq 2(2a+1)^d\HC_{\geq\rho}^\gamma(Y)$.
\end{proof}

\begin{lemma}
\label{lemma_metric_entropy_at_scaled_scale}
For all $a, \rho > 0$, all non-empty, compact $X \subseteq \R^d$, and all $\gamma \in [0,d]$,
\begin{align*}
\NC \big( X, \rho \big) &\asymp_{a,d} \NC \big( X, a \rho \big), \\
\HC_{\geq \rho}^\gamma \big( X \big) &\asymp_{a,d} \HC_{\geq a \rho}^\gamma \big( X \big).
\end{align*}
\end{lemma}

\begin{proof}
Replacing $\rho$ with $a \rho$, we may assume without loss of generality in both statements that $0 < a \leq 1$.

Since $0 < a \leq 1$, we have that $\NC \big( X, \rho \big) \leq \NC \big( X, a \rho \big)$.  To see the reverse inequality, let $X' \subseteq X$ be a maximal $(a\rho)$-separated subset of $X$. Since the set $X'$ intersects any ball of diameter $\rho$ in at most $\ll_{a,d} 1$ many points, it may be thinned to an $\rho$-separated subset $X''$ of $X'$ with cardinality $|X''| \gg_{a,d} |X'|$.  Therefore, $\NC(X, a\rho) = |X'| \ll_{a,d} |X''| \leq \NC(X, \rho)$.

Since $0 < a \leq 1$, we have that $\HC_{\geq a \rho}^\gamma \big( X \big) \leq \HC_{\geq \rho}^\gamma \big( X \big)$. To see the reverse inequality, let $X \subseteq \cup_i B_i$ be an open cover of $X$ by balls $B_i$ with $\diam B_i \geq a \rho$ and $\sum_i (\diam B_i)^\gamma \leq 2 \HC_{\geq a \rho}^\gamma(X)$.  Replace $B_i$ with an open ball $C_i$ with the same center and with diameter $\diam B_i / a$.  Since $B_i \subseteq C_i$, we have that $X \subseteq \cup_i C_i$ is an open cover of $X$ by balls $C_i$ with $\diam C_i \geq \rho$. Therefore,
\[\HC_{\geq \rho}^\gamma (X) \leq \sum_i (\diam C_i)^\gamma = a^{-\gamma} \sum_i (\diam B_i)^\gamma \leq 2 a^{-\gamma} \HC_{\geq a \rho}^\gamma(X),\]
as was to be shown.
\end{proof}

\begin{lemma}
\label{lemma_hausdorff_content_bound_under_lipschitz_map}
For all $\rho > 0$, all non-empty, compact $X \subseteq \R^d$, all Lipschitz $\varphi: \R^d \to \R^k$ with Lipschitz constant $a > 0$, and all $\gamma \in [0,d]$,
\begin{align*}
\NC (\varphi (X), \rho) &\ll_{a,d} \NC (X, \rho), \\
\HC_{\geq \rho}^\gamma (\varphi (X)) &\ll_{a,d} \HC_{\geq \rho}^\gamma (X).
\end{align*}
\end{lemma}

\begin{proof}
Let $X' \subseteq X$ be such that $\varphi(X')$ is a maximal $\rho$-separated subset of $\varphi(X)$.  Since $\varphi$ has Lipschitz constant $a$, the points of $X'$ are $\rho / a$-separated.  Thus, by \cref{lemma_metric_entropy_at_scaled_scale},
\[\NC (\varphi (X), \rho) = |X'| \leq \NC (X, \rho / a) \ll_{a,d} \NC (X, \rho).\]
verifying the first inequality.

To see the second inequality, note that if $B $ is an open ball in $\R^d$, then the diameter of $\varphi (B)$ is bounded from above by $a\cdot\diam B$. Hence, there exists an open ball $C \subseteq \R^k$ with $\diam B\leq \diam C \leq \max(a,1) \diam B$ and such that $\varphi(B)\subseteq C$.

If $\cup_i B_i$ is a cover of $X$ by open balls $B_i$ with $\diam B_i \geq \rho$ then, finding for each $B_i$ a ball $C_i$ as described above, we obtain a cover $\cup_i C_i$ of the image set $\varphi(X)$ by open balls $C_i\subseteq\R^k$ with $\rho\leq \diam C_i \leq \max(a,1) \diam B_i$.
It follows that
\[\HC_{\geq \rho}^\gamma (\varphi (X)) \leq \max(a,1)^\gamma\,  \HC_{\geq \rho}^\gamma (X),\]
as was to be shown.
\end{proof}

\begin{definition}
The real number $\gamma$ is a \emph{Frostman exponent for a measure $\mu$} if there exists a constant $c > 0$ such that for all balls $B \subseteq \R^d$,
\begin{align}
\label{eqn:basic_frostman_inequality}
    \mu \big( B \big) \leq c (\diam B)^\gamma.
\end{align}
If \eqref{eqn:basic_frostman_inequality} holds only for balls $B$ of diameter greater than / less than $\rho$, then $\gamma$ is a \emph{Frostman exponent at scales larger than / smaller than $\rho$}, respectively.
\end{definition}

The following lemmas are discrete versions of the well-known mass distribution principle and Frostman's lemma.  This pair of results describes a close relationship between the discrete Hausdorff content of a set and the Frostman exponents of measures supported on that set.

\begin{lemma}[{cf. \cite[Lemma 1.2.8]{bishopperesbook}}]
\label{lem_consequence_of_discrete_frostman}
Let $c, \rho >0$ and $\mu$ be a measure on $\R^d$. If for all balls $B \subseteq \R^d$ of diameter at least $\rho$ we have $\mu (B) \leq c (\diam B)^\gamma$, then $\HC_{\geq \rho}^\gamma(\supp \mu) \geq \|\mu\| / c$.
\end{lemma}

\begin{proof}
Let $\eps > 0$, and let $\{B_i\}_{i \in I}$ be a cover of $\supp \mu$ with balls $B_i$ of diameter $\delta_i \geq \rho$ and with $\sum_{i \in I} \delta_i^\gamma \leq (1+\eps) \HC_{\geq \rho}^\gamma(\supp \mu )$. Then
\[\| \mu \| \leq \mu \left( \bigcup_i B_i \right) \leq \sum_i c \delta_i^\gamma \leq c (1+\eps) \HC_{\geq \rho}^\gamma(\supp \mu ).\]
The conclusion follows because $\eps > 0$ was arbitrary.
\end{proof}

\begin{lemma}
\label{lemma_finitary_frostman}
There exists a constant $c > 0$, depending only on the dimension $d \in \N$, for which the following holds. For all non-empty, compact $X \subseteq [0,1]^d$ and all $\rho,\gamma > 0$ there exists a measure $\mu$ supported on $X$ with $\| \mu \| \geq \HC_{\geq \rho}^\gamma (X)$ and with the property that for all balls $B$ of diameter at least $\rho$, $\mu(B) \leq c (\diam B)^\gamma$.
\end{lemma}

\begin{proof}
This requires only a small modification to the proof of Frostman's Lemma found in \cite[Lemma 3.1.1]{bishopperesbook}. 
By adjusting the constant $c$, it suffices to prove the lemma for $\rho$ of the form $2^{-k}$. 
Construct the $2$-adic tree corresponding to the set $X$ down to level $k$.
More precisely, the vertices of the tree at level $\ell$ are the closed, $2$-adic cubes of the form 
$$\left[\frac{i_1}{2^\ell},\frac{i_1+1}{2^\ell}\right]\times\cdots\times\left[\frac{i_d}{2^\ell},\frac{i_d+1}{2^\ell}\right]\text{ for some }i_1,\dots,i_d\in\{0,\dots,2^\ell-1\}$$
that have non-empty intersection with the set $X$.
Two vertices are adjacent in the tree if one of the corresponding cubes contains the other. 
Associate to each leaf $v$ (i.e. a vertex at level $k$) of the tree an arbitrary point $x_v$ in $X$ which belongs to the corresponding $2$-adic cube.  

Instead of defining a measure $\mu$ on the space of infinite paths through the tree as is done in \cite{bishopperesbook}, we define $\mu$ to be an atomic measure supported on the finite set $S=\{x_v \ | \ v\text{ is a leaf}\}$ that are associated to leaves of the tree.

Let $E$ be the set of edges in the tree.
We define an edge conductance (or capacity) function $c:E\to[0,1]$ as follows:
an edge $e$ connecting vertices on levels $\ell-1$ and $\ell$ is given an edge conductance of $c(e)=2^{-\ell\gamma}$.
Fix a maximal flow $f:E\to[0,1]$ from the root of the tree to the leaves.
This means that for every vertex $v$ of the tree which is neither the root nor one of the leafs, the sum of $f(e)$ over all edges connecting $v$ to a vertex at a higher level equals the value of $f$ on the (unique) edge connecting $v$ to a vertex of a lower level.
Moreover, $f$ is restricted by the conductance (so that $f(e)\leq c(e)$ for all $e\in E$) and attains the highest possible value (among all such flows $f$) of the sum over all edges connecting to a leaf.
Define the $\mu$-mass of each point $x_v\in S$ to be equal to the value of $f$ on the (unique) edge adjacent to the leaf $v$.

Every 2-adic cube $B$ with $2^{-k} \leq \diam B \leq 2^{-1}$ and with non-empty intersection with $X$ corresponds to an edge in the tree. By the choice of edge conductance and the fact that the maximal flow is a legal flow, $\mu(B) \leq (\diam B)^\gamma$. (Note that this inequality also holds for $B = [0,1]^d$.) A $2$-adic grid cover of $X$ with cells of diameter at least $2^{-k}$ corresponds to a cut-set of the tree.  By the MaxFlow-MinCut theorem, the measure $\mu$ has total mass equal to the minimum cut, which is necessarily greater than $\HC_{\geq 2^{-k}}^\gamma (X)$, concluding the proof of the lemma.
\end{proof}

\subsection{Multiplicatively invariant sets and restricted digit Cantor sets}
\label{sec_invariant_sets_and_missing_digit_cantor_sets}

In this short subsection, we record some basic facts about multiplicatively invariant subsets of $[0,1]$ and the subclass of restricted digit Cantor sets.

\begin{definition}
\label{def_mult_on_torus_def}
Let $r \in \N$, $r \geq 2$, and $X \subseteq [0,1]$. 
\begin{itemize}
    \item The map $T_r: [0,1] \to [0,1]$ is defined by $T_r x = \{rx\}$, the fractional part of the real number $rx$.
    \item The set $X$ is \emph{$\times r$-invariant} if it is closed and $T_r X \subseteq X$.
    \item The set $X$ is \emph{multiplicatively invariant} if it is $\times r$-invariant for some $r \geq 2$.
\end{itemize}
We stress that, by our definition, all multiplicatively invariant sets are closed.
\end{definition}

Multiplicatively invariant sets behave well in regards to dimension: their Hausdorff and Minkowski dimensions agree, and so by \cite[Theorem 8.10]{mattila_geometry_of_sets_1995} the dimension of a Cartesian products of such sets is the sum of the dimensions of the factors.

\begin{lemma}
\label{thm:furstenbergdimsofsubshift}
\label{dimension_of_product_of_invariant_sets}
If $X, Y \subseteq [0,1]$ are multiplicatively invariant, then $\dimh X = \dimm X$ and $\dimh (X \times Y) = \dimm (X \times Y) = \dimh X + \dimh Y$.
\end{lemma}

\begin{proof}
The first fact is proven in \cite[Proposition III.1]{furstenbergdisjointness}.  The second follows immediately from \cite[Corollary 8.11]{mattila_geometry_of_sets_1995} and the fact that $\dimh X = \udimm X$.
\end{proof}

Restricted digit Cantor sets are important examples of multiplicatively invariant sets, and the natural Bernoulli measures they support will play an important role in the theorems in this section.

\begin{definition} \leavevmode
\begin{itemize}
    \item The \emph{base-$r$ restricted digit Cantor set with digits from $\DC \subseteq \{0, \dots, r-1\}$} is
    \[\CC_{r, \DC} = \left\{ \sum_{i=1}^\infty \frac{d_i}{r^i} \ \middle | \ (d_i)_{i \in \N} \subseteq \DC \right\},\]
    the set of those real numbers in $[0,1]$ expressible in base-$r$ using only digits from $\DC$.
    \item The \emph{base-$r$ restricted digit Cantor measure with digits from $\DC \subseteq \{0, \dots, r-1\}$}, denoted $\cmeas_{r, \DC}$, is the $(1/|\DC|)$-Bernoulli measure on $\CC_{r, \DC}$,  defined as
    $$\mu_{r,\DC}\left(\left[\frac j{r^i},\frac{j+1}{r^i}\right)\right)
    =
    \begin{cases}|\DC|^{-i}&\text{ if }\left[\frac j{r^i},\frac{j+1}{r^i}\right)\cap \CC_{r, \DC}\neq\emptyset\\
    0&\text{otherwise}
    \end{cases}.$$
    \item The \emph{dimension}\footnote{There are many natural and useful ways to define the dimension of a measure.  In this paper, we will need only to consider the dimension of products of restricted digit Cantor measures, a class of measures for which most notions of dimension coincide.  Thus, we define ``$\dim \mu$'' for such measures $\mu$ in a highly specialized way instead of giving a general definition of the symbol.} of the measure $\cmeas_{r, \DC}$ is $\dim \cmeas_{r, \DC} \defeq \log | \DC | / \log r$.  
    We also define the dimension of a product of such measures to be the sum of the dimensions of the factors.    
\end{itemize}
\end{definition}

The dimensions of a product of restricted digit Cantor sets $\CC_{r, \DC_r} \times \CC_{s, \DC_s}$ and of its associated product measure $\mu \defeq \cmeas_{r, \DC_r} \times \cmeas_{s, \DC_s}$ coincide and are equal to $\log r / \log |\DC_r| + \log s / \log |\DC_s|$.  In fact, such a measure $\mu$ is highly regular, in the sense that for all balls $B \subseteq \R^2$ of diameter $0 < \delta < 1$ centered at a point in the support of $\mu$,
\begin{align}
    \label{lemma_dimension_of_cantor_sets_and_measures}
\mu (B) \asymp \delta^{\dim \mu},
\end{align}
where the asymptotic constants are independent of $\delta$.  This follows from the fact that such an estimate holds for single restricted digit Cantor measures, an easy exercise left to the reader.

While multiplicatively invariant sets can be vastly more complicated than restricted digit Cantor sets, the following lemma shows that the former can be approximated from above (with respect to dimension) by the latter.  The result is well known; for a proof, see \cite[Prop. 9.3]{wu}.

\begin{lemma}
\label{lem_embedding_into_missing_digit_set}
Let $X \subseteq [0,1]$ be multiplicatively invariant.  For all $\eps > 0$, there exists a restricted digit Cantor set $X'$ containing $X$ such that $\dimh X' < \dimh X + \eps$.
\end{lemma}

\subsection{A uniform Frostman exponent projection theorem}
\label{sec_frost_exp_projection_theorem}

For $t \in \R$, denote by $\Pi_t: \R^2 \to \R$ the oblique projection $\Pi_t(x,y) = x + ty$.
The goal in this subsection is to prove \cref{thm_uniform_frost_exp_for_projections} below, which is a result about Frostman exponents for oblique projections of products of restricted digit Cantor measures.
This theorem follows implicitly from the results in \cite{shmerkin}, but since the exact statement does not appear in the literature, we provide a complete proof. 
We stress the uniformity over the projection parameter $t$, which will be crucial to our applications later.

\begin{theorem}
\label{thm_uniform_frost_exp_for_projections}
Let $\mu$ be the product of two restricted digit Cantor measures whose bases are multiplicatively independent. For all compact $I \subseteq \R \setminus \{0\}$ and all $0<\gamma< \min(\dim \mu, 1)$, there exists $c>0$ such that for all $\rho \in [0,1]$,
\[ \sup_{t \in I , \ x \in \R} \ 
\Pi_t \mu \big( B(x, \rho) \big) \leq c \rho^{\gamma}.\]
\end{theorem}

Let $2 \leq r < s$ be multiplicatively independent integers, $\DC_r\subseteq\{0,\dots,r-1\}$ and $\DC_s\subseteq\{0,\dots,s-1\}$ sets of digits, and $\CC_{r,\DC_r} \subseteq[0,1]$ and $\CC_{s,\DC_s} \subseteq[0,1]$ the base-$r$ and base-$s$ restricted digit Cantor sets with allowed digits $\DC_r$ and $\DC_s$, respectively.
Let $\cmeas_{r, \DC_r}$ and $\cmeas_{s, \DC_s}$ the restricted digit Cantor measures on $\CC_{r,\DC_r}$ and $\CC_{s,\DC_s}$, respectively, and let $\mu = \cmeas_{r, \DC_r} \times \cmeas_{s, \DC_s}$.

We will prove \cref{thm_uniform_frost_exp_for_projections} for the measure $\mu$ by first proving the following theorem, which we derive from a careful application of Shermkin's recent $L^q$-dimension projection theorem \cite[Theorem 1.11]{shmerkin}. Denote by $\dyad_m$ the dyadic partition of $\R$ into intervals of length $2^{-m}$, and denote by $\log$ the base-$2$ logarithm.

\begin{theorem}
\label{theorem_shmerkin_lq_projection}
For all $q \in (1, \infty)$ and all compact $I \subseteq \R \setminus \{0\}$,
\[\lim_{m \to \infty}  \sup_{t \in I} \left| - \frac{\log \sum_{Q \in \dyad_m} \Pi_t \mu (Q)^q }{ (q-1)m} - \min(\dim \mu, 1) \right| = 0.\]
\end{theorem}

\begin{proof}
It suffices to prove \cref{theorem_shmerkin_lq_projection} for intervals $I \subseteq (0,\infty)$. Indeed, note that the set $1-\CC_{s,\DC_s} = \CC_{s, \tilde{\DC_s}}$ is a base-$s$ restricted digit Cantor set with digits from $\tilde{\DC_s} = s-1-\DC_s$ whose associated restricted digit Cantor measure $\cmeas_{s,\tilde{\DC_s}}$ is the image of the measure $\cmeas_{s, \DC_s}$ under $x \mapsto 1-x$.  
It follows that for $t < 0$, $\Pi_{t} \mu$ is a translate of $\Pi_{-t} (\cmeas_{r, \DC_r} \otimes \cmeas_{s, \tilde{\DC_s}})$, a measure which satisfies the conclusion of the theorem. To prove the theorem for $I\subseteq(0,\infty)$, it suffices to prove it for every interval $I$ of the form $I=[\xi,\xi s)$, where $\xi>0$, since every compact subset of $(0,\infty)$ is contained in a finite union of intervals of this form.

Let $\xi > 0$ and $\lambda = 1/r$.  
Let $T: [0,1) \to [0,1)$ be the irrational rotation by $\beta=\log r/\log s$, $Tx = x+\beta\bmod1$.
For $t>0$, let $S_t:\R\to\R$ denote the multiplication by $t$.
Let $\Delta_r$ and $\Delta_s$ be the normalized counting measures on $\DC_r$ and $\DC_s$, respectively, and for $x \in [0,1)$, define
\[\Delta(x)=\begin{cases}\Delta_r&\text{ if }x\geq\beta
\\
\Delta_r* S_{\xi s^x} \Delta_s&\text{ if }x<\beta\end{cases}.\]
Given $x \in [0,1)$ and $n \in \N$, define $\mu_{x,0}=\delta_0$ and $\mu_{x,n} = \mu_{x,n-1}*S_{\lambda^n}\Delta(T^nx)$, where we denote by $S_\lambda \nu$ the pushforward of the measure $\nu$ under $S_\lambda$.
To each $x\in X$, we associate the measure
\[\mu_x=\conv_{n=1}^\infty S_{\lambda^n}\Delta(T^nx) = \lim_{n\to\infty}\mu_{x,n}.\]

The tuple $([0,1), T, \Delta, \lambda)$ is an example of what Shmerkin calls a ``pleasant model''  (\cite[Definition 1.9]{shmerkin}).
As such, it follows from \cite[Theorem 1.11]{shmerkin} that for all $q \in (1, \infty)$,
\begin{align}
    \label{eq_pleasant}
    \lim_{m \to \infty}  \sup_{x \in [0,1)} \left| - \frac{\log \sum_{Q \in \dyad_m}\mu_x(Q)^q }{ (q-1)m} - \min(\alpha, 1) \right| = 0,
\end{align}
where
\[\alpha= \alpha (q) = \frac1{(q-1)\log{\lambda}}\int_0^1\log\|\Delta(x)\|_q^q\dd x\]
and $\|\nu\|_q^q=\sum_{y\in\R}\nu(\{y\})^q$.  To finish the proof of \cref{theorem_shmerkin_lq_projection}, we will show that for all $x \in [0,1)$ and all $q > 1$, $\mu_x = \Pi_{\xi s^x}\mu$ and $\alpha = \dim \mu$.

To see that for each $x\in [0,1)$ the measure $\mu_x$ is equal to $\Pi_{\xi s^x}\mu$, observe first that
\begin{align}
\notag
\mu_{x,n}&=\mu_{x,n-1}*S_{\lambda^n}\Delta(T^nx)\\
&=\begin{cases}
\mu_{x,n-1}*S_{r^{-n}}\Delta_r&\text{ if }\{x+n\beta\}\geq\beta
\\
\label{pleasant1}
\mu_{x,n-1}*S_{r^{-n}}\left(\Delta_r * S_{\xi s^{\{x+n\beta\}}}\Delta_s\right)&\text{ if }\{x+n\beta\}<\beta
\end{cases}.
\end{align}
Note that 
\[{r^{-n}}s^{\{x+n\beta\}}=s^{-n\beta}s^{\{x+n\beta\}}=s^{x}s^{-\lfloor x+n\beta\rfloor}.\]
Borrowing notation from Shmerkin, let $n'(x) \defeq \lfloor x+n\beta\rfloor$; it is a function of both $x$ and $n$.
Note that $n'(x)$ can equivalently be described as the cardinality of the set $\{i\in\{1,\dots,n\} \ | \ \{x+i\beta\}<\beta\}$.
Now \eqref{pleasant1} becomes
\begin{align}
\label{pleasant2}
\mu_{x,n}=\begin{cases}
\mu_{x,n-1}*S_{r^{-n}}\Delta_r&\text{ if }\{x+n\beta\}\geq\beta
\\
\mu_{x,n-1}*S_{r^{-n}}\Delta_r* S_{s^{-n'(x)}\xi s^x}\Delta_s&\text{ if }\{x+n\beta\}<\beta
\end{cases}.
\end{align}
Since convolution is commutative, the fact that the orbit $\{Tx,\dots,T^nx\}$ visits $[0,\beta)$ exactly $n'(x)$ times and \eqref{pleasant2} imply that
\[\mu_{x,n}=\left(\conv_{i=1}^n r^{-i}\Delta_r\right)* S_{\xi s^x}\left(\conv_{i=1}^{n'(x)}S_{s^{-i}}\Delta_s\right).\]
Now for all $x \in[0,1)$,
\[\lim_{n\to\infty}\conv_{i=1}^n S_{r^{-i}}\Delta_r= \cmeas_{r, \DC_r} \qquad\text{and}\qquad \lim_{n\to\infty} \conv_{i=1}^{n'(x)}S_{s^{-i}}\Delta_s = \cmeas_{s, \DC_s},\]
which proves that for every $x\in[0,1)$, $\mu_x = \lim_{n\to\infty}\mu_{x,n} = \cmeas_{r, \DC_r} \ast S_{\xi s^x }\cmeas_{s, \DC_s} =\Pi_{\xi s^x}\mu$, as claimed.

To finish the proof, it remains to show that the value $\alpha$ in \eqref{eq_pleasant} equals the dimension of $\mu$, which is $\dim\mu=\dim \cmeas_{r, \DC_r} + \dim \cmeas_{s, \DC_s} =\frac{\log|\DC_r|}{\log r}+\frac{\log|\DC_s|}{\log s}$. 
Note that for almost every $x<\beta$, $\|\Delta(x)\|_q^q=\sum_{i\in\DC_r}\sum_{j\in\DC_s}\left(\frac1{|\DC_r||\DC_s|}\right)^q=|\DC_r|^{1-q}|\DC_s|^{1-q}$, and for all $x\geq\beta$, $\|\Delta(x)\|_q^q=\sum_{i\in\DC_r}\left(\frac1{|\DC_r|}\right)^q=|\DC_r|^{1-q}$.
Therefore, by the definition of $\alpha$,
\begin{align*}
\alpha &= \frac1{(q-1)\log\lambda}\int_0^1\log\|\Delta(x)\|_q^q\dd x \\
& =
\frac1{(1-q)\log r}\left(\int_0^\beta\log\|\Delta(x)\|_q^q\dd x+\int_\beta^1\log\|\Delta(x)\|_q^q\dd x\right) \\
& =\frac1{\log r}\Big(\beta\big(\log|\DC_r|+\log|\DC_s|\big)+(1-\beta)\log|\DC_r|\Big) \\
&= \frac{\log|\DC_r|}{\log r}+\frac{\log|\DC_s|}{\log s},
\end{align*}
as was to be shown.
\end{proof}

Though we have not developed the terminology for it, the conclusion in \cref{theorem_shmerkin_lq_projection} concerns the $L^q$-dimension of the images of $\mu$ under oblique projections.  The following lemma allows us to derive from \cref{theorem_shmerkin_lq_projection} a statement concerning Frostman exponents of the projected measures.

\begin{lemma}[cf. {\cite[Lemma 1.7]{shmerkin}}]
\label{lemma_frostman_exponent_from_lq_dim}
Let $\mu$ be a probability measure on $\R$, $q > 1$, and $\gamma > 0$.  If for all $m \geq M$,
\begin{align}
\label{eqn_lower_bound_on_finite_lq_dim}
    \frac{-\log \sum_{Q \in \dyad_m} \mu(Q)^q}{(q-1)m} \geq \gamma,
\end{align}
then for all $x \in \R$ and all $\rho < 2^{-M}$, $\mu(B(x,\rho)) \leq 2\rho^{\left(1- 1/q \right) \gamma}$.
\end{lemma}

\begin{proof}
Note that the inequality in \eqref{eqn_lower_bound_on_finite_lq_dim} rearranges to
\[\sum_{Q \in \dyad_m} \mu(Q)^q \leq 2^{-m(q-1)\gamma}.\]
Thus, for all $Q \in \dyad_m$,
\[\mu(Q)^q \leq \sum_{Q \in \dyad_m} \mu(Q)^q \leq 2^{-m(q-1)\gamma}.\]
This gives the desired inequality for those intervals that are elements of the partition $\dyad_m$ for $m \geq M$.  Any interval of length $2^{-(m+1)} \leq \rho < 2^{-m}$ is covered by at most two elements of the partition $\dyad_{m+1}$, giving the result.
\end{proof}

We are now in a position to deduce \cref{thm_uniform_frost_exp_for_projections} from \cref{theorem_shmerkin_lq_projection} and \cref{lemma_frostman_exponent_from_lq_dim}.

\begin{proof}[Proof of \cref{thm_uniform_frost_exp_for_projections}]
Let $I \subseteq \R \setminus \{0\}$ be compact and $0 < \gamma < \gamma' < \min(\dim \mu, 1)$.  Let $q > 1$ be large enough so that $(1-1/q)\gamma' > \gamma$.  It follows from \cref{theorem_shmerkin_lq_projection} that there exists $M \in \N$ such that for all $t \in I$ and all $m \geq M$,
\[ \frac{-\log \sum_{Q \in \dyad_m} \Pi_t \mu (Q)^q }{ (q-1)m} \geq \gamma'.\]
Let $0 < \rho_0 < 2^{-M}$ be small enough so that $2 \rho_0^{(1-1/q)\gamma'} < \rho_0^\gamma$. It follows from \cref{lemma_frostman_exponent_from_lq_dim} that for all $\rho < \rho_0$, all $t \in I$, and all $x \in \R$,
\[\Pi_t \mu \big( B(x, \rho) \big) \leq 2 \rho^{(1-1/q)\gamma'}.\]
Since the $\Pi_t \mu$ mass of any ball is at most 1, by setting $c = \rho_0^{-\gamma}$, the conclusion of \cref{thm_uniform_frost_exp_for_projections} holds for all $\rho \in [0,1]$.
\end{proof}

\subsection{Geometric transversality in the reals}
\label{sec_main_continuous_intersection_results}

Here we employ \cref{thm_uniform_frost_exp_for_projections} to deduce upper bounds on the packing number of intersections of multiplicatively invariant sets.  
The idea in the proof below is borrowed from \cite[Lemma 1.8]{shmerkin}.

\begin{theorem}
\label{thm_size_of_support_on_fibers_for_inv_sets}
Let $r$ and $s$ be multiplicatively independent positive integers, and let $X, Y \subseteq [0,1]$ be $\times r$- and $\times s$-invariant sets, respectively. Define $\ogamma = \max(0, \dim X + \dim Y -1 )$. For all compact $I \subseteq \R \setminus \{0\}$ and $\eps > 0$,
\[\lim_{\rho \to 0^+} \ \sup_{\substack{t \in I \\ x \in \R}} \ \frac{ \NC \Big( (X \times Y) \cap \Pi_t^{-1}\big(B(x, \rho)\big)\,, \,\rho \Big) } {\rho^{-(\ogamma + \eps)} } = 0.\]
\end{theorem}

\begin{proof}
Let $I \subseteq \R \setminus \{0\}$ be compact and $\eps > 0$.  According to \cref{lem_embedding_into_missing_digit_set}, we can embed $X$ and $Y$ into restricted digit Cantor sets of slightly higher dimension.  Thus, there exists a product of restricted digit Cantor measures $\mu$ of dimension $\dim \mu < \dimh X + \dimh Y + \eps/4$ such that $X \times Y \subseteq \supp \mu$.

From \cref{thm_uniform_frost_exp_for_projections}, we have that there exists $\rho_0 > 0$ such that for all $\rho < \rho_0$, all $t \in I$, and all $x \in \R$,
\[\Pi_t \mu( B(x, 2\rho)) \leq \rho^{\min(\dim \mu,1) - \eps / 4}.\]

Let $\rho < \rho_0$, $t \in I$, and $x \in \R$.  
By \eqref{lemma_dimension_of_cantor_sets_and_measures} and the fact that $\rho_0$ is sufficiently small, every ball of radius $\rho$ centered at a point of $\supp \mu$ has $\mu$-mass greater than $\rho^{\dim \mu + \eps / 4}$. 
Therefore,
\[\NC \big( (\supp \mu) \cap \Pi_t^{-1}(B(x, \rho)), 2\rho \big) \cdot \rho^{\dim \mu + \eps / 4} \leq \mu \big( \Pi_t^{-1}(B(x, 2\rho)) \big) \leq \rho^{\min(\dim \mu,1) - \eps / 4}.\]
It follows now from the fact that $X \times Y \subseteq \supp \mu$ and \cref{lemma_metric_entropy_at_scaled_scale} that
\begin{align*}
\NC \big( (X \times Y) \cap \Pi_t^{-1}(B(x, \rho)), \rho \big) &\leq \NC \big( (\supp \mu) \cap \Pi_t^{-1}(B(x, \rho)), \rho \big) \\
& \ll \NC \big( (\supp \mu) \cap \Pi_t^{-1}(B(x, \rho)), 2\rho \big) \\
& \leq \rho^{\min(\dim \mu,1) - \dim \mu - \eps / 2} \\
&= \rho^{-(\max(0, \dim \mu - 1) + \eps / 2)} \\
&\leq \rho^{-(\ogamma + 3\eps/4)}.
\end{align*}
The limit in the conclusion of the theorem follows.
\end{proof}

The following corollary is formulated in a way that will make it convenient to apply in the integer setting.

\begin{corollary}
\label{theorem_main_continuous_intersections_v2}
Let $r$ and $s$ be multiplicatively independent positive integers, and let $X, Y \subseteq [0,1]$ be $\times r$- and $\times s$-invariant sets, respectively. Define $\ogamma = \max(0, \dim X + \dim Y -1 )$. For all compact $I \subseteq \R \setminus \{0\}$ and all $\eps > 0$,
\[\lim_{\rho \to 0^+} \ \sup_{\substack{\lambda, \eta \in I \\ \sigma, \tau \in \R}} \ \frac{ \NC \big( \big[\lambda X + \tau \big]_{\rho} \cap \big[\eta Y + \sigma \big]_{\rho}, \rho \big) } {\rho^{-(\ogamma + \eps)} } = 0.\]
\end{corollary}
Note that, taking fixed $\lambda, \eta, \sigma$ and $\tau=0$ this corollary recovers the Shermkin-Wu theorem encapsulated in \eqref{eqn_furstenberg_intersection}.
\begin{proof}
Let $I \subseteq \R \setminus \{0\}$ be compact and $\eps > 0$. Denote by $\pi_1: (x,y) \mapsto x$ the first coordinate projection.  The following facts are straightforward to verify:
\begin{itemize}
    \item $\lambda [X]_{\rho} = [\lambda X]_{|\lambda| \rho}$;
    \item $[X + \tau]_\rho = [X]_\rho + \tau$;
    \item $[X]_\rho \cap ([\eta Y]_{\rho} + \sigma) = \pi_1 \big( [X \times Y]_\rho \cap \Pi_{-\eta}^{-1} (\sigma) \big)$, using that $X \times Y$ is equipped with the $L^1$ metric;
    \item if $\varphi$ has Lipschitz constant $L$, then $[X]_\rho \cap \varphi^{-1}(\sigma) \subseteq [X \cap \varphi^{-1} B(\sigma, L \rho)]_\rho$.
\end{itemize}
Using these facts in order, we see that there exist $c_1, c_2 > 1$ depending only on $I$ such that
\begin{align}
\label{eqn_intersection_subset_inclusion}
\begin{aligned}
    [\lambda X + \tau]_\rho \cap [\eta Y + \sigma]_\rho & \subseteq \lambda \left( [X]_{c_1\rho} \cap \left (\left[ \frac{\eta}{\lambda} Y \right]_{c_1\rho} + \frac{\sigma - \tau}{\lambda} \right)\right) + \tau\\
    &=\lambda \pi_1 \left( [X \times Y]_{c_1\rho} \cap \Pi_{-\eta / \lambda}^{-1} \left( \frac{\sigma - \tau}{\lambda}\right) \right) + \tau \\
    &\subseteq \lambda \pi_1 \left( \left[ (X \times Y) \cap \Pi_{-\eta / \lambda}^{-1}  B \left( \frac{\sigma - \tau}{\lambda}, c_1c_2 \rho \right) \right]_{c_1 c_2 \rho} \right) + \tau.
\end{aligned}
\end{align}
We have need for four more easily-verified facts:
\begin{itemize}
    \item $\NC ( Z + \tau, \rho) = \NC (Z, \rho)$;
    \item $\NC ( \lambda Z, \rho) = \NC (Z, \rho / |\lambda|)$;
    \item $\NC( \pi_1(Z), \rho) \leq \NC( Z, \rho)$;
    \item $\NC ([Z]_\delta, \rho) \leq \delta / \rho \NC(Z, \rho)$.
\end{itemize}
Applying $\NC( \, \cdot \, , \rho)$ to both sides of \eqref{eqn_intersection_subset_inclusion} and using the preceding facts in order, we have that there exists $c_3 > 1$ depending only on $I$ such that
\[\NC \big( \big[\lambda X + \tau \big]_{\rho} \cap \big[\eta Y + \sigma \big]_{\rho}, \rho \big) \leq c_3 \NC \big( (X \times Y) \cap \Pi_{-\eta / \lambda}^{-1}  B \big( (\sigma - \tau) / \lambda, c_3 \rho \big), \rho / c_3 \big).\]
The conclusion of the corollary now follows from \cref{thm_size_of_support_on_fibers_for_inv_sets} by appealing to \cref{lemma_metric_entropy_at_scaled_scale}.
\end{proof}

\subsection{Additive transversality of sumsets in the reals}
\label{sec_main_continuous_sumset_results}

In this subsection, we use \cref{thm_size_of_support_on_fibers_for_inv_sets} to show that sets which are multiplicatively invariant with respect to multiplicatively independent bases are transverse in an additive combinatorial sense.  The core ideas here appear in \cite[Corollary 7.4]{shmerkin}, and we develop it in the context of the discrete Hausdorff dimension here.

The following lemma is a packing number analogue of the useful fact that if the fibers of a map $X \to Y$ between finite sets $X$ and $Y$ are uniformly bounded in cardinality, then the image of the map must be large.

\begin{lemma}
\label{lemma_bounded_metric_entropy_fibers}
Let $\varphi: \R^d \to \R^k$, $X \subseteq \R^d$ be bounded, and $\rho > 0$.  If $W > 0$ is such that for all $x \in \R^k$,
\begin{align*}
    \NC \Big( X \cap \varphi^{-1} \big(B (x, 2\rho)\big)\,,\, \rho \Big) \leq W,
\end{align*}
then $\NC \big(\varphi(X), \rho \big) \geq \NC \big( X, \rho \big) / W$.
\end{lemma}

\begin{proof}
Let $X'$ be a $\rho$-separated subset of $X$ of maximal cardinality so that $|X'| = \NC(X, \rho)$.  Since $\varphi(X')$ is covered by $\NC( \varphi(X'), \rho)$-many balls of radius $2\rho$, the set $X'$ is covered by $\NC( \varphi(X'), \rho)$-many pre-images of balls of radius $2\rho$ under $\varphi$.  Thus, there exists $x \in \R^k$ such that
\[\frac{|X'|}{\NC(\varphi(X'),\rho)} \leq \big| X' \cap \varphi^{-1} (B (x, 2\rho)) \big| \leq \NC \big( X \cap \varphi^{-1} (B (x, 2\rho)), \rho \big) \leq W.\]
It follows that
\[\frac{\NC(X, \rho)}{W} = \frac{|X'|}{W} \leq \NC(\varphi(X'),\rho) \leq \NC(\varphi(X),\rho),\]
as was to be shown.
\end{proof}

\begin{theorem}
\label{theorem_main_continuous_sumsets_v2}
Let $r$ and $s$ be multiplicatively independent positive integers, and let $X, Y \subseteq [0,1]$ be $\times r$- and $\times s$-invariant sets, respectively. Define $\ogamma = \max(0, \dimh X + \dimh Y - 1)$.  For all compact $I \subseteq \R \setminus \{0\}$, all $\eps > 0$, all $0 \leq \gamma \leq 1$, all sufficiently small $\rho > 0$ (depending on $X, Y, I$, $\eps$, and $\gamma$), all compact, non-empty $X' \subseteq X$, $Y' \subseteq Y$, and all $\lambda, \eta \in I$,
\begin{align}
\label{intro_lower_bound_on_metric_entropy_in_subsets_main_theorem_v2}
\NC \big( \lambda X' + \eta Y' , \rho \big) &\geq \frac{\NC \big( X' \times Y', \rho \big)}{\rho^{-(\ogamma + \eps)}}, \text{ and}\\
\label{intro_lower_bound_on_hausdorff_content_in_subsets_main_theorem_v2}
\HC_{\geq \rho}^{\gamma} \big( \lambda X' + \eta Y' \big) &\gg_{I,\gamma,\eps}  \HC_{\geq \rho}^{\gamma + \ogamma + \eps} \big( X' \times Y' \big).
\end{align}
\end{theorem}

\begin{proof}
It suffices by dilating, appealing to \cref{lemma_metric_entropy_at_scaled_scale}, and absorbing asymptotic constants into the $\rho^\eps$ term to prove the following: for all compact $I \subseteq \R \setminus \{0\}$, all $\eps > 0$, all $0 \leq \gamma \leq 1$, all sufficiently small $\rho_0 > 0$ (depending on $I$, $\eps$, and $\gamma$), all compact, non-empty $X' \subseteq X$, $Y' \subseteq Y$, all $t \in I$, and all $0 < \rho < \rho_0$,
\begin{align}
\label{intro_lower_bound_on_metric_entropy_in_subsets_main_theorem_v3}
\NC \big( \Pi_t( X' \times Y') , \rho \big) &\geq \frac{\NC \big( X' \times Y', \rho \big)}{\rho^{-(\ogamma + \eps)}}, \text{ and}\\
\label{intro_lower_bound_on_hausdorff_content_in_subsets_main_theorem_v3}
\HC_{\geq \rho}^{\gamma} \big( \Pi_t( X' \times Y') \big) &\gg \rho_0 \HC_{\geq \rho}^{\gamma + \ogamma + \eps} \big( X' \times Y' \big),
\end{align}
where, recall, $\Pi_t(x,y) = x + ty$.

Let $I \subseteq \R \setminus \{0\}$ be compact, $\eps > 0$, and $0 \leq \gamma \leq 1$.  It follows by \cref{thm_size_of_support_on_fibers_for_inv_sets} (with $\eps/2$ as $\eps$) that for all sufficiently small $\rho > 0$, all $t \in I$, and all $x \in \R$,
\begin{align*}
    \NC \big( (X \times Y) \cap \Pi_t^{-1}(B(x, 2\rho)), 2\rho \big)  \leq (2\rho)^{-(\ogamma + \eps / 2)}.
\end{align*}
Fix such a sufficiently small $0< \rho_0 < 1$, and ensure also that it is small enough so that $\rho_0^{-\eps / 2}$ is greater than the asymptotic constant appearing in \cref{lemma_metric_entropy_at_scaled_scale} (with $a = d = 2$). It follows that for all $0 < \rho < \rho_0$, all compact, non-empty $X' \subseteq X$, $Y' \subseteq Y$, all $t \in I$, and all $x \in \R$,
\begin{align}
\label{eqn_metric_entropy_small_in_tubes_for_xp_times_yp}
    \NC \big( (X' \times Y') \cap \Pi_t^{-1}(B(x, 2\rho)), \rho \big)  \leq \rho^{-(\ogamma + \eps)}.
\end{align}
Now \eqref{intro_lower_bound_on_metric_entropy_in_subsets_main_theorem_v3} follows immediately from \cref{lemma_bounded_metric_entropy_fibers} (with $X' \times Y'$ as $X$).

To show \eqref{intro_lower_bound_on_hausdorff_content_in_subsets_main_theorem_v3}, let $0 < \rho < \rho_0$ and $X' \subseteq X$, $Y' \subseteq Y$ be compact, non-empty. By \cref{lemma_finitary_frostman}, there exists a measure $\nu$ supported on $X' \times Y'$ with $\|\nu\| \geq \HC_{\geq \rho}^{\gamma + \ogamma + \eps} (X' \times Y')$ and such that for all $x \in \R^2$ and all $\delta \geq \rho$,
\begin{align}
\label{eqn_frostman_exp_on_nu_measure}
    \nu \big(B(x, \delta / 2) \big) \leq c_1 \delta^{\gamma + \ogamma + \eps},
\end{align}
where $c_1 > 1$ is an absolute constant.  Using the fact that $\supp \nu \subseteq X' \times Y' \subseteq X \times Y$, it follows from \eqref{eqn_metric_entropy_small_in_tubes_for_xp_times_yp} that for all $0 < \delta <  \rho_0$, all $t \in I$, and all $x \in \R$,
\begin{align}
\label{eqn_derived_from_metric_entropy_small_in_tubes_for_xp_times_yp}
    \NC \big( \supp \nu \cap \Pi_{t}^{-1}(B(x, \delta / 2)), \delta / 4 \big) \leq c_2 \delta^{-(\ogamma + \eps)},
\end{align}
where $c_2 > 1$ is an absolute constant.

The inequality in \eqref{eqn_derived_from_metric_entropy_small_in_tubes_for_xp_times_yp} implies that as long as $\delta < \rho_0$, the part of the support of $\nu$ contained in any tube $\Pi_{t}^{-1} (B(x, \delta / 2))$ can be covered by $c_2\delta^{-(\ogamma + \eps )}$ many balls of diameter $\delta$. The inequality in \eqref{eqn_frostman_exp_on_nu_measure} says that as long as $\delta \geq \rho$, each of those balls has $\nu$-measure at most $c_1\delta^{\gamma + \ogamma + \eps}$.  Therefore, we have that for all $\rho \leq \delta < \rho_0$, all $t \in I$, and all $x \in \R$,
\begin{align}
\label{eqn_measure_of_projected_balls}
    \nu \big( \Pi_{t}^{-1}(B(x, \delta / 2)) \big) \leq c_1\delta^{\gamma + \ogamma + \eps} c_2 \delta^{-(\ogamma + \eps)} = c_1 c_2 \delta^{\gamma}.
\end{align}

We aim now to deduce \eqref{intro_lower_bound_on_hausdorff_content_in_subsets_main_theorem_v3} from \eqref{eqn_measure_of_projected_balls}. Let $0 < \rho < \rho_0$, and let $\cup_i B_i$ be a cover of $\Pi_t (X' \times Y')$ by open balls of diameter at least $\rho$.  If some ball $B_i$ is such that $\diam B_i \geq \rho_0$, then $\sum_i (\diam B_i)^{\gamma} \geq \rho_0^\gamma \geq \rho_0$.  Otherwise, all balls in the cover have diameter less than $\rho_0$, and it follows then from \eqref{eqn_measure_of_projected_balls} that
\[c_1c_2 \sum_{i} ( \diam B_i )^{\gamma} \geq \|\Pi_t \nu\| = \|\nu\| \geq \HC_{\geq \rho}^{\gamma + \ogamma + \eps} (X' \times Y').\]
In either case, we have that
\[\sum_{i} ( \diam B_i )^{\gamma} \geq \min \big( \rho_0, (c_1c_2)^{-1} \HC_{\geq \rho}^{\gamma + \ogamma + \eps} (X' \times Y') \big) \geq \rho_0 (c_1c_2)^{-1} \HC_{\geq \rho}^{\gamma + \ogamma + \eps} (X' \times Y'),\]
where the second inequality follows from the fact that both quantities in the minimum are at most $1$.  Since the cover was arbitrary, we conclude the inequality in \eqref{intro_lower_bound_on_hausdorff_content_in_subsets_main_theorem_v3}.
\end{proof}

In the statement of the following corollary, it is useful to recall \cref{dimension_of_product_of_invariant_sets}: all of the notions of dimension for $X$, $Y$, and $X \times Y$ coincide, and $\dim(X \times Y) = \dim X + \dim Y$.

\begin{corollary}\label{cor_Hochman_Shmerkin_for_subsets}
Let $r$ and $s$ be multiplicatively independent positive integers, and let $X, Y \subseteq [0,1]$ be $\times r$- and $\times s$-invariant sets, respectively. For all $\dim \in \{\ldimdm, \udimdm, \dimh\}$, for all compact subsets $X' \subseteq X$ and $Y' \subseteq Y$, and for all $\lambda, \eta \in \R \setminus \{0\}$, 
\begin{itemize}
    \item if $\dim X + \dim Y \leq 1$, then
\begin{align}
\label{cor_ldimdh_for_arbitrary_subsets_sumsets_case_1}
    \dim \big(  \lambda X' + \eta Y'  \big) = \dim \big( X' \times Y' \big);
\end{align}
    \item if $\dim X + \dim Y > 1$, then
\begin{align}
\label{cor_ldimdh_for_arbitrary_subsets_sumsets_case_2}
    \dim \big(  \lambda X' + \eta Y'  \big) \geq \dim \big( X' \times Y' \big) - \dim \big(X \times Y \big) + 1.
\end{align}
\end{itemize}
\end{corollary}

Note that \cref{cor_Hochman_Shmerkin_for_subsets} extends the theorem of Hochman and Shmerkin encapsulated by \eqref{eqn_in_thm_HS_localentropy}.
Indeed, setting $X' = X$ and $Y' = Y$, it follows from \eqref{cor_ldimdh_for_arbitrary_subsets_sumsets_case_1} and \eqref{cor_ldimdh_for_arbitrary_subsets_sumsets_case_2} that $\dim \big( \lambda X + \eta Y  \big) \geq \min \big(1, \dim(X \times Y) \big)$.  
Using the fact that $(x,y) \mapsto \lambda x + \eta y$ is Lipschitz, the bounds in \cref{lemma_hausdorff_content_bound_under_lipschitz_map} immediately give the required upper bounds to yield equality in \eqref{eqn_in_thm_HS_localentropy}.

\begin{proof}

Define $\ogamma = \max(0, \dimh (X \times Y) - 1)$, and let $X' \subseteq X$ and $Y' \subseteq Y$. To show \eqref{cor_ldimdh_for_arbitrary_subsets_sumsets_case_1} and \eqref{cor_ldimdh_for_arbitrary_subsets_sumsets_case_2},
it suffices to show
\begin{align}
    \label{eqn_sufficient_to_show_in_corollary}
    \dim \big( \lambda X' + \eta Y' \big) \geq \dim \big( X' \times Y' \big) - \ogamma.
\end{align}
Indeed, this is the lower bound in \eqref{cor_ldimdh_for_arbitrary_subsets_sumsets_case_2} and the upper bound derived from \cref{lemma_hausdorff_content_bound_under_lipschitz_map} combined with this lower bound gives the desired equality in \eqref{cor_ldimdh_for_arbitrary_subsets_sumsets_case_1}.

Let $\dim \in \{\ldimm, \udimm, \dimh \}$ and $\lambda, \eta \in (0,\infty)$.
If $\dim(X' \times Y') \leq\ogamma$, the conclusion is immediate, so we can proceed under the assumption that $\dim(X' \times Y') > \ogamma$.

Let $\eps > 0$, and let $\gamma = \dim(X' \times Y') - \ogamma - 2\eps$. 
It follows from \cref{theorem_main_continuous_sumsets_v2} that there exists a small $\rho_0 > 0$ such that for all $0 < \rho < \rho_0$,
\begin{gather*}
    \frac{\NC \big( \lambda X' + \eta Y' , \rho \big)}{\rho^{-\gamma}} \geq \frac{\NC \big( X' \times Y', \rho \big)}{\rho^{-(\gamma+\ogamma + \eps)}},\\
    \HC_{\geq \rho}^{\gamma} \big( \lambda X' + \eta Y' \big) \geq \rho_0 \HC_{\geq \rho}^{\gamma + \ogamma + \eps} \big( X' \times Y' \big).
\end{gather*}
Consider the first inequality if $\dim$ is the Minkowski dimension and the second inequality if $\dim$ is the Hausdorff dimension.  Because $\gamma + \ogamma + \eps = \dim(A' \times B') - \eps$, the limit infimum (if $\dim$ is a lower dimension) or limit supremum (if $\dim$ is an upper dimension) as $N$ tends to infinity of the right hand side is positive. It follows that
\[\dim \big( \lambda X' + \eta Y' \big) \geq \dim \big( X' \times Y' \big) - \ogamma - \eps.\]
The inequality in \eqref{eqn_sufficient_to_show_in_corollary} now follow from the fact that $\eps >0$ was arbitrary, concluding the proof.
\end{proof}

\section{Discrete fractal geometry and multiplicatively invariant subsets of the integers}
\label{sec_discrete_fractal_geometry}
In this section, we introduce the notation and terminology involved in the study of fractal geometry in the positive integers and develop the basic results concerning multiplicatively invariant subsets.  To prove the results in this section and the transversality results in the next, we relate $\times r$-invariant subsets of the integers to symbolic subshifts on $r$ symbols and to $\times r$-invariant subsets of $[0,1]$.

\subsection{Notions of dimension for subsets of integers}\label{sec_discrete_dimensions}

To measure the size of subsets of $\Nz$, we will make use of the (upper and lower) mass dimension and the (upper and lower) discrete Hausdorff dimension, which were introduced in \cref{sec_intro_integers}, but which we recall for a more detailed discussion in this section. The upper and lower mass dimensions and the upper Hausdorff dimension are also treated systematically in \cite{barlow_taylor_92}; we will state the properties we require from these quantities with the aim of making this presentation self-contained. These dimensions join a bevy of other natural notions of dimension for subsets of the integers, integer lattices, and more general discrete sets; see \cite{naudts1,naudts2,barlow_taylor_1989,iosevich_rudnev_tuero_2014,limamoreira}.

\begin{definition}
\label{def_mass_and_h_dimensions}
Let $A \subseteq \Nz^d$ be non-empty.
\begin{itemize}
    \item The \emph{lower mass dimension} of $A$ is
\[\ldimdm A \,=\, \liminf_{N \to \infty} \frac{\log \big| A \cap [0,N)^d \big|}{\log N}.\]
The \emph{upper mass dimension}, $\udimdm A$, is defined analogously with a limit supremum in place of the limit infimum.  If $\ldimdm A = \udimdm A$, then this value is the \emph{mass dimension} of $A$, $\dimdm A$.

\item The \emph{lower discrete Hausdorff dimension} of $A$ is
\[\ldimdh A = \sup \left\{\gamma \geq 0 \ \middle| \ \liminf_{N \to \infty} \frac{\HC_{\geq 1}^\gamma \big(A \cap [0,N)^d \big)}{N^\gamma} > 0 \right\}.\]
The \emph{upper discrete Hausdorff dimension}, $\udimdh A$, is defined analogously with a limit supremum in place of the limit infimum.  If $\ldimdh A = \udimdh A$, then this value is the \emph{discrete Hausdorff dimension} of $A$, $\dimdh A$.
\end{itemize}
\end{definition}

As the notation suggests, the mass and discrete Hausdorff dimensions are defined in analogy to the Minkowski and Hausdorff dimensions, respectively.  The analogy becomes clearer on noting that
\begin{align}
    \label{eqn_mass-dim_to_Mink-dim}
    \big|A \cap [0,N)^d \big| &= \NC \left(  \frac{A\cap [0,N)^d}{N},\,  N^{-1} \right), \\
    \label{eqn_disc_to_cont_Haus-dim}
    \frac{\HC_{\geq 1}^\gamma \big(A \cap [0,N)^d \big)}{N^\gamma} &= \HC_{\geq N^{-1}}^\gamma \left( \frac{A\cap [0,N)^d}{N} \right),
\end{align}
so that the mass and discrete Hausdorff dimensions are capturing, in some sense, the Minkowski and Hausdorff dimensions of the sequence of sets $N \mapsto A/ N$ in the unit cube.

As a word of caution, note that our terminology does not match exactly with the terminology used in \cite{barlow_taylor_92}. What we call the upper discrete Hausdorff dimension is called $\dim_L$ in \cite{barlow_taylor_92} (see Lemma 2.3 in that paper), while the discrete Hausdorff dimension defined in that work does not appear in our work. Our choice of terminology is motivated by the connections drawn in our work between the discrete and continuous notions of dimension.

\begin{lemma}
\label{lem_discrete_dim_properties}
Let $A, B\subseteq \Nz^d$, $\lambda > 0$, and $\sigma \in \R^d$.
\begin{enumerate}[label=(\Roman*)]
    \item\label{itm_dim_I}
    For all $\dim \in \{\ldimdm, \udimdm, \ldimdh, \udimdh\}$, $\dim A \in [0,d]$.
    \item\label{itm_dim_II}
    For all $\dim \in \{\ldimdm, \udimdm, \ldimdh, \udimdh\}$, $\dim A = \dim \big( \lfloor \lambda A + \sigma \rfloor \big)$, where $\lfloor \lambda A + \sigma \rfloor=\{\lfloor \lambda n + \sigma \rfloor \ | \ n\in A \}$.
    \item\label{itm_dim_III} For all $\dim \in \{\udimdm, \udimdh\}$, $\dim (A \cup B) = \max \big( \dim A, \dim B \big)$.
    \item\label{itm_dim_III_point_5} For all $r \in \N$, $r \geq 2$,
    \begin{align*}
    \ldimdm A = \liminf_{N \to \infty} \frac{\log |A \cap [0,r^N)^d |}{N\log r},
    \end{align*}
    and the analogous statement with $\udimdm$ in place of $\ldimdm$ and limit supremum in place of limit infimum holds.
    \item\label{itm_dim_IV} For all $r \in \N$, $r \geq 2$,
    \begin{align*}
    \ldimdh A = \sup \left\{ \gamma \geq 0 \ \middle| \  \liminf_{N \to \infty} \frac{\HC_{\geq 1}^\gamma \big(A \cap [0,r^N)^d \big)}{r^{\gamma N}} > 0 \right\},
    \end{align*}
    and the analogous statement with $\udimdh$ in place of $\ldimdh$ and limit supremum in place of limit infimum holds.
\end{enumerate}
\end{lemma}

Note that the sets in \cref{discrete_examples} \ref{ex_discrete_set_two} below show that the statement in \ref{itm_dim_III} does not hold for the lower mass and lower discrete Hausdorff dimensions.

\begin{proof}
The statements in \ref{itm_dim_I} through \ref{itm_dim_III_point_5} follow from straightforward calculations which are left to the reader.  

Both of the statements in \ref{itm_dim_IV} follow from \eqref{eqn_disc_to_cont_Haus-dim} and the fact that for all $\gamma \geq 0$ and all $r^K \leq N \leq r^{K+1}$,
\[\frac{\HC_{\geq 1}^\gamma \big(A \cap [0,r^K)^d \big)}{r^{K\gamma}} \leq r^\gamma \frac{\HC_{\geq 1}^\gamma \big(A \cap [0,N)^d \big)}{N^\gamma} \leq r^{2 \gamma} \frac{ \HC_{\geq 1}^\gamma \big(A \cap [0,r^{K+1})^d \big)}{r^{(K+1)\gamma}}.\]
Indeed, this shows that the limit infimum (resp. limit supremum) of the sequence $N \mapsto \HC_{\geq 1}^\gamma \big(A \cap [0,r^N) \big) / r^{N\gamma}$ is non-zero if and only if the limit infimum (resp. limit supremum) of the sequence $N \mapsto \HC_{\geq 1}^\gamma \big(A \cap [0,N) \big) / N^\gamma$ is non-zero.
\end{proof}

\begin{lemma}
\label{lemma_dim_comparisons}
For all $A \subseteq \Nz^d$,
\begin{gather*}
    \ldimdh A \leq \ldimdm A \leq \udimdm A,\\
    \ldimdh A \leq \udimdh A \leq \udimdm A,
\end{gather*}
and no other comparisons are possible in general.
\end{lemma}

\begin{proof}
It is immediate from the definitions that $\ldimdm A \leq \udimdm A$ and $\ldimdh A \leq \udimdh A$, and the set in \cref{discrete_examples} \ref{ex_discrete_set_one} below shows that neither of these inequalities are, in general, equalities.

To see that $\ldimdh A \leq \ldimdm A$ and that $\udimdh A \leq \udimdm A$, note that by covering $A \cap [0,N)^d$ by $|A \cap [0,N)^d|$ many balls of diameter $1$ it follows that
\[\frac{\HC_{\geq 1}^\gamma \big(A \cap [0,N)^d \big)}{N^\gamma} \leq \frac{|A \cap [0,N)^d|}{N^\gamma}.\]
If $\gamma > \ldimdm A$ (resp. $\gamma > \udimdm A$), then the limit infimum (resp. limit supremum) of the right hand side is zero, implying that $\gamma \geq \ldimdh A$ (resp. $\gamma \geq \udimdh A$). It follows that $\ldimdh A \leq \ldimdm A$ and $\udimdh A \leq \udimdm A$.  The set in \cref{discrete_examples} \ref{ex_discrete_set_three} below shows that neither of these inequalities are, in general, equalities.

To see that no other comparisons are possible, it suffices to show that there can in general be no comparison between $\udimdh$ and $\ldimdm$.  This is demonstrated by the sets in \cref{discrete_examples} \ref{ex_discrete_set_one} and \ref{ex_discrete_set_three} below.
\end{proof}

The following examples are meant to illustrate the extent to which the mass and discrete Hausdorff dimensions relate for subsets of $\N_0$. These examples do not feature the type of structures that we are concerned with in this work, so we leave some of the details to the reader.

\begin{examples}\label{discrete_examples} \leavevmode
\begin{enumerate}[label=(\roman*)]
\item \label{ex_discrete_set_one} Let $(x_n)_{n=0}^\infty \subseteq \Nz$ be any sequence which satisfies $\lim_{n \to \infty}\log (x_{n+1} - x_{n} ) / \log x_{n+1} = 1$, and define
\[A \defeq \{0\} \cup \bigcup_{n = 0}^\infty \{x_{2n}, x_{2n}+1,\ldots, x_{2n+1} \}.\]
It is easy to check that $\ldimdm A = \ldimdh A = 0$ and that $\udimdm A = \udimdh A = 1$.

\item \label{ex_discrete_set_two} Let $A$ be the set from \ref{ex_discrete_set_one}. Put $B = \{0\} \cup \big(\Nz \setminus A \big)$. Then $\ldimdm B = \ldimdh B = 0$ while $\udimdm B = \udimdh B = 1$, and $A+B = A\cup B = \Nz$.

\item \label{ex_discrete_set_three} Define
\[A = \{0, \ldots, 16\} \cup \bigcup_{n=2}^\infty \big\{ 2^n, \ldots, 2^n + \lfloor 2^{n-n/\log n} \rfloor \big\}.\]
It is quick to check that the mass dimension of $A$ exists and $\dimdm A = 1$. On the other hand, by covering $A$ with the intervals in its definition, it can be shown that the discrete Hausdorff dimension of $A$ exists and $\dimdh A = 0$.
\end{enumerate}
\end{examples}

We conclude this section by proving some basic upper and lower bounds on the dimension of product sets.

\begin{lemma}
\label{lemma_dim_bounds_on_product_sets}
For all non-empty $A_1, \ldots, A_d \subseteq \N_0$,
\begin{align}
\label{eqn_upper_bound_on_mass_dim_of_product}
    \udimdm \big( A_1 \times \cdots \times A_d \big) &\leq \sum_{i=1}^d \udimdm A_i,\\
\label{eqn_lower_bound_on_h_dim_of_product}
\ldimdh \big( A_1 \times \cdots \times A_d \big) &\geq \sum_{i=1}^d \ldimdh A_i.
\end{align}
In particular, if $\ldimdh A_i = \udimdm A_i$ for each $i \in \{1, \ldots, d\}$, then for all $\dim \in \{\ldimdm,\allowbreak \udimdm,\allowbreak \ldimdh, \udimdh\}$, 
\[\dim \big( A_1 \times \cdots \times A_d \big) = \dim A_1 + \cdots + \dim A_d.\]
\end{lemma}

\begin{proof}
The inequality in \eqref{eqn_upper_bound_on_mass_dim_of_product} is immediate from the definition of upper mass dimension.  To prove the inequality in \eqref{eqn_lower_bound_on_h_dim_of_product}, define $\gamma_i = \ldimdh A_i$ and $\gamma = \sum_{i=1}^d \gamma_i$.  Define $A = A_1 \times \cdots \times A_d$.

Let $\eps > 0$ and $N \in \N$.  It follows by \cref{lemma_finitary_frostman} that there exists a measure $\mu_i$ supported on $(A_i / N) \cap [0,1)$ with $\| \mu_i\| \geq \HC_{\geq N^{-1}}^{\gamma_i-\eps} \big( (A_i / N) \cap [0,1) \big)$ and such that for all balls $B$ of diameter at least $N^{-1}$, $\mu_i(B) \leq c \diam(B)^{\gamma_i - \eps}$.

Consider the product measure $\mu = \mu_1 \times \cdots \times \mu_d$; it is supported on the set $A$ and has the property that for all balls $B$ of diameter at least $N^{-1}$, $\mu(B) \leq c^d \diam(B)^{\gamma - d \eps}$.  It follows by \cref{lem_consequence_of_discrete_frostman} and \eqref{eqn_disc_to_cont_Haus-dim} that
\begin{align*}
    \frac{\HC_{\geq 1}^{\gamma- d\eps} ( A \cap [0,N)^d )}{N^{\gamma - d\eps}} &= \HC_{\geq N^{-1}}^{\gamma - d \eps} \left( \frac{A}N \cap [0,1)^d \right) \\
    &\geq c^{-d} \prod_{i=1}^d \HC_{\geq N^{-1}}^{\gamma_i-\eps} \big( (A_i / N) \cap [0,1) \big)\\
    &= c^{-d} \prod_{i=1}^d \frac{\HC_{\geq 1}^{\gamma_i-\eps} (A_i \cap [0,N))}{N^{\gamma_i - \eps}}.
\end{align*}
By the definition of the lower discrete Hausdorff dimension, the limit infimum as $N$ tends to infinity of the right hand side of the previous inequality is positive, whereby $\ldimdh A \geq \gamma - d \eps$.  The conclusion of the lemma follows since $\eps > 0$ was arbitrary.
\end{proof}

\subsection{Dimension regularity of multiplicatively invariant sets}
\label{sec_first_dim_results}

In this section, we prove that the mass and discrete Hausdorff dimensions of a multiplicatively invariant set (cf. \cref{def_xr_invariant_integers}) exist and coincide.
This is accomplished by adapting an argument of Furstenberg \cite[Prop. III.1]{furstenbergdisjointness} from the continuous setting.

\begin{proposition}
\label{lemma_dimensions_coincide_for_invariant_sets}
If $A \subseteq \N_0$ is multiplicatively invariant (see \cref{def_xr_invariant_integers}), then
\[\ldimdh A = \udimdh A = \ldimdm A = \udimdm A.\]
In particular, the mass and discrete Hausdorff dimensions of $A$ exist and coincide.
\end{proposition}

Before the proof, we introduce some notation that will be useful throughout this section and the following ones.  
Fix $r \in \N$, $r \geq 2$, and denote by $\Lambda_r$ the \emph{alphabet} $\{0, \ldots, r-1\}$.  An element $w \in \Lambda_r^\ell$ is a \emph{word} of \emph{length} $|w| = \ell$. The set of all \emph{finite words} is $\Lambda_r^* = \cup_{\ell = 0}^\infty \Lambda_r^\ell$, and the set of all \emph{infinite words} is $\Lambda_r^{\Nz}$.  
The \emph{empty word} is the sole element of the set $\Lambda_r^0$.
The \emph{concatenation} of the word $w \in \Lambda_r^\ell$ with the word $v \in \Lambda_r^k$ is denoted by juxtaposition: the word $wv$ is an element of $\Lambda_r^{\ell+k}$.  
We write $w^k$ for the word $w$ concatenated with itself $k$ many times.  
Finally, we write $w = w_0 \cdots w_{\ell-1}$ to indicate that the letters of $w$ are $w_0, \ldots, w_{\ell-1} \in \Lambda_r$, in that order.  

For $w = w_0 \cdots w_{\ell-1} \in \Lambda_r^{\ell}$, define an element in $\Nz$ by
\[(w)_r \defeq w_0 r^{\ell-1} + w_1 r^{\ell-2} + \cdots + w_{\ell-2} r + w_{\ell-1}.\]
The function $( \, \cdot \, )_r: \Lambda_r^* \to \Nz$ serves as the primary link between subsets of non-negative integers and words.  In the following subsection, we will use $( \, \cdot \, )_r$ to connect $\times r$-invariant subsets of $\Nz$ with symbolic subshifts.  Note that $( \, \cdot \, )_r$ is surjective, and is injective when restricted to $\Lambda_r^\ell$ for some $\ell \in \Nz$.

As a final ingredient before the proof of \cref{lemma_dimensions_coincide_for_invariant_sets}, we give an equivalent characterization of the lower discrete Hausdorff dimension, $\ldimdh$.

\begin{lemma}
\label{remark_recharacterization_of_ldim}
For all $A \subseteq \Nz$,
\begin{align*}
\ldimdh A = \sup \left\{ \gamma \geq 0 \ \middle| \  \liminf_{N \to \infty} \frac{\HC_{\geq 1}^{\gamma,*} \big(A \cap [0,r^N) \big)}{r^{N\gamma}} > 0 \right\},
\end{align*}
where $\HC_{\geq 1}^{\gamma,*} (X)$ is defined to be
\[\min \left\{ \sum_{i \in I} r^{d_i \gamma} \ \middle| \ X \subseteq \bigcup_{i \in I} \bigg( (w^{(i)}0^{d_i})_r + [0,r^{d_i}) \bigg), \ w^{(i)} \in \Lambda_r^*, \ d_i \in \Nz \right\}.\]
\end{lemma}

\begin{proof}
It suffices to show that for all finite $X \subseteq \Nz$, $\HC_{\geq 1}^{\gamma,*} (X) \asymp \HC_{\geq 1}^{\gamma} (X)$, and then appeal to \cref{lem_discrete_dim_properties} \ref{itm_dim_IV}.  That $\HC_{\geq 1}^{\gamma,*} (X) \geq \HC_{\geq 1}^{\gamma} (X)$ follows immediately from the definitions.  To show that $\HC_{\geq 1}^{\gamma,*} (X) \ll \HC_{\geq 1}^{\gamma} (X)$, use the fact that any interval in $\Nz$ of length $\ell$ can be covered by at most two intervals of the form $(w0^{d})_r + [0,r^{d})$, where $d = \lceil \log_r \ell \rceil$.
\end{proof}

\begin{proof}[Proof of \cref{lemma_dimensions_coincide_for_invariant_sets}]
Suppose $A \subseteq \Nz$ is $\times r$-invariant. 
Let $\gamma > \ldimdh A$. We will show that $\limsup_{M \to \infty} |A \cap [0,r^M)| / r^{M \gamma} < \infty$, from which it follows that $\udimdm A \leq \gamma$.  Since $\gamma > \ldimdh A$ is arbitrary, it will follow that $\udimdm A \leq \ldimdh A$. It will follow then from \cref{lemma_dim_comparisons} that $\ldimdh A=\udimdh A=\ldimdm A=\udimdm A$, which will conclude the proof of the lemma.

According to \cref{remark_recharacterization_of_ldim}, there exists $N \in \N$ and a collection of intervals $B_i = (w^{(i)} 0^{d_i})_r + [0,r^{d_i})$, $i \in I$, that cover $A \cap [0,r^N)$ and for which $\sum_{i \in I} r^{(d_i-N) \gamma} < 1$.
By prepending zeros onto each $w^{(i)}$, we may assume that $|w^{(i)}| + d_i = N$.  Note that for all $w \in \Lambda_r^{N}$, $(w)_r \in B_i$ if and only if $w = w^{(i)}w'$ for some $w' \in \Lambda_r^{d_i}$.

Let $M \in \N$, $M > N$, and let $n \in A \cap [0, r^M)$.  
Write $n = (w)_r$, where $w \in \Lambda_r^M$ (so that $w$ may have leading zeroes).  
Since $A$ is $\Phi_r$-invariant, $\Phi_r^{M-N}(n) = (w_1 \cdots w_N)_r \in A \cap [0, r^N)$. 
Since $A \cap [0,r^N) \subseteq \cup_i B_i$, there exists $i_1 \in I$ such that $(w_1 \cdots w_N)_r \in B_{i_1}$.  
It follows that $w = w^{(i_1)}w'$ for some $w' \in \Lambda_r^{M-d_{i_1}}$.  
Since $A$ is $\Psi_r$-invariant, applying $\Psi_r$ to $n$ between $0$ and $|w^{(i_1)}|$-many times (depending on how many initial zeroes there are in the word $w^{(i_1)}$) to $n$, we see that $(w')_r \in A$.  
Repeating the argument with $(w')_r \in A$, there exists $i_2 \in I$ such that $w' = w^{(i_2)}w''$ for some $w'' \in \Lambda_r^{M-d_{i_1}-d_{i_2}}$.  
Repeating further, we see that there exist $i_1, \ldots, i_k \in I$ such that $w = w^{(i_1)}\cdots w^{(i_k)}v$, where $v \in \Lambda_r^{< N}$.

Using the factorization of words $w \in \Lambda_r^M$ for which $(w)_r \in A$ described in the previous paragraph and recalling that $-|w^{(i)}| = d_i - N$, we see that
\begin{align*}
    \frac{\big| A \cap [0,2^M)\big|}{r^{M\gamma}} &= \sum_{w \in \Lambda_r^M \; : \; (w)_r \in A} r^{-|w|\gamma} \\
    &\leq \left( \sum_{v \in \Lambda_r^{< N}} r^{-|v|\gamma}\right) \left( 1 + \sum_{i_1 \in I} r^{(d_{i_1} - N) \gamma} + \sum_{i_1, i_2 \in I} r^{(d_{i_1} - N + d_{i_2} - N) \gamma}  + \cdots \right)\\
    &= \left( \sum_{v \in \Lambda_r^{< N}} r^{-|v|\gamma}\right) \left( 1-\sum_{i \in I} r^{(d_i - N) \gamma}\right)^{-1}.
\end{align*}
Since the final quantity is finite and independent of $M$, and since $M > N$ was arbitrary, it follows that $\limsup_{M \to \infty} |A \cap [0,r^M)| / r^{M \gamma} < \infty$, as was to be shown.
\end{proof}

\begin{corollary}
\label{cor_dim_of_products_of_invariant_sets}
If $A_1, \ldots, A_d \subseteq \Nz$ are multiplicatively invariant (with respect to any bases), then for all $\dim \in \{\ldimdm, \udimdm, \ldimdh, \udimdh\}$,
\[\dim \big( A_1 \times \cdots \times A_d \big) = \dim A_1 + \cdots + \dim A_d.\]
\end{corollary}

\begin{proof}
This follows immediately by combining \cref{lemma_dim_bounds_on_product_sets} and \cref{lemma_dimensions_coincide_for_invariant_sets}.
\end{proof}

\subsection{Connections to symbolic dynamics}
\label{sec_connection_symbolic_dynamcis}

Throughout this subsection, we use $\sigma$ to denote the left-shift on $\Lambda_r^{\Nz}$, which is defined by
\[
\sigma:  (w_n)_{n\in\Nz}\mapsto (w_{n+1})_{n\in\Nz}.
\]
We endow $\Lambda_r$ with the discrete topology and $\Lambda_r^{\Nz}$ with the product (or Tychonoff) topology. 
In the context of symbolic dynamics, any closed subset of $\Lambda_r^{\Nz}$ satisfying $\sigma(\Sigma) \subseteq \Sigma$ is called a \define{subshift}.
The \define{language set} associated to a subshift $\Sigma$ is the set of all the finite words, including the empty word, appearing in the elements of $\Sigma$, i.e.,
\[
\mathcal{L}(\Sigma)=\big\{w_0 \cdots w_{\ell - 1} \ \big| \ w = w_0 w_1 \cdots \in \Sigma,~\ell\in\Nz \big\}.
\]
The language set of any subshift can be naturally embedded into the integers in two ways, giving rise to the following definition.

\begin{definition}
\label{def_language-set}
The \define{$r$-language sets} associated to a subshift $\Sigma\subseteq \Lambda_r^{\Nz}$ are the sets $A_\Sigma,B_\Sigma\subseteq\Nz$ defined by
\begin{align*}
A_{\Sigma} &\,=\, \Big\{ (w_0 \cdots w_{\ell - 1})_r = w_0 r^{\ell-1} + w_1 r^{\ell-2} + \cdots + w_{\ell-2} r + w_{\ell-1} \ \big| \ w_0 \cdots w_{\ell - 1} \in \mathcal{L}(\Sigma)\Big\},\\
B_{\Sigma} &\,=\, \Big\{ (w_{\ell - 1} \cdots w_0)_r = w_{\ell-1} r^{\ell-1} + w_{\ell-2} r^{\ell-2} + \cdots + w_{1} r + w_{0} \ \big| \ w_0 \cdots w_{\ell - 1} \in \mathcal{L}(\Sigma)\Big\},
\end{align*}
where $(w)_r=0$ when $w$ is the empty word.
\end{definition}
The following proposition uses $r$-language sets to relate $\times r$-invariant sets with subshifts of $\Lambda_r^{\Nz}$.  It is a generalization of some of the results in \cite[Section 3]{limamoreira}, where subsets of integers arising from shifts of finite type are defined and studied.

\begin{proposition}
\label{prop_correspondence_integers_subshifts}
The $r$-language sets $A_\Sigma, B_\Sigma \subseteq \Nz$ corresponding to any non-empty subshift $\Sigma \subseteq \Lambda_r^{\Nz}$ are $\times r$-invariant sets, and have discrete mass and Hausdorff dimensions equal to the normalized topological entropy of the symbolic subshift $(\Sigma,\sigma)$, i.e.,
\begin{align}
\label{eqn_in_prop_correspondence_integers_subshifts}
    \dimdh A_\Sigma = \dimdm A_\Sigma =\dimdh B_\Sigma = \dimdm B_\Sigma = \frac{h_{\mathrm{top}}(\Sigma,\sigma)}{\log r}.
\end{align}
Moreover, for any $\times r$-invariant set $B\subseteq\Nz$, there exists a subshift $\Sigma\subseteq \Lambda_r^{\Nz}$ such that $B$ coincides with the $r$-language set $B_\Sigma$ associated to $\Sigma$.
\end{proposition}

\begin{remark}
    The second part of \cref{prop_correspondence_integers_subshifts} does not hold with $A_\Sigma$ in place of $B_\Sigma$ in general. As an example, let $k \in \N$ and put $B:=\{0,1,2,\ldots,k\}\subseteq\Nz$.
    It is clear that for any $r \geq 2$, the set $B$ is a $\times r$-invariant set.  However, note that for any subshift $\Sigma$, the set $A_\Sigma$ is either $\{0\}$ or infinite, so we can not have that $A_\Sigma=B$.
\end{remark}

\cref{prop_correspondence_integers_subshifts} shows that $r$-language sets 1) provide us with a natural way of producing examples of $\times r$-invariant subsets of the non-negative integers; and 2) allow us to employ tools and techniques from symbolic dynamics to study $\times r$-invariant sets.  Before the proof, we give some examples of $\times r$-invariant subsets of $\Nz$ arising this way.

\begin{examples}\label{examples_of_times_r_invariant_sets} \leavevmode
In each of the examples below, the language of the subshift $\Sigma$ used to generate the $r$-language set $A_\Sigma$ is invariant under reversing words.  Therefore, in each example, $B_\Sigma = A_\Sigma$.
\begin{itemize}
\item The classical \define{golden mean shift} is the subshift of $\{0,1\}^{\Nz}$ consisting of all binary sequences with no two consecutive $1$'s. This leads to a natural example of a $\times 2$-invariant set $A_{\mathrm{golden}}\subseteq\Nz$ consisting of all integers whose binary digit expansion does not contain two consecutive $1$'s. Since the topological entropy of the golden mean shift is known the equal $\log((1+\sqrt{5})/2)$ (cf.\ \cite[Example 4.1.4]{lind_marcus_book}), it follows from \cref{prop_correspondence_integers_subshifts} that the dimension of $A_{\mathrm{golden}}$ equals $\log((1+\sqrt{5})/2)/\log 2$. Integer sets corresponding to the broader class of subshifts of finite type were also considered by Lima and Moreira in \cite{limamoreira}.

\item The \define{even shift} is the subshift of $\{0,1\}^{\Nz}$ consisting of all binary sequences so that between any two $1$'s there are an even number of $0$'s. The corresponding $\times 2$-invariant set $A_{\mathrm{even}}\subseteq\Nz$ consists of all integers whose binary digit expansion has an even number of $0$'s between any two $1$'s. Since the topological entropy of the golden mean shift coincides with the topological entropy of the even shift (cf.\ \cite[Example 4.1.6]{lind_marcus_book}), we conclude that $A_{\mathrm{even}}$ and $A_{\mathrm{golden}}$ have the same dimension. 

\item The \define{prime gap shift} is the subshift of $\{0,1\}^{\Nz}$ consisting of all binary sequences such that there is a prime number of $0$'s between any two $1$'s.
This corresponds to the $\times 2$-invariant set $A_{\mathrm{prime}}\subseteq\Nz$ of all those numbers written in binary in which there is a prime number of $0$'s between any two $1$'s.
For example, the first $17$ elements of $A_{\mathrm{prime}}$ are: $0, 1, 2, 4, 8, 9, 16, 17, 18, 32, 34, 36, 64, 65, 68, 72, 73$.
The entropy of the prime gap shift is approximately $0.30293$, (cf.\ \cite[Exercise 4.3.7]{lind_marcus_book})
which implies that the dimension of $A_{\mathrm{prime}}$ is approximately $0.437$.
\end{itemize}
\end{examples}

\begin{proof}[Proof of \cref{prop_correspondence_integers_subshifts}]
Let $\Sigma \subseteq \Lambda_r^{\Nz}$ be a subshift, and let $A_\Sigma$ and $B_\Sigma$ be the associated $r$-language sets.
We begin with the proof that the set $A_\Sigma$ is $\times r$-invariant.
Note first that $0 \in A_\Sigma$ because the empty word is in $\mathcal{L}(\Sigma)$. 
Let $n \in A_\Sigma$, $n \geq 1$.  Because $\Sigma$ is shift-invariant, there exists a word $w = w_0 \cdots w_{\ell - 1}\in \mathcal{L}(\Sigma)$ such that $w_0 \neq 0$ and $(w)_r = n$. 
We see that
\[\Phi_r(n) = (w_0 \cdots w_{\ell - 2})_r \quad \text{ and } \quad \Psi_r(n) = (w_1 \ldots w_{\ell-1})_r.\]
Since $\mathcal{L}(\Sigma)$ is closed under prefixes, $\Phi_r(n) \in A_\Sigma$, and since $\Sigma$ is shift-invariant, $\Psi_r(n) \in A_\Sigma$.  
This shows that $A_\Sigma$ is $\times r$-invariant.
The proof that $B_\Sigma$ is $\times r$-invariant is identical, only with the order of letters reversed.

Next we will show \eqref{eqn_in_prop_correspondence_integers_subshifts}. Since $A_\Sigma$ and $B_\Sigma$ are $\times r$-invariant, it follows from \cref{lemma_dimensions_coincide_for_invariant_sets} that $\dimdh A_\Sigma = \dimdm A_\Sigma$ and $\dimdh B_\Sigma = \dimdm B_\Sigma$. Therefore, it suffices to verify that $\dimdm A_\Sigma = \dimdm B_\Sigma = h_{\mathrm{top}}(\Sigma,T)/\log r$. 

Let $\mathcal{L}_\ell(\Sigma)$ denote the set of words of length $\ell$ appearing in the language set $\mathcal{L}(\Sigma)$, i.e.,
\[
\mathcal{L}_\ell(\Sigma)\,\defeq \,\big\{w_0w_1 \cdots w_{\ell-1} \ \big| \ w = w_0 w_1 \cdots \in \Sigma \big\}.
\]
It is well known (see, for instance, \cite[Theorem 7.13 (i)]{Walters75}) that the topological entropy of $(\Sigma,\sigma)$ is given by
\begin{align}
\label{eqn_h_top_limit_exists_and_equal}
  h_{\mathrm{top}}(\Sigma,\sigma)\,=\,\lim_{\ell\to\infty} \frac{1}{\ell}\log|\mathcal{L}_\ell(\Sigma)|,
\end{align}
where the limit as $\ell\to\infty$ on the right hand side is known to exist.
We claim that for all $\ell \in \Nz$,
\begin{align}
\label{eqn_bounds_on_A_sigma_size_on_initial_interval}
    \big| \mathcal{L}_\ell (\Sigma) \big| \leq \big| A_\Sigma \cap [0,r^\ell) \big| \leq \big| \bigcup_{k=0}^\ell \mathcal{L}_k(\Sigma) \big|.
\end{align}
Indeed, the first inequality follows immediately from the fact that $( \, \cdot \, )_r : \Lambda_r^\ell \to [0,r^\ell)$ is injective. 
For the second inequality, associate to each $n \in A_\Sigma \cap [0,r^\ell)$ a word $w \in \mathcal{L}(\Sigma)$ such that $w_0 \neq 0$ and $(w)_r = n$.  Since $n < r^\ell$, $|w| \leq \ell$.  The second inequality follows then from the fact that the association just described is bijective.

Using the fact that the limit in \eqref{eqn_h_top_limit_exists_and_equal} exists, it is a short exercise to show that $\lim_{\ell \to \infty} \allowbreak \log \big| \cup_{k=0}^\ell \mathcal{L}_k(\Sigma)\big| / \ell$ exists and is equal to $h_{\mathrm{top}}(\Sigma,\sigma)$. 
It follows from the inequalities in \eqref{eqn_bounds_on_A_sigma_size_on_initial_interval} that $\dimdm A_\Sigma = h_{\mathrm{top}}(\Sigma,T)/\log r$. The same argument shows that similarly $\dimdm B_\Sigma = h_{\mathrm{top}}(\Sigma,T)/\log r$, verifying the equality in \eqref{eqn_in_prop_correspondence_integers_subshifts}.

Finally, suppose $B \subseteq \Nz$ is a $\times r$-invariant set. We will prove that there exists a subshift $\Sigma \subseteq \Lambda_r^{\Nz}$ for which $B_\Sigma = B$.  
Let $\Sigma^{(\ell)}$ denote the set of all infinite words $w_0 w_1 \cdots \in \Lambda_r^{\Nz}$ for which $(w_{\ell - 1} \cdots w_0)_r \in B$, and
define
\begin{align}
    \label{eqn_def_of_subshift_whose_language_equals_a}
    \Sigma\,\coloneqq\, \bigcap_{\ell \in\N } \Sigma^{(\ell-1)}.
\end{align}
Being an intersection of closed sets, $\Sigma$ is closed. 
From $\Phi_r(B)\subseteq B$, it follows that $\sigma(\Sigma^{(\ell)})\subseteq \Sigma^{(\ell)}$, whereby $\sigma(\Sigma)\subseteq \Sigma$.  This proves that $(\Sigma,\sigma)$ is a subshift.
From the construction, it is clear that $B_\Sigma \subseteq B$.

On the other hand, if $(w_{\ell - 1} \cdots w_0)_r \in B$, then the infinite word $w_0\cdots w_{\ell-1}00\cdots \in \Sigma$. It follows that $(w_{\ell - 1} \cdots w_0)_r \in B_\Sigma$, showing that $B = B_\Sigma$.
\end{proof}

We note that the identification of $\times r$-invariant subsets of $\Nz$ and subshifts of $\Lambda_r^{\Nz}$ given by \cref{prop_correspondence_integers_subshifts} is not bijective.  The subshift $\Sigma$ defined in \eqref{eqn_def_of_subshift_whose_language_equals_a} can be shown to be the largest such that $B_\Sigma = B$, but in general there can be infinitely many distinct subshifts $\Sigma'$ such that $B_{\Sigma'} = B$.

As a corollary to \cref{prop_correspondence_integers_subshifts} we obtain the following result, which plays an important role in most of our main results.

\begin{corollary}
\label{cor_correspondence_integers_subshifts}
For any $\times r$-invariant $A\subseteq \Nz$, the set 
\[
A'\coloneqq\bigcap_{k\in\Nz}\bigcap_{\ell\in\Nz}\Phi_r^k\Psi_r^\ell(A)
\]
satisfies $\Phi_r(A')=\Psi_r(A')=A'$ (in particular, $A'$ is $\times r$-invariant) and $\dimdh A' = \dimdm A' = \dimdm A$.
\end{corollary}

\begin{proof}
Note that $A'$ is the largest subset of $A$ satisfying $\Phi_r(A')=\Psi_r(A')=A'$; in particular, it is $\times r$-invariant. Therefore, to prove $\dimdm A'=\dimdm A$, it suffices to find a subset $A''\subseteq A$ satisfying $\Phi_r(A'')=\Psi_r(A'')=A''$ and $\dimdm A'' = \dimdm A$. Appealing to \cref{lemma_dimensions_coincide_for_invariant_sets}, this would also prove that $\dimdh A'=\dimdm A$. If $\dimdm A =0$, then there is nothing to show, so let us proceed under the assumption that $\dimdm A>0$.

According to \cref{prop_correspondence_integers_subshifts}, we can find a subshift $\Sigma\subseteq \Lambda_r^{\Nz}$ such that $A$ coincides with the $r$-language set $B_\Sigma$ associated to $\Sigma$.
Let $\mu$ be an ergodic  $\sigma$-invariant Borel probability measure on $\Sigma$ of maximal entropy (the existence of such a measure follows from, e.g.\ \cite[Theorem 8.2 + Theorem 8.7 (v)]{Walters75}).
Let $\Sigma''$ denote the support of $\mu$, and observe that  $(\Sigma'',\sigma)$ is a subshift of $(\Sigma,\sigma)$ with $h_{\mathrm{top}}(\Sigma,\sigma)=h_{\mathrm{top}}(\Sigma'',\sigma)$.
Moreover, since $\mu$ is ergodic, almost every point in $\Sigma''$ has a dense orbit (by Birkhoff's ergodic theorem) and almost every point is recurrent (by Poincar\'e's recurrence theorem).
Therefore there exists a point $x\in\Sigma''$ which visits every non-empty open set in $\Sigma''$ infinitely often.

Let $A''\subseteq\Nz$ be the $r$-language set associated to $\Sigma''$, i.e., $A''=B_{\Sigma''}$. 
Since $\Sigma''\subseteq \Sigma$, we have $A''\subseteq A$.
Also, by \cref{prop_correspondence_integers_subshifts}, $\dimdm A=h_{\mathrm{top}}(\Sigma,\sigma) / \log r$, $\dimdm A''=h_{\mathrm{top}}(\Sigma'',\sigma) / \log r$, and $h_{\mathrm{top}}(\Sigma,\sigma)=h_{\mathrm{top}}(\Sigma'',\sigma)$, which implies $\dimdm A = \dimdm A''$.
All that remains to be shown is that $\Phi_r(A'')=\Psi_r(A'')=A''$.

Since $A''$ is an $r$-language set, it is $\times r$-invariant, so we already have the inclusions
\[
\Phi_r(A'')\subseteq A''\qquad\text{and}\qquad \Psi_r(A'')\subseteq A''.
\]
To prove the reverse inclusions, let $n\in A''$, and let $w_0\cdots w_{\ell - 1}\in \mathcal{L}(\Sigma'')$ be such that $n = (w_{\ell - 1} \cdots w_0)_r \in A''$.
Since the point $x$ visits every open set of $\Sigma''$ infinitely often, the word $w_0\cdots w_{\ell - 1}$ appears in $x$ infinitely often.
This implies that $x$ cannot be equal to $w_0\cdots w_{\ell - 1}0^\infty$ and so there exists a non-zero letter $u \in \Lambda_r$ and some $k\in\Nz$ such that the word $w_0\cdots w_{\ell - 1}0^ku$ appears in $x$ and hence in $\mathcal{L}(\Sigma'')$.
Now $(u 0^k w_{\ell - 1} \cdots w_0)_r \in A''$ and $\Psi_r(u 0^k w_{\ell - 1} \cdots w_0)_r = (w_{\ell - 1} \cdots w_0)_r = n$, showing that $A'' \subseteq \Psi_r(A'')$.

Invoking again the fact that the word $w_0\cdots w_{\ell - 1}$ appears infinitely often in $x$, there must exist a letter $v \in \Lambda_r$ such that the word $vw_0\cdots w_{\ell - 1}$ appears in $x$ and hence belongs to $\mathcal{L}(\Sigma'')$.
Now $(w_{\ell - 1} \cdots w_0 v)_r \in A''$ and $\Phi_r(w_{\ell - 1} \cdots w_0 v)_r = n$, showing that $A'' \subseteq \Phi_r(A'')$.
\end{proof}

A well-known fact from geometric measure theory states that if $X\subseteq[0,1]$ is multiplicatively invariant and has Hausdorff dimension $1$ then $X=[0,1]$ (see \cite[discussion after Conjecture 2]{furstenbergtransversality}).
The following corollary of \cref{prop_correspondence_integers_subshifts} offers a discrete analogue of this result and may be of independent interest.

\begin{corollary}
\label{cor_uniqueness_of_mass_dimension_1}
If $A\subseteq\Nz$ is multiplicatively invariant and $\udimdm A=1$, then $A=\Nz$.
\end{corollary}

\begin{proof}
Suppose $A$ is $\times r$-invariant with $\udimdm A=1$. 
It follows from \cref{lemma_dimensions_coincide_for_invariant_sets} that $\dimdm A = 1$.  
In view of \cref{prop_correspondence_integers_subshifts}, there exists a subshift $\Sigma\subseteq \Lambda_r^{\Nz}$ such that $A = B_\Sigma$ and $h_{\mathrm{top}}(\Sigma,\sigma)=\log r$. However, the only subshift of $\Lambda_r^{\Nz}$ with full entropy is the full shift. Hence $\Sigma=\Lambda_r^{\Nz}$, which implies $A= B_\Sigma = \Nz$. 
\end{proof}

\subsection{Connections to fractal geometry of the reals}
\label{sec_connection_01-interval}

The purpose of this subsection is to establish a connection between $\times r$-invariant subsets of the non-negative integers and $\times r$-invariant subsets of $[0,1]$. Recall that $X\subseteq [0,1]$ is called \define{$\times r$-invariant} if it is closed and $T_r X\subseteq X$, where $T_r\colon x\mapsto rx\bmod1$.

First, we remark that every $\times r$-invariant subset of $[0,1]$ can be ``lifted'' to a $\times r$-invariant subset of $\Nz$. Indeed, if $X\subseteq[0,1]$ is $\times r$-invariant, then one can show that the set
\[
\big\{\lfloor r^k x\rfloor \ | \ x\in X,~k \in \Nz \big\}
\]
is $\times r$-invariant. We will not make use of this fact, so we leave the details to the interested reader.
Of more importance to us is the converse direction, stated in the following proposition. 
Recall from \cref{sec_continuous_fractal_geometry} the definition of Hausdorff distance.

\begin{proposition}
\label{prop_correspondence_integers_01-interval}
For any $\times r$-invariant set $A\subseteq \Nz$, the sequence $X_k\coloneqq (A\cap [0,r^k))/r^k$ converges with respect to the Hausdorff metric $d_H$ as $k\to\infty$ to a $\times r$-invariant set $X\subseteq [0,1]$ satisfying $\dimm X=\dimdm A$.
\end{proposition}

We remark that by \cref{thm:furstenbergdimsofsubshift} and \cref{lemma_dimensions_coincide_for_invariant_sets}, the Minkowski and Hausdorff dimensions of multiplicatively invariant sets in $\Nz$ and $[0,1]$ coincide.  Thus, either dimension can be used in the conclusion of \cref{prop_correspondence_integers_01-interval}. For the proof of the proposition, we will need two technical lemmas.

\begin{lemma}
\label{lem_haus_containment}
Let $A\subseteq \Nz$, and define $X_k\coloneqq (A\cap [0,r^k))/r^k$.
\begin{enumerate}
[label=(\Roman{enumi}),ref=(\Roman{enumi}),leftmargin=*]
\item\label{itm_haus_containment_1}
If $\Phi_r(A)\subseteq A$, then for any $k,l\in\N$ with $l\geq k$, we have $X_l\subseteq [X_k]_{r^{-k}}$.
\item\label{itm_haus_containment_2}
If $\Phi_r(A)\supseteq A$, then for any $k,l\in\N$ with $l\geq k$, we have $X_k\subseteq [X_l]_{r^{-k}}$.
\end{enumerate}
In particular, if $\Phi_r(A)=A$ then for all $l\geq k$, we have $d_H(X_l,X_k)\leq r^{-k}$.
\end{lemma}

\begin{proof}
It is helpful to note first that for all $n, l, k \in \N$ with $l \geq k$,
\begin{align}
\label{dist_between_n_and_Phi_n}
\left| \frac{n}{r^l} - \frac{\Phi_r^{l-k}(n)}{r^k} \right| \leq \frac 1{r^k}.
\end{align}
This inequality follows easily from the fact that $\Phi_r^{l-k}(n) = \lfloor n / r^{l-k} \rfloor$. For the proof of part~\ref{itm_haus_containment_1}, let $y\in X_l$ and write $y= m/r^l$ for some $m\in A$. Note that $\tilde{m}\coloneqq \Phi_r^{l-k}(m)$ belongs to $A\cap [0,r^k)$ because $\Phi_r(A)\subseteq A$. Then, setting $\tilde{y}\coloneqq \tilde{m}/r^k$, we see that $\tilde{y}\in X_k$ and, by \eqref{dist_between_n_and_Phi_n}, $d(y,\tilde{y})\leq r^{-k}$. This proves $X_l\subseteq [X_k]_{r^{-k}}$.

Next, we prove part \ref{itm_haus_containment_2}. For any $x\in X_k$ we can find $n\in A\cap [0,r^k)$ such that $x=n/r^k$.
Since $A\subseteq \Phi_r^{l-k}(A)$, there exists $\tilde{n}\in A\cap [0,r^l)$ such that
\[
\Phi_r^{l-k}(\tilde{n})\,=\, n.
\]
Now $\tilde{x}\coloneqq \tilde{n}/r^l$ belongs to $X_l$ and it follows from \eqref{dist_between_n_and_Phi_n} that $d(x,\tilde{x})\leq r^{-k}$. This proves $X_k\subseteq [X_l]_{r^{-k}}$.
\end{proof}

\begin{lemma}
\label{lem_rise}
Suppose $A\subseteq \Nz$ satisfies $\Phi_r(A)\subseteq A$, and define $A'\coloneqq \bigcap_{k\in\N}\Phi_r^k(A)$. Also, set $X_k\coloneqq (A\cap [0,r^k))/r^k$ and $X_k'\coloneqq (A'\cap [0,r^k))/r^k$.
Then $\lim_{k\to\infty}d_H(X_k, X_k')=0$.
\end{lemma}
\begin{proof}
Let $\eps > 0$, and let $m \in \N$ such that $2r^{-m} < \eps$.
Since $\Phi_r(A)\subseteq A$, we have
\[
A\cap [0,r^m)~\supseteq~\Phi_r(A)\cap [0,r^m)~\supseteq~ \Phi_r^2(A)\cap [0,r^m)~\supseteq~ \Phi_r^3(A)\cap [0,r^m)~\supseteq~\ldots.
\]
In particular, the sequence $k\mapsto \Phi_r^k(A)\cap [0,r^m)$ eventually stabilizes, which happens exactly when $\Phi_r^k(A)\cap [0,r^m)=A'\cap [0,r^m)$.
It follows from \eqref{dist_between_n_and_Phi_n} that
\[X_k\subset\left[\frac{\Phi_r^{k-m}(A)\cap [0,r^m)}{{r^m}}\right]_{r^{-m}}.\]
Therefore, for large enough $k$, $X_k\subset[X'_m]_{r^{-m}}$. 
On the other hand it is clear that $X_k'\subset X_k$.
Finally, since from \cref{lem_haus_containment} we have that $d_H(X'_k,X'_m)<r^{-m}$, we conclude that $X_k'\subset X_k\subset [X_m']_{r^{-k}}\subset[X_k']_{2r^{-m}}$, when it follows that $d_H(X_k,X_k')<\epsilon$.
\end{proof}

\begin{proof}[Proof of \cref{prop_correspondence_integers_01-interval}] Define $A'\coloneqq \bigcap_{k\in\Nz}\Phi_r^k(A)$ and $X_k'\coloneqq (A'\cap [0,r^k))/r^k$.
In view of \cref{lem_rise}, the sequence $k\mapsto X_k$ converges with respect to the Hausdorff metric if and only if the sequence $k\mapsto X_k'$ converges.
Since $A'=\Phi_r(A')$, it follows from \cref{lem_haus_containment} that
\[
d_H(X_k',X_l')\leq r^{-k},\qquad\text{for all}~k,l\in\N~\text{with}~l\geq k.
\]
This implies that $k\mapsto X_k'$ is a Cauchy sequence, and hence it is convergent (recall that by the Blaschke selection theorem, the set of all non-empty, compact subsets of $[0,1]$ equipped with the Hausdorff distance is a complete metric space).  Let $X' = \lim_{k \to \infty} X_k'$, and note that $X' \subseteq X$.

Next, let us show that $X$ is $\times  r$-invariant. Since $\Psi_r(A)\subseteq A$, a simple computation shows $T_r(X_k)\subseteq X_{k-1}$.
Therefore, using $X=\lim_{k\to\infty} X_k$ and the fact that $T_r$ is continuous on $[0,1)\setminus \{0,\frac{1}{r},\ldots,\frac{r-1}{r}\}$, we get that for any closed set $C\subseteq [0,1)\setminus \{0,\frac{1}{r},\ldots,\frac{r-1}{r}\}$,
\begin{align*}
T_r(X\cap C)&=  T_r \big(\lim_{k\to\infty}(X_k\cap C) \big)
\\
&= \lim_{k\to\infty} T_r(X_k\cap C)
\\
&\subseteq \lim_{k\to\infty} T_r(X_k)
\\
&\subseteq \lim_{k\to\infty} X_{k-1}
\\
&= X.
\end{align*}
It follows that $T_r(X\setminus \{0,\frac{1}{r},\ldots,\frac{r-1}{r}\})\subseteq X$. Since $0\in X$, we obtain $T(\{0,\frac{1}{r},\ldots,\frac{r-1}{r}\})\subseteq X$, and hence $T_r(X)\subseteq X$, as desired.

Finally, we must show $\dimm X=\dimdm A$. As guaranteed by \cref{cor_correspondence_integers_subshifts}, $\dimdm A= \dimdm A'$. 
By combining part \ref{itm_haus_containment_1} of \cref{lem_haus_containment} with \cref{lem_box_counting_estimate}, we see that
\begin{align}\label{eqn_limsup_is_less}\begin{aligned}
0 &\leq \liminf_{k \to \infty} \left( \frac{\log\NC \big(X_k, r^{-k} \big)}{k\log r} - \frac{\log\NC \big(X, r^{-k}\big)}{k\log r} \right)\\
&= \dimdm A - \limsup_{k \to \infty} \frac{\log\NC \big(X, r^{-k}\big)}{k\log r},
\end{aligned}\end{align}
where the equality follows from the fact that $\dimdm A=\lim_{k\to\infty}\frac{1}{k\log r}\log\NC (X_k, r^{-k} )$ (cf.\ equation \eqref{eqn_mass-dim_to_Mink-dim}). On the other hand, using part \ref{itm_haus_containment_2} of \cref{lem_haus_containment}, \cref{lem_box_counting_estimate}, and the fact that $\dimdm A'=\lim_{k\to\infty}\frac{1}{k\log r}\log\NC (X_k', r^{-k} )$, we see
\begin{align}\label{eqn_limsup_is_more}\begin{aligned}
0 &\leq \liminf_{k \to \infty} \left( \frac{\log\NC \big(X', r^{-k} \big)}{k\log r} - \frac{\log\NC \big(X_k', r^{-k} \big)}{k\log r} \right) \\
&= \liminf_{k \to \infty} \frac{\log\NC \big(X', r^{-k} \big)}{k\log r} - \dimdm A'.
\end{aligned}\end{align}
Combining \eqref{eqn_limsup_is_less} and \eqref{eqn_limsup_is_more} with the fact that $X'\subseteq X$, we see
\[\dimdm A' \leq \liminf_{k \to \infty} \frac{\log\NC \big(X', r^{-k} \big)}{k\log r} \leq \limsup_{k \to \infty} \frac{\log\NC \big(X, r^{-k}\big)}{k\log r} \leq \dimdm A.\]
Since $\dimdm A = \dimdm A'$ and $X'\subseteq X$, we conclude that $\dimm X$ exists and is equal to $\dimdm A$.
\end{proof}

\section{Transversality between multiplicatively invariant subsets of the integers}
\label{sec_integer_transversality}

In this section, we prove our main results, Theorems \ref{prop_discrete_furstenberg}, \ref{mainthm_integer_intersections}, \ref{mainthm_integer_sumsets}, and \ref{theorem_lmp_analogue}.  
As in the other sections, the positive integers $r$ and $s$ are fixed, and the implicit constants appearing in asymptotic notation may depend on $r$ and $s$ without further indication.

\subsection{Sets which are simultaneously multiplicatively invariant}
\label{sec_proof_of_thm_B}

In this subsection we give a proof of \cref{prop_discrete_furstenberg}.  We follow the notation and terminology established in \cref{sec_first_dim_results}.  We say that a non-negative integer $n$ \emph{begins with the word $w$ in base $s$} if there exists $d \in \Nz$ and $n_0 \in [0,s^{d})$ such that
\begin{align}\label{eqn_begins_with_omega}
n = (w)_s s^d + n_0.
\end{align}
If $w = w_0 \cdots w_{\ell - 1}$ and $w_0 \neq 0$, this means that the $\ell$ most significant digits in the base-$s$ expansion of $n$ are $w_0$, $w_1$, \dots, $w_{\ell-1}$, in order.

\begin{lemma}
\label{lemma_interals_determine_starting_words}
For all $w \in \Lambda_s^{\ell}$, there is an arc $I_w \subseteq [0,1)$ modulo 1 (meaning that $I_w$ is an interval when $0$ and $1$ are identified) with the property that for all $x \geq (w)_s$, the integer $\lfloor x \rfloor$ begins with $w$ in base $s$ if and only if $\{ \log x / \log s\} \in I_w$.
\end{lemma}

\begin{proof}
Let $w \in \Lambda_s^{\ell+1}$. It follows from \eqref{eqn_begins_with_omega} that a positive integer $n$ begins with $w$ in base $s$ if and only if there exists $d \in \Nz$ such that
\[(w)_s s^d \leq n < \big( (w)_s + 1) s^d.\]
Therefore, a positive real number $x$ has the property that $\lfloor x \rfloor$ begins with $w$ in base $s$ if and only if there exists $d \in \Nz$ such that
\begin{align*}
    (w)_s s^d \leq x < \big( (w)_s + 1) s^d.
\end{align*}
The previous inequality is equivalent to
\begin{align}\label{proof_of_B_interval_inequality}
\frac{\log (w)_s}{\log s}  + d \leq \frac {\log x}{\log s} < \frac{\log \big( (w)_s + 1\big)}{\log s}  + d.
\end{align}
Let $I_w$ be the modulo 1 arc from the fractional part of $\log (w)_s / \log s$ to the fractional part of $\log \big( (w)_s + 1\big) / \log s$ in the positive direction. We see that for all $x \geq (w)_s$, the integer $\lfloor x \rfloor$ begins with $w$ in base $s$ if and only if \eqref{proof_of_B_interval_inequality} holds, which happens if and only if $\big\{ \log x / \log s \big\} \in I_w$.
\end{proof}

Recall from \cref{sec_continuous_fractal_geometry} that $[A]_\delta$ denotes the $\delta$-neighborhood of $A$.

\begin{lemma}
\label{lemma_B_contains_words_beginning_with_omega}
Let $r$ and $s$ be multiplicatively independent positive integers, and let $A \subseteq \Nz$ be $\times r$ invariant and infinite. If $\lambda, \delta > 0$, $\tau \in \R$, and $B \subseteq \Nz$ are such that $\lambda A + \tau \subseteq [B]_\delta$, then for all $w \in \Lambda_s^*$, there exists an integer in $B$ that begins with $w$ in base $s$.
\end{lemma}

\begin{proof}
Let $w \in \Lambda_s^*$, and let $I_w$ be the arc from \cref{lemma_interals_determine_starting_words}.  Let $I_w'$ be the middle 3rd subinterval of $I_w$, and let $\xi$ be the length of $I_w'$. Define $\al = \log r / \log s$. Since $\alpha$ is irrational, there exists $K \in \N$ such that the set $\big\{\{i \al\} \ | \ i \in \{0, \ldots, K\} \big\}$ is $\xi$-dense in $[0,1)$.

Since $A$ is infinite, there exists $n \in A$ sufficiently large (to be specified momentarily) such that $\lambda n / s^K + \tau \geq (w)_s + \delta + \lambda$. Since $A$ is $\Phi_r$-invariant, $n$, $\lfloor n / r \rfloor$, \dots, $\lfloor n / r^K \rfloor$ are all elements of $A$.  Let $i \in \{0, \ldots, K\}$.  Since $\lambda A + \tau \subseteq [B]_\delta$, the real number $\lambda \lfloor n / r^i \rfloor + \tau$ is within a distance $\delta$ of the set $B$.  Therefore, there exists $t_i \in \R$, $|t_i| \leq \lambda + \delta$, such that $\lambda n / r^i + \tau + t_i \in B$.

By the mean value theorem, ensuring that $n$ is sufficiently large, we see that for all $i \in \{0, \ldots, K\}$,
\begin{align}
\label{eqn_xr-inv_approx}
    \left| \frac{\log \big( \lambda n / r^i + \tau + t_i\big)}{\log s} - \frac{\log \big( \lambda n / r^i\big)}{\log s} \right| < \xi.
\end{align}
It follows from the fact that $\log \big( \lambda n / r^i\big) / \log s = \log (\lambda n) / \log s - i \alpha$ and from our choice of $K$ that there exists $i \in \{0, \ldots, K\}$ such that $\big \{\log \big( \lambda n / r^i\big) / \log s\big\} \in I_w'$.  It follows from \eqref{eqn_xr-inv_approx} and the definition of $\xi$ that $\big \{\log \big( \lambda n / r^i + \tau + t_i\big) / \log s\big\} \in I_w$.  By our choice of $n$ and the fact that $i \leq K$, we have that $\lambda n / r^i + \tau + t_i \geq (w)_s$.  Therefore, \cref{lemma_interals_determine_starting_words} gives that $\lambda n / r^i + \tau + t_i$, an integer in $B$, begins with the word $w$ in base $s$, as was to be shown.
\end{proof}

\begin{proof}[Proof of \cref{prop_discrete_furstenberg}]
Let $r$ and $s$ be multiplicatively independent positive integers, and let $A, B \subseteq \Nz$ be $\times r$- and $\times s$-invariant sets, respectively. 
Suppose $\lambda, \eta > 0$, $\sigma, \tau \in \R$ and $\delta > 0$ are such that $\lambda A + \tau \subseteq \big[ \eta B + \sigma \big]_\delta$.
We need to show that
then either $A$ is finite or $B = \Nz$.

Suppose $A$ is infinite; we will argue that $B = \Nz$.  
Since $B$ is $\times s$-invariant, it suffices to show that for all $w \in \Lambda_s^*$, there exists an integer in $B$ that begins with $w$ in base $s$.

Let $w \in \Lambda_s^*$. 
It follows from \eqref{eqn_affine_falls_in_affine_neighborhood} that $\lambda' A + \tau' \subseteq [B]_{\delta'}$, where $\lambda' = \lambda / \eta$, $\tau' = (\tau - \sigma) / \eta$ and $\delta'=\delta/\eta$. 
Since $A$ is $\times r$-invariant and infinite, \cref{lemma_B_contains_words_beginning_with_omega} gives that some integer in $B$ begins with $w \in \Lambda_s^*$ in base $s$, as was to be shown.
\end{proof}

\subsection{Intersections of multiplicatively independent invariant sets}
\label{sec_intersections_section}

In this subsection, we prove \cref{mainthm_integer_intersections}, showing that $\times r$- and $\times s$-invariant sets are geometrically transverse in the sense that the dimension of the intersection of one with any affine image of the other is small.
In fact we prove the following stronger version.

\begin{theorem}
\label{theorem_main_discrete_intersections_v2}
Let $r$ and $s$ be multiplicatively independent positive integers, and let $A, B \subseteq \Nz$ be $\times r$- and $\times s$-invariant sets, respectively. Define $\ogamma = \max \big(0, \ \dimdh A + \dimdh B - 1 \big)$. 
For every compact set $I \subseteq \R \setminus \{0\}$ and $\eps > 0$,
\begin{align}
\label{eqn_intervals_on_inside_suffices_for_intersections}
    \lim_{N \to \infty} \ \sup_{\substack{\lambda, \eta \in I \\ \sigma, \tau \in \R}} \ \frac{ \big| \big\lfloor \lambda \big( A \cap \big[0,N \big) \big) + \tau \big \rfloor \cap \big\lfloor \eta \big( B \cap \big[0,N \big) \big) + \sigma \big\rfloor \big| } {N^{\ogamma + \eps}}  = 0.
\end{align}
In particular, for all $\lambda, \eta, \sigma, \tau \in \R$,
\begin{align}
\label{eqn_intersection_dim_inequa_for_infinite_sets}
    \udimdm \big( \lfloor \lambda A + \tau \rfloor \cap \lfloor \eta B + \sigma \rfloor \big)  \leq \max \big(0, \ \dimdh A + \dimdh B - 1 \big).
\end{align}
\end{theorem}

\begin{proof}
Let $I \subseteq \R \setminus \{0\}$ be compact and $\eps > 0$.
Since $\big\lfloor \lambda \big( A \cap \big[0,N \big) \big) + \tau \big\rfloor \subseteq \big[ \lambda \big( A \cap \big[0,N \big) \big) + \tau \big]_1$ and $\big\lfloor \eta \big( B \cap \big[0,N \big) \big) + \sigma \big\rfloor \subseteq \big[ \eta \big( B \cap \big[0,N \big) \big) + \sigma \big]_1$, the cardinality in the numerator on the left hand side of \eqref{eqn_intervals_on_inside_suffices_for_intersections} is bounded from above by
\[\NC \big( \big[ \lambda \big( A \cap \big[0,N \big) \big) + \tau \big]_1 \cap \big[ \eta \big( B \cap \big[0,N \big) \big) + \sigma \big]_1, 1 \big),\]
which is quickly seen to be equal to
\begin{align}
\label{eqn_intersections_div_by_N}
    \NC \left( \left[ \lambda \left( \frac {A \cap \big[0,N \big)}{N} \right) + \frac{\tau}{N} \right]_{N^{-1}} \cap \left[ \eta \left( \frac{B \cap \big[0,N \big)}{N} \right) + \frac{\sigma}N \right]_{N^{-1}}, N^{-1} \right).
\end{align}

Define for every $k,\ell\in\N$ the sets 
\[X_k \coloneqq \frac{A\cap[0,r^k)}{r^k} \ \ \text{ and } \ \ Y_\ell \coloneqq \frac{B\cap[0,s^\ell)}{s^\ell}.\]
Define $k_N\coloneqq \lfloor \log N/ \log r\rfloor + 1$ and $\ell_N\coloneqq \lfloor \log N/ \log s\rfloor + 1$, and note that
\[N=r^{k_N} r^{\{\log N/\log r\} - 1} = s^{\ell_N} s^{\{\log N/\log s\} - 1}.\]
Since $N \leq \min(r^{k_N},s^{\ell_N})$, we have that $A \cap [0,N)\subseteq A \cap [0,r^{k_N})$ and $B \cap [0,N)\subseteq B \cap [0,s^{\ell_N})$.  Therefore, the expression in \eqref{eqn_intersections_div_by_N} is bounded from above by
\[\NC \big( \big[ \lambda r^{1 - \{\log N/\log r\}} X_{k_N} + \tau / N \big]_{N^{-1}} \cap \big[ \eta s^{1-\{\log N/\log s\}} Y_{\ell_N} + \sigma / N \big]_{N^{-1}}, N^{-1} \big).\]

Since $I\subseteq \R \setminus \{0\}$ is compact, there exists $t>1$ such that $I \subseteq \pm [t^{-1}, t]$. If $\lambda$ and $\eta$ belong to $I$, then $\lambda r^{1-\{\log N/\log r\}}$ and $\eta s^{1-\{\log N/\log s\}}$ belong to $J \defeq \pm [t^{-1}, \max (r,s) t]$. Therefore, to show \eqref{eqn_intervals_on_inside_suffices_for_intersections}, it suffices to prove
\begin{align}
\label{eqn_suffices_for_intersections_with_xk_yl_reform}
    \lim_{N \to \infty} \ \sup_{\substack{\lambda, \eta \in J \\ \sigma, \tau \in \R}} \ \frac{ \NC \big( \big[ \lambda X_{k_N} + \tau \big]_{N^{-1}} \cap \big[ \eta Y_{\ell_N} + \sigma \big]_{N^{-1}}, N^{-1} \big) } {N^{\ogamma + \eps}}  = 0.
\end{align}

In view of \cref{prop_correspondence_integers_01-interval}, the limits $X:=\lim_{k \to \infty} X_k$ and $Y:=\lim_{\ell \to \infty} Y_\ell$ exist in the Hausdorff metric.
Moreover, $X$ and $Y$ are $\times r$- and $\times s$-invariant, respectively, and $\dimh X=\dimdh A$, $\dimh Y=\dimdh B$.  
By \cref{lem_haus_containment}, we have that $d_H(X_{k_N}, X) \leq N^{-1}$ and $d_H(Y_{\ell_N}, Y) \leq N^{-1}$. Put $a = \max J$, and note that for all $\lambda, \eta \in J$ and $\sigma, \tau \in \R$,
\begin{align}
\label{eqn_replace_approx_to_X_with_X_by_increasing_nbhd}
    \big[ \lambda X_{k_N} + \tau \big]_{N^{-1}} \cap \big[ \eta Y_{\ell_N} + \sigma \big]_{N^{-1}} \subseteq \big[ \lambda X  + \tau \big]_{aN^{-1}} \cap \big[ \eta Y + \sigma \big]_{aN^{-1}}.
\end{align}

We can now manipulate the left hand side of \eqref{eqn_suffices_for_intersections_with_xk_yl_reform} using \eqref{eqn_replace_approx_to_X_with_X_by_increasing_nbhd}, \cref{lemma_metric_entropy_at_scaled_scale}, and \cref{theorem_main_continuous_intersections_v2} (with $J$ as $I$), to get
\begin{align*}
    & \limsup_{N \to \infty} \ \sup_{\substack{\lambda, \eta \in J \\ \sigma, \tau \in \R}} \ \frac{ \NC \big( \big[ \lambda X_{k_N} + \tau \big]_{N^{-1}} \cap \big[ \eta Y_{\ell_N} + \sigma \big]_{N^{-1}}, N^{-1} \big) } { N^{\ogamma + \eps}} \\
    & \qquad \leq  \limsup_{N \to \infty} \ \sup_{\substack{\lambda, \eta \in J \\ \sigma, \tau \in \R}} \ \frac{ \log \NC \big( \big[ \lambda X + \tau \big]_{aN^{-1}} \cap \big[ \eta Y + \sigma \big]_{aN^{-1}}, N^{-1} \big) } { N^{\ogamma + \eps}} \\
    & \qquad \ll  \lim_{N \to \infty} \ \sup_{\substack{\lambda, \eta \in J \\ \sigma, \tau \in \R}} \ \frac{ \NC \big( \big[ \lambda X + \tau \big]_{aN^{-1}} \cap \big[ \eta Y + \sigma \big]_{aN^{-1}}, a N^{-1} \big) } { (aN)^{\ogamma + \eps}} = 0.
\end{align*}
This verifies \eqref{eqn_suffices_for_intersections_with_xk_yl_reform} and concludes the proof of \eqref{eqn_intervals_on_inside_suffices_for_intersections}.

To show \eqref{eqn_intersection_dim_inequa_for_infinite_sets}, let $\lambda, \eta, \sigma, \tau \in \R$.  Put $M = 3\max \big( |\lambda|^{-1}, |\eta|^{-1} \big)$, and note that for all $N \geq \max \big(|\sigma|, |\tau| \big)$,
\[\big\lfloor \lambda A + \tau \big\rfloor \cap \big\lfloor \eta B + \sigma \big\rfloor \cap \big[0, N \big) \subseteq \big\lfloor \lambda \big(A \cap\big[0, MN \big) \big) + \tau \big\rfloor \cap \big\lfloor \eta \big( B \cap  \big[0, MN \big) \big)+ \sigma \big\rfloor.\]
It follows from this containment and \eqref{eqn_intervals_on_inside_suffices_for_intersections} that for all $\eps > 0$,
\[\frac{1}{M^{\ogamma + \eps}} \lim_{N \to \infty} \frac{\big| \big\lfloor \lambda A + \tau \big\rfloor \cap \big\lfloor \eta B + \sigma \big\rfloor \cap \big[0, N \big) \big|}{N^{\ogamma + \eps}}  = 0.\]
This proves \eqref{eqn_intersection_dim_inequa_for_infinite_sets} and concludes the proof of the theorem.
\end{proof}

\begin{remark}
We note two modifications to the statement of \cref{theorem_main_discrete_intersections_v2} that can be proved with minor corresponding modifications made to the proof.  First, the initial interval $[0,N)$ can be replaced by an interval symmetric about the origin, $(-N, N)$.  Though $A$ and $B$ consist of positive integers, this is meaningful because the theorem allows for $\lambda$ and/or $\eta$ to be negative. Second, using the floor function to round to the integer lattice is a mere convenience: the result hold when the sets $\lambda X + \tau$ and $\eta Y + \sigma$ are rounded to any other discrete subgroup (or translate of a discrete subgroup) of $\R$.
\end{remark}

\subsection{Sums of multiplicatively independent invariant sets}\label{sec_sumsintegers}

In this subsection, we prove \cref{mainthm_integer_sumsets}, showing that sets which are multiplicatively invariant with respect to multiplicatively independent bases are transverse in an additive combinatorial sense.  The results can be phrased in terms of the size (cardinality or Hausdorff content) of finite subsets of multiplicatively invariant sets.  The upper bounds on the size of the sumsets are contained in \cref{lemma_upper_bounds_on_dims_of_sumsets} and follow from general considerations.  The difficulty in the main results is in proving the lower bounds, which are handled in \cref{theorem_main_discrete_sums_v2} and are derived from their continuous counterparts in \cref{theorem_main_continuous_sumsets_v2}.

\begin{lemma}
\label{lemma_upper_bounds_on_dims_of_sumsets}
For all finite, non-empty $A', B' \subseteq \N_0$, all $\lambda, \eta > 0$, and all $0 \leq \gamma \leq 1$,
\begin{align}
\label{eqn_mass_dim_sumset_upper_bound}
    \big|  \big\lfloor \lambda A' + \eta B' \big\rfloor \big| &\leq \big| A' \times B' \big|,\\
\label{eqn_haus_content_sumset_upper_bound}
    \HC_{\geq 1}^{\gamma} \big( \big\lfloor \lambda A' + \eta B' \big\rfloor \big) &\ll_{\max(\lambda, \eta)} \HC_{\geq 1}^{\gamma} \big( A' \times B' \big).
\end{align}
Moreover, for all $A, B \subseteq \Nz$, all $\dim \in \{\ldimdm, \udimdm, \ldimdh, \udimdh\}$, and all $\lambda, \eta > 0$,
\[\dim \big( \lfloor \lambda A + \eta B \rfloor \big) \leq \min \big( 1, \dim (A \times B) \big).\]
\end{lemma}

\begin{proof}
Let $A', B' \subseteq \N_0$ be finite, non-empty, let $\lambda, \eta > 0$, and let $0 \leq \gamma \leq 1$. Denote by $\varphi: \R^2 \to \R$ the map $\varphi(x,y) = \lambda x + \eta y$; it is Lipschitz with Lipschitz constant $\max(\lambda, \eta)$.  Note that $\varphi (A' \times B') = \lambda A' + \eta B'$.

The upper bound in \eqref{eqn_mass_dim_sumset_upper_bound} follows from the fact that $|\lfloor \varphi(A' \times B')\rfloor| \leq |\varphi(A' \times B')| \leq |A' \times B'|$, while the upper bound in \eqref{eqn_haus_content_sumset_upper_bound} follows from \cref{lemma_hausdorff_near_implies_contents_near} and \cref{lemma_hausdorff_content_bound_under_lipschitz_map} via
\[\HC_{\geq 1}^{\gamma} \big( \lfloor \varphi(A' \times B') \rfloor \big) \asymp\HC_{\geq 1}^{\gamma} \big( \varphi( A' \times B') \big) \ll_{\max(\lambda, \eta)} \HC_{\geq 1}^{\gamma} \big( A' \times B' \big).\]

To prove the dimension inequality for $A, B \subseteq \Nz$, note that there exists $M \in \N$, depending only on $\max(\lambda,\eta)$, such that for all $N \in \N$,
\begin{align}
    \label{eqn_containment_on_finite_segments}
    \big\lfloor \lambda A + \eta B \big\rfloor \cap [0,N) \subseteq \big\lfloor \lambda (A \cap [0,NM)) + \eta (B \cap [0,NM)) \big\rfloor.
\end{align}
Let $\dim \in \{\ldimdm, \udimdm, \ldimdh, \udimdh\}$, and let $\gamma > \dim (A \times B)$.  It follows from \eqref{eqn_mass_dim_sumset_upper_bound}, \eqref{eqn_haus_content_sumset_upper_bound}, and \eqref{eqn_containment_on_finite_segments} that
\begin{align*}
    \frac{\big|\big\lfloor \lambda A + \eta B \big\rfloor \cap [0,N)\big|}{N^{\gamma}} &\leq M^{\gamma }\frac{\big|\big( A \times B \big) \cap [0,NM)^2\big| }{(NM)^{\gamma}}, \\
    \frac{\HC_{\geq 1}^{\gamma} \big( \big\lfloor \lambda A + \eta B \big\rfloor \cap [0,N)\big)}{N^{\gamma}} &\ll_{\max(\lambda, \eta)} M^{\gamma }\frac{\HC_{\geq 1}^{\gamma} \big(\big( A \times B \big) \cap [0,NM)^2\big) }{(NM)^{\gamma}}.
\end{align*}
Considering the first or second inequality (if $\dim$ is the discrete Minkowski or Hausdorff dimension, respectively), the limit infimum or limit supremum (if $\dim$ is a lower or upper dimension, respectively) of the quantity on the right hand side is equal to zero because $\gamma > \dim (A \times B)$.  It follows that $\dim \big( \lfloor \lambda A + \eta B \rfloor \big) \leq \gamma$. This suffices for the conclusion of the lemma since $\gamma > \dim (A \times B)$ was arbitrary and since $\dim \big( \lfloor \lambda A + \eta B \rfloor \big)$ is clearly bounded from above by 1.
\end{proof}

\begin{theorem}
\label{theorem_main_discrete_sums_v2}
Let $r$ and $s$ be multiplicatively independent positive integers, and let $A, B \subseteq \Nz$ be $\times r$- and $\times s$-invariant sets, respectively. Define $\ogamma = \max(0, \dimh (A \times B) - 1)$.  For all compact $I \subseteq (0, \infty)$, all $0\leq \gamma \leq 1$, all $\eps > 0$, all sufficiently large $N$ (depending on $A$, $B$, $I$, $\gamma$, and $\eps$), all non-empty $A' \subseteq A \cap [0,N)$ and $B' \subseteq B \cap [0,N)$, and all $\lambda, \eta \in I$,
\begin{align}
\label{eqn_equivalence_D_E_1_mass_dim}
    \big|  \big\lfloor \lambda A' + \eta B' \big\rfloor \big| &\geq \frac{\big| A' \times B' \big|}{N^{\ogamma + \eps}}, \text{ and}\\
\label{eqn_equivalence_D_E_1}
    \frac{\HC_{\geq 1}^{\gamma} \big( \big\lfloor \lambda A' + \eta B' \big\rfloor \big)}{N^{\gamma}} &\gg_{I, \gamma, \eps} \frac{\HC_{\geq 1}^{\gamma + \ogamma + \eps} \big( A' \times B' \big)}{N^{\gamma + \ogamma + \eps}}.
\end{align}
\end{theorem}

\begin{proof}
For all $k,\ell\in\N$, define the sets 
\[X_k \coloneqq \frac{A\cap[0,r^k)}{r^k} \quad \text{ and } \quad Y_\ell \coloneqq \frac{B\cap[0,s^\ell)}{s^\ell}.\]
Let $X=\lim_{k \to \infty} X_k$ and $Y=\lim_{\ell \to \infty} Y_\ell$ in the Hausdorff metric; \cref{prop_correspondence_integers_01-interval} gives that these limits exist, that $X$ and $Y$ are $\times r$- and $\times s$-invariant subsets of $[0,1]$, respectively, and that $\dimh X=\dimdh A$ and $\dimh Y=\dimdh B$.  For $N \in \N$, define $k_N\coloneqq \lfloor \log N/ \log r\rfloor + 1$ and $\ell_N \coloneqq \lfloor \log N/ \log s\rfloor + 1$, and note that
\begin{align}
\label{eqn_N_written_as_powers_of_r_and_s}
    N=r^{k_N} r^{\{\log N/\log r\}-1}=s^{\ell_N} s^{\{\log N/\log s\}-1}.
\end{align}
By \cref{lem_haus_containment}, we have that
\begin{align}
\label{eqn_dist_from_approx_to_full_sets}
    d_H(X_{k_N}, X) \leq N^{-1} \text{ and }d_H(Y_{\ell_N}, Y) \leq N^{-1}.
\end{align}

Let $I \subseteq (0, \infty)$ be compact, $0 \leq \gamma \leq 1$, and $\eps > 0$. Define $J \defeq [\min I, rs \max I]$.  
Next we invoke \cref{theorem_main_continuous_sumsets_v2} with $J$ in place of $I$ and either $\eps/2$ in place of $\eps$ (to prove \eqref{eqn_equivalence_D_E_1_mass_dim}) or $\eps$ as it is (to prove \eqref{eqn_equivalence_D_E_1}).
Let $N$ be sufficiently large, to be specified later, but in particular so that $\rho:=1/N$ is sufficiently small for \cref{theorem_main_continuous_sumsets_v2} to apply (with $\eps$ as either $\eps / 2$ or $\eps$).

Let $A' \subseteq A \cap [0,N)$ and $B' \subseteq B \cap [0,N)$ be non-empty, and $\lambda, \eta \in I$.  
It follows from \eqref{eqn_N_written_as_powers_of_r_and_s} that $N \leq \min(r^{k_N}, s^{\ell_N})$, whereby
\[\frac{A'}{r^{k_N}} \subseteq X_{k_N} \text{ and } \frac{B'}{s^{k_N}} \subseteq Y_{k_N}.\]
Combining these facts with \eqref{eqn_dist_from_approx_to_full_sets}, it follows from \cref{lemma_subsets_for_hausdorff_distance} that there exist non-empty compact sets $X' \subseteq X$ and $Y' \subseteq Y$ such that
\begin{align}
\label{eqn_Xprime_and_Yprime_as_near_subsets}
    d_H \left( X', \frac{A'}{r^{k_N}} \right) \leq N^{-1} \text{ and } d_H \left( Y', \frac{B'}{s^{\ell_N}} \right) \leq N^{-1}.
\end{align}
Define $\lambda' = r^{k_N} \lambda / N = r^{1-\{\log N/\log r\}} \lambda$ and $\eta' = s^{\ell_N} \eta / N = s^{1-\{\log N/\log s\}} \eta$.  Note that $\lambda', \eta' \in J$ and that
\begin{align}
\label{eqn_sumset_equality_in_proof_of_discrete_sumset_theorem}
    \lambda' \frac{A'}{r^{k_N}} + \eta' \frac{B'}{s^{\ell_N}} = \frac{\lambda A' + \eta B'}{N}.
\end{align}
Combining \eqref{eqn_Xprime_and_Yprime_as_near_subsets} and \eqref{eqn_sumset_equality_in_proof_of_discrete_sumset_theorem} with basic properties of the Hausdorff distance, we see that
\begin{align}
\label{eqn_sumset_hausdorff_distance_1}
    d_H \left( \lambda' X' + \eta' Y', \frac{\lambda A' + \eta B'}{N}\right) &\leq 2 rs \max (I) N^{-1}, \text{ and} \\
    \label{eqn_product_set_hausdorff_distance_1}
    d_H \left( X' \times Y', \frac{A'}{r^{k_N}} \times \frac{B'}{s^{\ell_N}} \right) &\leq N^{-1}.
\end{align}
It follows from \cref{lemma_hausdorff_content_bound_under_lipschitz_map} and \eqref{eqn_product_set_hausdorff_distance_1} that
\begin{gather}
\label{eqn_sumset_hausdorff_distance_2}
    \NC \left( X' \times Y', N^{-1} \right) \asymp \NC \left( \frac{A' \times B'}{N}, N^{-1} \right) = \NC \left( A' \times B', 1 \right) = |A' \times B'|, \text{ and} \\
    \label{eqn_product_set_hausdorff_distance_2}
    \HC_{\geq N^{-1}}^{\gamma + \ogamma + \eps} \left( X' \times Y' \right) \asymp \HC_{\geq N^{-1}}^{\gamma + \ogamma + \eps} \left( \frac{A' \times B'}{N} \right) = \frac{\HC_{\geq 1}^{\gamma + \ogamma + \eps} \left( A' \times B' \right)}{N^{\gamma + \ogamma + \eps}}.
\end{gather}

Appealing to \eqref{eqn_sumset_hausdorff_distance_1},  \cref{lemma_hausdorff_near_implies_contents_near}, \cref{theorem_main_continuous_sumsets_v2} (with $\eps / 2$ as $\eps$), and \eqref{eqn_sumset_hausdorff_distance_2}, we see that
\begin{gather*}
    |\lfloor \lambda A' + \eta B' \rfloor| 
    \asymp
    \NC \Big( \lambda A' + \eta B', 1 \Big) 
    =
    \NC \left( \frac{\lambda A' + \eta B'}{N}, N^{-1} \right) 
    \\ \asymp_I
    \NC \big( \lambda' X' + \eta' Y', N^{-1} \big) 
    \geq 
    \frac{\NC \big( X' \times Y', N^{-1} \big)}{N^{\ogamma + \eps / 2}}
    \asymp 
    \frac{|A' \times B'|}{N^{\ogamma + \eps / 2}}.
\end{gather*}
Thus, there exists a constant $C > 0$ depending only on $r$, $s$, and $I$ for which $|\lfloor \lambda A' + \eta B' \rfloor| \geq |A' \times B'| / (CN^{\ogamma + \eps / 2})$.  The inequality in \eqref{eqn_equivalence_D_E_1_mass_dim} follows as long as $N^{\eps/2} > C$.

Replacing cardinality and packing number with the $\gamma$-dimensional discrete Hausdorff content and appealing to  \eqref{eqn_sumset_hausdorff_distance_1}, \cref{lemma_hausdorff_near_implies_contents_near}, \cref{theorem_main_continuous_sumsets_v2} (with $\eps$ as $\eps$), and \eqref{eqn_product_set_hausdorff_distance_2} in the same way, we see that
\begin{gather*}
    \frac{\HC_{\geq 1}^{\gamma} \big( \big\lfloor \lambda A' + \eta B' \big\rfloor \big)}{N^\gamma}
    \asymp
    \frac{\HC_{\geq 1}^{\gamma} \big( \lambda A' + \eta B' \big)}{N^\gamma} 
    =
    \HC_{\geq N^{-1}}^{\gamma} \left( \frac{\lambda A' + \eta B'}{N} \right)
    \\\asymp_I 
    \HC_{\geq N^{-1}}^{\gamma} \big( \lambda' X' + \eta' Y' \big)
    \gg_{I,\gamma,\eps}
    \HC_{\geq N^{-1}}^{\gamma + \ogamma + \eps} \big( X' \times Y' \big)
    \asymp
    \frac{\HC_{\geq 1}^{\gamma + \ogamma + \eps} \big( A' \times B' \big)}{ N^{\gamma+\ogamma + \eps}}.
\end{gather*}
This is precisely the inequality in \eqref{eqn_equivalence_D_E_1}, completing the proof.
\end{proof}

In the following corollary, note that it is a consequence of \cref{cor_dim_of_products_of_invariant_sets} that all four discrete notions of dimension, $\ldimdm, \udimdm, \ldimdh, \udimdh$, coincide for multiplicatively invariant sets $A$ and $B$ and their Cartesian product $A \times B$.  In particular,
\[\dim (A \times B) = \dim A + \dim B\]
for any $\dim \in \{\ldimdm, \udimdm, \ldimdh, \udimdh\}$.

\begin{corollary}
\label{cor_dim_of_sumsets_of_artbirary_subsets_v2}
Let $r$ and $s$ be multiplicatively independent positive integers, and let $A, B \subseteq \Nz$ be $\times r$- and $\times s$-invariant sets, respectively. For all $\dim \in \{\ldimdm, \udimdm, \ldimdh, \udimdh\}$ and $\lambda, \eta \in (0,\infty)$,
\begin{align}
\label{eqn_correct_dim_for_sumsets_of_invariant_sets}
    \dim \big( \lfloor \lambda A + \eta B \rfloor \big) = \min \big(1, \dim (A \times B) \big).
\end{align}
Moreover, for all $A' \subseteq A$ and $B' \subseteq B$,
\begin{itemize}
    \item if $\dim A + \dim B \leq 1$, then
\begin{align}
\label{cor_ldimdh_for_arbitrary_subsets_discrete_sumsets_case_1}
    \dim \big( \lfloor \lambda A' + \eta B' \rfloor \big) = \dim \big( A' \times B' \big);
\end{align}
    \item if $\dim A + \dim B > 1$, then
\begin{align}
\label{cor_ldimdh_for_arbitrary_subsets_discrete_sumsets_case_2}
    \dim \big( \lfloor \lambda A' + \eta B' \rfloor \big) \geq \dim \big( A' \times B' \big) - \dim \big(A \times B \big) + 1.
\end{align}
\end{itemize}
\end{corollary}

\begin{proof}
First, note that \eqref{eqn_correct_dim_for_sumsets_of_invariant_sets} is a consequence of \eqref{cor_ldimdh_for_arbitrary_subsets_discrete_sumsets_case_1} and \eqref{cor_ldimdh_for_arbitrary_subsets_discrete_sumsets_case_2}.  
Indeed, setting $A' = A$ and $B' = B$, if $\dim A + \dim B \leq 1$ then \eqref{eqn_correct_dim_for_sumsets_of_invariant_sets} becomes \eqref{cor_ldimdh_for_arbitrary_subsets_discrete_sumsets_case_1}, and if $\dim A + \dim B > 1$ then \eqref{cor_ldimdh_for_arbitrary_subsets_discrete_sumsets_case_2} implies that $\dim \big( \lfloor \lambda A + \eta B \rfloor \big) \geq 1$.  
Since any subset of $\Nz$ has dimension at most $1$, \eqref{eqn_correct_dim_for_sumsets_of_invariant_sets} follows in this case as well.

Define $\ogamma = \max(0, \dimh (A \times B) - 1)$, and let $A' \subseteq A$ and $B' \subseteq B$. To show \eqref{cor_ldimdh_for_arbitrary_subsets_discrete_sumsets_case_1} and \eqref{cor_ldimdh_for_arbitrary_subsets_discrete_sumsets_case_2},
it suffices to show
\begin{align}
\label{eqn_sufficient_to_show_in_discrete_corollary}
    \dim \big( \lfloor \lambda A' + \eta B' \rfloor \big) \geq \dim \big( A' \times B' \big) - \ogamma.
\end{align}
Indeed, this is the lower bound in \eqref{cor_ldimdh_for_arbitrary_subsets_discrete_sumsets_case_2}, and the upper bound guaranteed by \cref{lemma_upper_bounds_on_dims_of_sumsets} combined with this lower bound gives the desired equality in \eqref{cor_ldimdh_for_arbitrary_subsets_discrete_sumsets_case_1}.

Let $\dim \in \{\ldimdm, \udimdm, \ldimdh, \udimdh\}$ and $\lambda, \eta \in (0,\infty)$.  If $\dim(A' \times B') = 0$, the conclusion is immediate, so we can proceed under the assumption that $\dim(A' \times B') > 0$.

There exists $M \in \N$ such that for all $N \in \N$,
\begin{align*}
    \big\lfloor \lambda A' + \eta B' \big\rfloor \cap [0,N) \supseteq \Big\lfloor\lambda \big( A' \cap [0,N/M) \big) + \eta \big( B' \cap [0,N/M) \big)\Big\rfloor.
\end{align*}
Let $\eps > 0$, and let $\gamma = \dim(A' \times B') - \ogamma - 2\eps$.  
Let $N$ be large enough that \cref{theorem_main_discrete_sums_v2} holds with $N/M$ in place of $N$, and define $A'' = A' \cap [0,N/M)$ and $B'' = B' \cap [0,N/M)$.  
It follows from \cref{theorem_main_discrete_sums_v2} that
\begin{align*}
    \frac{\big| \big( \lambda A' + \eta B' \big) \cap [0,N) \big|}{N^\gamma} 
    &\geq
    \frac{|A'' \times B''|}{N^\gamma(N/M)^{\ogamma + \eps}} 
    =
    M^{\ogamma + \eps}\frac{|( A' \times B') \cap [0,N/M)^2|}{N^{\gamma + \ogamma + \eps}},
    \\
    \frac{\HC_{\geq 1}^{\gamma} \big( \big\lfloor \lambda A' + \eta B' \big\rfloor\cap[0,N) \big)}{N^{\gamma}} 
    & \gg_{\lambda,\eta,\gamma,\eps} 
    \frac{\HC_{\geq 1}^{\gamma + \ogamma + \eps}(A'' \times B'')}{N^{\gamma} (N/M)^{\ogamma + \eps}}
    \\
    &=
    M^{\ogamma + \eps} \frac{\HC_{\geq 1}^{\gamma + \ogamma + \eps}(( A' \times B') \cap [0,N/M)^2)}{N^{\gamma + \ogamma + \eps}}.
\end{align*}
Consider the first inequality if $\dim$ is the discrete Minkowski dimension and the second inequality if $\dim$ is the discrete Hausdorff dimension.  Because $\gamma + \ogamma + \eps = \dim(A' \times B') - \eps$, the limit infimum (if $\dim$ is a lower dimension) or limit supremum (if $\dim$ is an upper dimension) as $N$ tends to infinity of the right hand side is positive. It follows that
\[\dim \big( \lfloor \lambda A' + \eta B' \rfloor \big) \geq \gamma.\]
The inequality in \eqref{eqn_sufficient_to_show_in_discrete_corollary} now follow from the fact that $\gamma=\dim (A' \times B') - \ogamma - 2\eps$ and $\eps  >0$ was arbitrary, concluding the proof.
\end{proof}

\subsection{An example that shows \texorpdfstring{$\Phi$}{R}-invariance does not suffice}
\label{section_counterexample}

Fix $2 \leq r < s$. In this section, we construct two sets $A, B \subseteq \Nz$ which satisfy the following properties:
\begin{enumerate}[label=(\Roman*)]
    \item\label{example_mass_dim_exists} the mass dimensions of $A$ and $B$ exist and $\dimdm A = \dimdm B = 1/2$;
    \item\label{example_mult_invariance} $rA \subseteq A$ and $sB \subseteq B$;
    \item\label{example_phi_invariance} $\Phi_r(A) = A$ and $\Phi_s(B)  = B$; and
    \item\label{example_sumset_small_dim} $\udimdm(A+B) \leq 4/5$.
\end{enumerate}

This example demonstrates that neither $\Phi$-invariance nor the invariance indicated in \ref{example_mult_invariance} suffice to obtain the result in \cref{cor_dim_of_sumsets_of_artbirary_subsets_v2}. This is in contrast to \cref{prop_discrete_furstenberg}, where the conclusion holds under the weaker assumption that the sets $A$ and $B$ are $\Phi_r$- and $\Phi_s$-invariant, respectively.  We do not know whether $\Psi$-invariance alone suffices in either \cref{prop_discrete_furstenberg} or \cref{cor_dim_of_sumsets_of_artbirary_subsets_v2}, but invariance under multiplication by $r$ and $s$ (in the sense of \ref{example_mult_invariance}) does not suffice to reach the conclusions in either theorem: the set of squares is invariant under multiplication by both 4 and 9 simultaneously, but has dimension equal to $1/2$, while the sets $A$ and $B$ above demonstrate that \cref{cor_dim_of_sumsets_of_artbirary_subsets_v2} does not hold under the assumption of invariance under multiplication.

In what follows, the interval notation $[a,b]$ is understood to mean $[a,b] \cap \Nz$. For $i, j \in \Nz$, let
\begin{align*}
    I_i = [r^i, r^i + r^{(i+1)/2}], \qquad J_j = [s^j, s^j + s^{(j+1)/2}],
\end{align*}
and then define
\begin{align*}
    A = \{0\} \cup \bigcup_{i, \ell \geq 0} r^\ell I_i, \qquad B = \{0\} \cup \bigcup_{j,m \geq 0} s^m J_j.
\end{align*}

First we will verify \ref{example_mass_dim_exists} by showing that the mass dimension of $A$ exists and is equal to $1/2$; the argument for $B$ is the same. It is easy to see that for all $N \geq 1$,
\[I_{N-1} \subseteq A \cap [1,r^N) \subseteq \bigcup_{\substack{i, \ell \geq 0 \\ i + \ell \leq N}} r^\ell I_i,\]
from which it follows that
\[r^{N/2} \leq \big|A \cap [0,r^N) \big| \leq (N+1)^2 (r^{(N+1) / 2}+1).\]
This shows that $\ldimdm A = \udimdm A = \dimdm A = 1/2$.

It is clear from the definition of the sets $A$ and $B$ that \ref{example_mult_invariance} holds.

Next we will verify \ref{example_phi_invariance} by showing that $\Phi_r(A) = A$; the same argument works to show that $\Phi_s(B) = B$. Since $r A \subseteq A$, we have that
\[A = \Phi_r(rA) \subseteq \Phi_r(A) = \{0\} \cup \bigcup_{i, \ell \geq 0} \Phi_r(r^\ell I_i).\]
Since $0 \in A$, we need only to verify that for all $i, \ell \geq 0$, $\Phi_r(r^\ell I_i) \subseteq A$. If $\ell \geq 1$, then $\Phi_r(r^\ell I_i) = r^{\ell-1} I_i \subseteq A$.  If $\ell = 0$ and $i = 0$, then we see $\Phi_r(I_0) = \{0\} \subseteq A$. If $\ell = 0$ and $i \geq 1$, then we see $\Phi_r(I_i) = [r^{i-1}, r^{i-1} + r^{(i-1)/2}] \subseteq I_{i-1} \subseteq A$.  Thus, $\Phi_r(A) = A$.

Finally we will verify \ref{example_sumset_small_dim} by showing that for all $N$ sufficiently large,
\begin{align}\label{example_inequality_to_show_for_sumset}
    \big|(A+B) \cap [0,r^N) \big| \leq 4 N^4 r^{4N / 5}.
\end{align}
Let $\sigma = \log s / \log r$. Because
\[B \cap [1,r^N) \subseteq \bigcup_{\substack{i, \ell \geq 0 \\ \sigma(j + m) \leq N}} s^m J_j,\]
we have that
\begin{align}\label{example_sum_to_split}
    \big|(A+B) \cap [0,r^N) \big| \leq 1+ \sum_{i,j,\ell,m} |r^\ell I_i + s^m J_j|,
\end{align}
where the sum is over all $i, j, \ell, m \geq 0$ for which $i + \ell \leq N$ and $\sigma(j+m) \leq N$.  We will estimate this sum from above by splitting the sum indices into two sets depending on the ``type'' of the pair $(i,j)$, which we now define.

A pair $(i,j)$ is of Type I if
\begin{align*}
    \frac{i+1}2 + \sigma \frac{j+1}2 \leq \frac {4N}5.
\end{align*}
Using the trivial bound $|C+D| \leq |C||D|$ for finite sets $C, D \subseteq \Nz$, we see that if $i, j, \ell$, and $m$ are such that $(i,j)$ is of Type I, then
\begin{align}\label{example_bound_for_type_I}
    |r^\ell I_i + s^m J_j| \leq |I_i||J_j| = r^{(i+1)/2} s^{(j+1)/2} \leq r^{4N/5}.
\end{align}

A pair $(i,j)$ is of Type II if it is not of Type I, that is, if
\begin{align}\label{example_type_two_inequality}
    \frac{i+1}2 + \sigma \frac{j+1}2 > \frac {4N}5.
\end{align}
Using the fact that $\sigma j \leq N$ and that $N$ is sufficiently large, we see from \eqref{example_type_two_inequality} that $(i-1)/2 > N/4$. It follows then from the fact that $i + \ell \leq N$ that
\begin{align}\label{example_key_inequality_one}
\ell + \frac{i+1}2 < \frac{4N}5.
\end{align}
Similarly, using that $i \leq N$ and the fact that $N$ is sufficiently large, we see from \eqref{example_type_two_inequality} that $\sigma (j-1) / 2 > N/ 4$.  It follows from the fact that $\sigma (j + m) \leq N$ that 
\begin{align}\label{example_key_inequality_two}
    \sigma \left( m + \frac{j+1}2 \right) < \frac{4N}5.
\end{align}
Now we are in a position to use the following fact: if $C, D \subseteq \Nz$ are contained in intervals of length $L$, $M$, respectively, then $C+D$ is contained in an interval of length $L+M$ and hence $|C+D| \leq L+M+1$. If $i, j, \ell$, and $m$ are such that $(i,j)$ is of Type II, then
\[|r^\ell I_i + s^m J_j| \leq r^{\ell + (i+1)/2} + s^{m + (j+1)/2} + 1.\]
Using \eqref{example_key_inequality_one} and \eqref{example_key_inequality_two}, we have that
\begin{align}\label{example_bound_for_type_II}
    |r^\ell I_i + s^m J_j| \leq 3 r^{4N/5}.
\end{align}

Finally, by splitting up the sum in \eqref{example_sum_to_split} into tuples for which the pairs $(i,j)$ are of Type I or Type II, we see by combining \eqref{example_bound_for_type_I} and \eqref{example_bound_for_type_II} that
the desired inequality in \eqref{example_inequality_to_show_for_sumset} holds.

\subsection{Iterated sums of a multiplicatively invariant set}
\label{sec_proof_of_thm_C}

In this section, we will prove \cref{theorem_lmp_analogue}. The strategy is to use tools from \cref{sec_connection_01-interval} to derive \cref{theorem_lmp_analogue} from the theorem of Lindenstrauss-Meiri-Peres, \cref{LMP_theorem}. Throughout this section, $r \geq 2$ is fixed and all of the asymptotic notation may implicitly depend on it.

\begin{remark}\label{lmp_remark}
There are some useful remarks to make before the proof.  Let $X_1$, $X_2$, \dots, $X_n \subseteq [0,1]$ be $\times r$-invariant sets.  The sumset $X_1 + \cdots + X_n$ may be interpreted in $\R/\Z$ or in $\R$.  Denote temporarily by $W_n$ the set $X_1 + \cdots + X_n$ interpreted modulo 1 as a subset of $[0,1]$ and by $Y_n$ the set $X_1 + \cdots + X_n$ interpreted in $\R$ as a subset of $[0,n]$. Two facts of particular relevance to us are: 1) set $W_n$ is $\times r$-invariant, and 2) $\dimh W_n = \dimh Y_n$.  The first fact follows easily from the fact that multiplication by $r$ is a group endomorphism of $(\R/\Z,+)$.  (In contrast, note that the sumset of $\times r$-invariant subsets of $\Nz$ is not necessarily $\times r$-invariant: if $A$ is the base-$10$ restricted digit Cantor set with allowed digits $0$ and $5$, then $A+A$ contains $10$ but does not contain $\Phi_{10}(10) = 1$, for example). The second fact follows immediately by writing $W_n = \cup_{i=0}^{n-1} \big((Y_n \cap [i,i+1]) - i\big)$ and using the translation-invariance and finite (countable) stability under unions of the Hausdorff dimension.
\end{remark}

\begin{proof}[Proof of \cref{theorem_lmp_analogue}]
Recall that $(A_i)_{i =1}^\infty$ is a sequence of $\times r$-invariant subsets of $\Nz$. For each $i \in \N$, let $A_i'$ be the set described in \cref{cor_correspondence_integers_subshifts}, and define $X_i \subseteq [0,1]$ to be the Hausdorff limit of the sequence $(A_i' \cap [0,r^N) / r^N)_{N =1}^\infty$ as in \cref{prop_correspondence_integers_01-interval}.  Since $\dimh X_i = \dimdh A_i' = \dimdh A_i$ and $\sum_{i=1}^\infty \dimdh A_i / \allowbreak | \log \dimdh A_i |$ diverges, we have that $\sum_{i=1}^\infty \dimdh X_i / \allowbreak | \log \dimdh X_i |$ diverges.  It follows by \cref{LMP_theorem} that
\begin{align}\label{eqn_lmp_result_cts}
    \lim_{n \to \infty} \dimh \big( X_1+ \cdots + X_n \big) = 1.
\end{align}
According to \cref{lmp_remark}, we can and will interpret the sum $X_1+ \cdots + X_n$ to be in $\R$.

We claim now that for all $n \in \N$, the discrete Hausdorff dimension of the set $A_1' + \cdots + A_n'$ exists and
\begin{align}\label{more_than_sufficient_in_LMP}
    \dimdh \big( A_1' + \cdots + A_n' \big) = \dimh \big( X_1+ \cdots + X_n \big).
\end{align}
Combined with \eqref{eqn_lmp_result_cts}, this suffices to conclude the proof of \cref{theorem_lmp_analogue} since $A_i' \subseteq A_i$ implies that $\dimdh \big( A_1' + \cdots + A_n' \big) \leq \ldimdh \big( A_1 + \cdots + A_n \big)$.

To show \eqref{more_than_sufficient_in_LMP}, let $n \in \N$, and define $k = \lfloor \log n / \log r \rfloor + 1$.  Define $B_n = A_1' + \cdots + A_n'$ and $Y_n = X_1 + \cdots + X_n$, where the sum defining $Y_n$ is understood to be in $\R$. Note that for all $N \geq k$,
\begin{align}\label{lmp_proof_eqn_five}
    \sum_{i=1}^n \frac{A_i'\cap [0,r^{N-k})}{r^N} \subseteq \frac{B_n \cap [0,r^N)}{r^N} \subseteq \sum_{i=1}^n \frac{A_i'\cap [0,r^N)}{r^N},
\end{align}
where the sums indicate sumsets.
The goal now is to compare the discrete Hausdorff contents of each of these sets at scale $r^{-N}$.

By the definition of the set $X_i$, it follows from \cref{lem_haus_containment} that \begin{align}\label{lmp_proof_eqn_one}
    d_H\left( \frac{A_i'\cap [0,r^N)}{r^N}, X_i \right) \ll r^{-N},
\end{align}
which implies by \cref{lemma_hausdorff_near_implies_contents_near} that for all $\gamma \in [0,1]$,
\begin{align}\label{lmp_proof_eqn_two}
    \HC_{\geq r^{-N}}^\gamma \left( \sum_{i=1}^n \frac{A_i'\cap [0,r^N)}{r^N} \right) \asymp_n \HC_{\geq r^{-N}}^\gamma \big( Y_n \big).
\end{align}
It also follows from \eqref{lmp_proof_eqn_one} that
\begin{align*}
    d_H\left( \frac{A_i'\cap [0,r^{N-k})}{r^N}, \frac{X_i}{r^{k}} \right) \ll_n r^{-N},
\end{align*}
which implies by \cref{lemma_hausdorff_near_implies_contents_near} that
\begin{align}\label{lmp_proof_eqn_four}
    \HC_{\geq r^{-N}}^\gamma \left( \sum_{i=1}^n \frac{A_i'\cap [0,r^{N-k})}{r^N} \right) \asymp_n \HC_{\geq r^{-N}}^\gamma \left( \frac{Y_n}{r^k} \right).
\end{align}

Combining \eqref{lmp_proof_eqn_five} with \eqref{lmp_proof_eqn_two} and \eqref{lmp_proof_eqn_four}, we see that
\[\HC_{\geq r^{-N}}^\gamma \left( \frac{Y_n}{r^k} \right) \ll_n \HC_{\geq r^{-N}}^\gamma \left( \frac{B_n \cap [0,r^N)}{r^N} \right) = \frac{\HC_{\geq 1}^\gamma \left( B_n \cap [0,r^N) \right)}{r^{N \gamma}} \ll_n \HC_{\geq r^{-N}}^\gamma \big( Y_n \big).\]
Letting $N$ tend to infinity and noting that $n$, and hence $k$, are fixed, these inequalities combine with \cref{lemma_connection_between_discrete_and_unlimited_H_contents}, \cref{lem_discrete_dim_properties} \ref{itm_dim_IV}, \eqref{eqn_lmp_result_cts}, and the fact that $\dimh (Y_n / r^k) = \dimh Y_n$ to prove the equality in \eqref{more_than_sufficient_in_LMP}.
\end{proof}

\section{Open directions}
\label{sec_future}

We collect in this section a number of interesting open questions concerning multiplicatively invariant subsets of the non-negative integers.
Though these questions and conjectures are stated for arbitrary $\times r$-invariant subsets of $\Nz$, many are already open and interesting for the special case of base-$r$ restricted digit Cantor sets.

\subsection{Positive density for sumsets of full dimension}
\label{sec_pos_den_for_sumsets}

In \cite[Problem 4.10]{hochman_dimension_theory_2018},
Hochman asks whether the sumset $X+Y$ of a $\times r$- and a $\times s$-invariant subset of $[0,1]    $ satisfying $\dimh X+\dimh Y>1$ has positive Lebesgue measure.
We remark that a projection theorem of Marstrand \cite[Theorem I]{Marstrand_1954} implies that $\lambda X + \eta Y$ has positive Lebesgue measure for a.e.\ $(\lambda, \eta)\in\R^2$, suggesting a possible affirmative answer. 
In \cite[Theorem 1.4]{glasscock_marstrand_for_integers}, a version of Marstrand's projection theorem for subsets of the integers was obtained, with Lebesgue measure replaced by the notion of upper natural density.\footnote{Given a set $E\subseteq\Z$, its \define{upper natural density} is defined by $\bar d(E) \defeq \limsup_{N\to\infty}|E\cap\{-N,\dots,N\}|/(2N+1)$.}
It therefore makes sense to consider the following integer analogue of Hochman's question.

\begin{question}\label{question_posden}
Let $r,s\in\N$ be multiplicatively independent, and let $A,B\subseteq\N_0$ be $\times r$- and $\times s$-invariant, respectively.
If $\dimdm A+\dimdm B>1$, then does the sumset $A+B$ have positive upper natural density? 
\end{question}

\subsection{Small intersections}
\label{sec_small_intersections}

While \cref{question_posden} considers the sum $A+B$ when sum of the dimensions is larger than $1$, it is also natural to ask about the intersection $A\cap B$ when the sum of the dimensions is below $1$.
A special case of a conjecture posed by Furstenberg in \cite{furstenbergtransversality} asserts that if $r,s\in\N$ are multiplicatively independent and $X,Y\subseteq[0,1]$ are $\times r$- and $\times s$-invariant, respectively, then 
\begin{align*}
\dim X+\dim Y<1 \implies X\cap Y\subseteq\Q.
\end{align*}
Furstenberg showed that an affirmative answer to this question implies that any large enough power of $2$ contains every digit (in base $10$), which is a variant of the conjecture of Erd\H os \cite{erdosmathmag} mentioned in the introduction.

The following question is inspired by Furstenberg's conjecture.
\begin{question}
\label{question_finiteintersection}
Let $r,s\in\N$ be multiplicatively independent, and let $A,B\subseteq\Nz$ be $\times r$- and $\times s$-invariant, respectively.
Is it true that
\begin{align*}
\dim A+\dim B<1 \implies A\cap B~\text{is finite}?
\end{align*}
\end{question}
A special case of this question is formulated in \cite[Conjecture 6.2]{yu_additive_properties_restricted_digits_arxiv}.
If the answer to \cref{question_finiteintersection} is positive, then Erd\H os' conjecture holds (this can be seen by taking $r=2$, $s=3$, $A$ to be the powers of $2$, and $B$ to be a restricted digit Cantor set).
A weaker version of this statement was established by Lagarias \cite{lagariasternarypowersoftwo}.

One can formulate a natural quantitative strengthening of \cref{question_finiteintersection} as follows.
Given $n,r,k\in\N$, let $d_{r,k}(n)$ be the number of subwords of $(n)_r$ of length at most $k$. Then the answer to \cref{question_finiteintersection} is positive if one can show that
\begin{align}
\label{eqn_EQN}
\limsup_{k \to \infty}\ \liminf_{n\to\infty}\left(\frac{\log d_{r,k}(n)}{k\log r}+
\frac{\log d_{s,k}(n)}{k\log s}\right)
= 1.
\end{align}
In fact, it suffices to prove that the expression in \eqref{eqn_EQN} is greater than or equal to 1. Indeed, by considering $n$ to be a power of $r$, for any $k\in\N$, $\liminf_{n\to\infty} \log d_{r,k}(n) / k\log r = \log 2k / k\log r$, whereby the expression in \eqref{eqn_EQN} is at most 1.  
We believe the limit in \eqref{eqn_EQN} as $k$ tends to infinity exists, but this would not be necessary to imply a positive answer to \cref{question_finiteintersection}.

\subsection{Difference sets}
For closed subsets $X,Y\subseteq[0,1]$, working with the difference set $X-Y$ is no harder than working with the sumset $X+Y$.
In particular, proving that
\[
\dimm\big( X - Y \big) =\min{}\big(\dimm X + \dimm Y, \ 1 \big)
\]
in \cref{eqn_in_thm_HS_localentropy} requires no additional work.
The story changes in the setting of the non-negative integers, where difference sets are much more cumbersome to handle, ultimately because the fibers of the map $(a,b) \mapsto a-b$ are not compact.
This observation explains why our main results in the integer setting only deal with sumsets $\lambda A + \eta B$ with $\lambda$ and $\eta$ both positive, and it naturally leads us to the following question.

\begin{question}\label{question_differences}
Let $r$ and $s$ be multiplicatively independent positive integers, and let $A, B \subseteq \Nz$ be $\times r$- and $\times s$-invariant, respectively. 
Is it true that
\[\dimdm ( A - B  )  \,=\, \min \big(\dimdm A + \dimdm B, \ 1 \big)?\]
\end{question}

The methods used in \cref{sec_sumsintegers} allow us to establish the lower bound
$\ldimdm ( A - B  )  \geq \min (\dimdm A + \dimdm B, \ 1 ).$
However, the upper bound $\udimdm ( A - B  )  \leq \min (\dimdm A + \dimdm B, \ 1 )$, which is straightforward for sums, remains open for differences.

There are many natural variants and extensions of \cref{question_differences}: one can replace $A-B$ with a more general expression $\lfloor\lambda A+\eta B\rfloor$ for any non-zero real numbers $\lambda,\eta$, or one can replace $\dimdm$ with $\dimdh$. One can ask about combinations of the form $\lfloor\lambda A'+\eta B'\rfloor$ for arbitrary subsets $A'$ and $B'$ of $A$ and $B$, or one can look only at the positive portion $(A-B)\cap\N$ of the difference set.
Our methods provide an outline for obtaining lower bounds, but upper bounds seem to require a new strategy.

\subsection{Analogous results for other notions of discrete dimension}
The \emph{upper Banach dimension} (or \emph{upper counting dimension}, cf.\ \cite{limamoreira} and \cite{glasscock_marstrand_for_integers}) of a set $A \subseteq \Nz$ is 
\[\dim^*A \defeq \limsup_{N-M\to\infty}\frac{\log\big|A\cap[M,N]\big|}{\log(N-M)}.\]
In general, we only have the inequality $\dim^*A\geq\udimdm A$, but if $A\subseteq\Nz$ is $\times r$-invariant, then it can be shown that $\dimdm A=\dimdh A = \dim^* A$.

\begin{question}
Let $r$ and $s$ be multiplicatively independent positive integers, and let $A, B \subseteq \Nz$ be $\times r$- and $\times s$-invariant, respectively. 
Is it true that
\begin{align*}
    \dim^* ( A + B  )  \,&=\, \min \big(\dim^* A + \dim^* B, \ 1 \big), \text{ and / or}\\
    \dim^* ( A \cap B  )  \,&\leq\, \max \big(\dim^* A + \dim^* B - 1, \ 0 \big).
\end{align*}
\end{question}
Note that the lower bound $\dim^* ( A + B  ) \geq \min \big(\dim^* A + \dim^* B, \ 1 \big)$ follows from \cref{mainthm_integer_sumsets} using the fact that $\dim^* \geq\udimdm$.

There are several other ways to define natural notions of dimensions for subsets of $\Nz$.  
Barlow and Taylor \cite{barlow_taylor_92} define, for example, a discrete notion of packing dimension.  
The main results in this article suggest possible analogues for their discrete packing dimension.

\subsection{Polynomial functions of multiplicatively invariant sets}
\label{subsec_polynomial}

The dimension of the sumset of affine images of multiplicatively invariant sets $A$ and $B$ is described in \cref{mainthm_integer_sumsets}.
It is natural to ask about the extent to which the results in that theorem might hold for the sumset of images of $A$ and $B$ under other functions.

In this subsection, for $n \in \N$, denote by $A^{(n)}$ the set of $n^{\text{th}}$-powers of elements of $A$: $A^{(n)} \defeq \{a^n \ | \ a\in A\}$.
The following question is a (special case of a) natural polynomial extension of \cref{mainthm_integer_sumsets}. 

\begin{question}
\label{question_polynomials}
Let $n,m\in\N$, let $r,s\in\N$ be multiplicatively independent, and let $A,B\subseteq\Nz$ be $\times r$- and $\times s$-invariant, respectively.
Is it true that
\begin{align}
\label{eq_openproblemeq1}
\dimdm \big( A^{(n)} + B^{(m)} \big) &= \min\left(\frac1n\dimdm A+\frac1m\dimdm B\,,\,1\right)?
\end{align}
\end{question}

When $A=B=\Nz$, an affirmative answer to \cref{question_polynomials} follows from basic facts in number theory.  It is easy to see that for any $A\subseteq\Nz$ for which the Minkowski dimension exists, the set $A^{(n)}$ has dimension $\dimdm A^{(n)}=\dimdm A / n$ (however it is not true in general that $A^{(n)}$ is $\times r$-invariant when $A$ is). Thus, for arbitrary sets $A$ and $B$ that satisfy a natural dimension condition (see footnote \ref{dim_condition_footnote}), it follows from the discrete version of Marstrand's projection theorem in \eqref{marstrand_result_intro_integers} that $\dimdm \big( \lfloor \lambda A^{(n)} + \eta B^{(m)} \rfloor \big)$ is equal to the right hand side of \eqref{eq_openproblemeq1} for Lebesgue almost every $\lambda, \eta > 0$.

We cannot rule out the possibility that \eqref{eq_openproblemeq1} holds when $n, m \geq 2$ for arbitrary sets $A$ and $B$ for which the Minkowski dimensions exist.  When $A = B$ and $n = m = 2$, equality in \eqref{eq_openproblemeq1} is an infinitary version of a conjecture attributed to Ruzsa; see \cite[Conjecture 5]{Cilleruelo_Granville}.

\subsection{Multiplicatively invariant sets in relation to other arithmetic sets in the integers}

In this paper, we are concerned with transversality between $\times r$- and $\times s$-invariant sets whenever $r$ and $s$ are multiplicatively independent. 
In principle, it makes sense to inquire about transversality (or \emph{independence}) between any two sets which are structured in different ways.
To keep the discussion short, we restrict to infinite arithmetic progressions (or congruence classes), the set of perfect squares, and the set of primes.
\begin{question}
Let $A\subseteq\Nz$ be a $\times r$-invariant set, and let $P$ be an infinite arithmetic progression.
Is it true that $\dimdm(A\cap P)$ is either $0$ or $\dimdm(A)$?
\end{question}

The answer is yes for restricted digit Cantor sets. In fact, it is proved in \cite{erdos_mauduit_sarkozy_restricted_digit_cantor_sets} that such sets satisfy ``good equidistribution properties'' in residue classes.

More generally, one could ask about the sum or the intersection of a $\times r$-invariant set and the image of an arbitrary polynomial with integer coefficients, for instance the set of perfect squares, $S=\{n^2 \ | \ n\in\Nz\}$.  Note that $\dimdm S = 1/2$.

\begin{question}
Let $A\subseteq\Nz$ be a $\times r$-invariant set.
Is is true that
\[\dimdm \big( A + S  \big)  \,=\, \min \big(\dimdm A + 1/2, \ 1 \big)\]
and/or
\[\udimdm\big(A\cap S\big)\,\leq\, \max\big(\dimdm A-1/2, \ 0\big)?\]
\end{question}
Note that the first expression in this question is a special case of the equality in \cref{question_polynomials}.  In a similar vein, one can ask about intersections with the set of prime numbers, $\P$. Note that $\dimdm \P = 1$.

\begin{question}\label{question_primes}
Let $A\subseteq\Nz$ be a $\times r$-invariant set.
Is is true that $\dimdm(A\cap\P)$ is either $0$ or $\dimdm(A)$?
\end{question}

Maynard showed in \cite{maynard_primes_with_missing_digits} that the answer to \cref{question_primes} is positive when $A$ is a restricted digit Cantor set where the number of restricted digits is small enough with respect to the base. 
In fact, he obtains a Prime Number Theorem in such sets, which is stronger than simply $\dimdm (A \cap \mathbb{P}) = \dimdm A$. 
\cref{question_primes} is open for general restricted digit Cantor sets, and may be very difficult in general.
The methods in this paper do not appear to shed new light on this line of inquiry.

\subsection{Transversality of multiplicatively invariant sets in the \texorpdfstring{$rs$}{rs}-adics}

The $rs$-adics is a non-Archimedean regime in which it is easy to ask questions analogous to those asked in this work.  Furstenberg proved in \cite[Theorem 3]{furstenbergtransversality} an analogue of \cref{thm_furstenberg} in the $rs$-adics.

Following Furstenberg, note that the maps $\Phi_r$ and $\Phi_s$, with domains extended to $\Z$, are uniformly continuous with respect to the $rs$-adic metric on $\Z$, and therefore extend to continuous transformations of the set of $rs$-adic integers, $\Z_{rs}$.  As a compact metric space, there is a natural Hausdorff dimension to measure the size of subsets of $\Z_{rs}$.
Let us call a set $X \subseteq \Z_{rs}$ \define{$\times r$-invariant} if it is closed and $\Phi_r X \subseteq X$.

\begin{question}
Let $r$ and $s$ be multiplicatively independent positive integers, and let $X, Y \subseteq \Z_{rs}$ be $\times r$- and $\times s$-invariant sets, respectively. Is it true that
\begin{align*}
    \dimh \big(X + Y\big) &= \min{} \big( \dimh X + \dimh Y, \ \dimh \Z_{rs} \big), \text{ and / or}\\
    \dimh \big(X \cap Y \big) &\leq \max{} \big( \dimh X + \dimh Y - \dimh \Z_{rs},  0 \big)?
\end{align*}
\end{question}

The upper bound on $\dimh \big(X \cap Y \big)$ in the previous question was conjectured by Furstenberg in \cite[Conjecture 3]{furstenbergtransversality}.  A positive answer to these questions would bring transversality results in the $rs$-adics in line with those in the real and integer settings.

\bibliographystyle{alphanum}
\bibliography{integercantorbib}

\bigskip
\footnotesize
\noindent
Daniel Glasscock\\
\textsc{University of Massachusetts Lowell}\par\nopagebreak
\noindent
\href{mailto:daniel_glasscock@uml.edu}
{\texttt{daniel{\_}glasscock@uml.edu}}

\bigskip
\footnotesize
\noindent
Joel Moreira\\
\textsc{University of Warwick}\par\nopagebreak
\noindent
\href{mailto:joel.moreira@warwick.ac.uk}
{\texttt{joel.moreira@warwick.ac.uk}}

\bigskip
\footnotesize
\noindent
Florian K.\ Richter\\
\textsc{{\'E}cole Polytechnique F{\'e}d{\'e}rale de Lausanne (EPFL)}\par\nopagebreak
\noindent
\href{mailto:f.richter@epfl.ch}
{\texttt{f.richter@epfl.ch}}

\end{document}